\documentclass[a4paper,11pt,reqno]{amsart}
\usepackage{amssymb,amsmath,hyperref}
\title[On a generalized Euler identity and mock theta functions]{New polar-finite forms of generalized Euler identities\\for $A_{1}^{(1)}$-string functions \\and mock theta conjecture-like identities}

\author{Stepan Konenkov}
\address{Department of Mathematics and Computer Science, Saint Petersburg State University, Saint Petersburg,  Russia, 199178}
\email{konenkov.stepan@yandex.ru}

\author{Eric T. Mortenson}
\address{Department of Mathematics and Computer Science, Saint Petersburg State University, Saint Petersburg,  Russia, 199178}
\email{etmortenson@gmail.com}

\renewcommand\theta{\vartheta}

\newtheorem{theorem}{Theorem}
\newtheorem{lemma}[theorem]{Lemma}
\newtheorem{corollary}[theorem]{Corollary}
\newtheorem{proposition}[theorem]{Proposition}
\theoremstyle{definition}

\newtheorem{remark}[theorem]{Remark}

\numberwithin{theorem}{section} 
\numberwithin{equation}{section}

\newcommand{\Z}{\mathbb{Z}}

\newcommand{\im}{\textnormal{Im}}

\makeatletter
\@namedef{subjclassname@2020}{\textup{2020} Mathematics Subject Classification}
\makeatother

\begin{document}

\date{2 February 2026}



\begin{abstract}
Determining the explicit forms and modularity for string functions and branching coefficients for Kac--Moody algebras after Kac, Peterson, and Wakimoto is an important problem.  For positive admissible-level string functions for the affine Kac--Moody algebra $A_{1}^{(1)}$, very little is known.  Here we apply the notion of quasi-periodicity to a generalized Euler identity of Schilling and Warnaar for the affine Kac--Moody algebra $A_{1}^{(1)}$.  For integral-level string functions the classical periodicity reduces the infinite sum of string functions in the generalized Euler identity to a finite sum of string functions with theta function coefficients.  For admissible-level, we similarly reduce to an analogous finite sum of string functions, but we also gain an additional finite sum of the form
\begin{equation*}
\sum_{i}\Phi_{i}(q)\Psi_{i}(q),
\end{equation*}
where the $\Phi_i(q)$'s are modular and depend only on the spin and the $\Psi_{i}(q)$'s are (mixed) mock modular Hecke-type double-sums and depend only on the quantum number.  For levels $1/2$, $1/3$, and $2/3$, we shall also see that the $\Psi_{i}(q)$'s give us families of mock theta conjecture-like identities for symmetric Hecke-type double-sums. Our work here focuses on evaluating the $\Psi_{i}(q)$'s, and our expressions utilize Ramanujan's second-order mock theta function $\mu_2(q)$ and third-order mock theta functions $f_{3}(q)$, $\omega_3(q)$, $\psi_{3}(q)$, and $\chi_3(q)$. 
\end{abstract}

\maketitle

 \tableofcontents

\section{Introduction}

Determining the explicit forms and modularity for string functions and branching coefficients for Kac--Moody algebras after Kac, Peterson, and Wakimoto is an important problem \cite{KP84, KW88advmath, KW88, KW90}.   Kac and Peterson \cite{KP84} derived the modular properties of string functions of the integrable highest weight representations on the whole modular group, and calculated them for certain cases in terms of theta functions. Subsequently, Kac and Wakimoto generalized these results and computed certain branching functions using methods of modular and conformal invariance \cite{KW88advmath}. 

\smallskip
The work of Kac, Peterson, and Wakimoto has been taken in multiple directions by physicists and mathematicians. In one direction, Ahn, Chung, and Tye \cite{ACT} introduced the generalized Fateev--Zamolodchikov parafermionic theories, which have applications in statistical physics and string theory. In another direction, Auger, Creutzig, and Ridout more recently pursued the corresponding logarithmic parafermionic vertex algebra for negative admissible-levels  \cite{ACR}. 

\smallskip
In a recent work, Borozenets and Mortenson \cite{BoMo2026, BoMo2025} and then Konenkov and Mortenson \cite{KM25} determined explicit forms for certain admissible-level string functions for the affine Kac--Moody algebra $A_{1}^{(1)}$.  For positive fractional-level string functions Borozenets and Mortenson obtained mock theta conjecture-like identities, and for negative fractional-level string functions, they obtained mixed false theta function expressions, see also \cite{AM09, ACR, BKMZ, BM, BKMN}.   Techniques included the use Hecke-type double-sums expressions \cite{HM, MZ2023}, mock modularity \cite{Zw2}, quasi-periodicity \cite{BoMo2026}, and polar-finite decompositions after Zwegers \cite{Zw2} and Dabholkar, Murthy, and Zagier \cite{DMZ}.  One is thus motivated to find applications of quasi-periodicity and polar-finite decompositions in other settings. 

\smallskip
 Dabholkar, Murthy, and Zagier \cite{DMZ} built upon work of Zwegers \cite{Zw2} and introduced a canonical decomposition of a meromorphic Jacobi form into a ``finite-part'' and a ``polar-part.''  The finite-part is a finite linear combination of theta functions with mock modular forms as coefficients, and the polar-part is entirely determined by the poles of the meromorphic Jacobi form.  Recently, Borozenets and Mortenson \cite{BoMo2025} introduced quasi-periodic relations to extend Zagier--Zwegers' analysis to the case of admissible characters, which are vector-valued meromorphic Jacobi forms \cite{KW88advmath}. Using the polar-finite decomposition of admissible characters they discovered mock theta conjecture-like identities for string functions of certain positive admissible-levels, thus extending their results in \cite{BoMo2026}. 

\smallskip
In this paper, we will apply quasi-periodicity to a generalized Euler identity of Schilling and Warnaar \cite{SW} after work of Kac and Wakimoto on branching functions \cite{KW90}.  We will thus obtain new polar-finite forms of generalized Euler identities for $A_{1}^{(1)}$-string functions, and we will extract new mock theta conjecture-like identities. 

\smallskip
For integral-level string functions it is classical that periodicity reduces the infinite sum of string functions in the generalized Euler identity to a finite sum.  For admissible-level, we similarly reduce to a finite sum of string functions, but we also gain an additional finite sum of the form
\begin{equation*}
\sum_{i}\Phi_{i}(q)\Psi_{i}(q),
\end{equation*}
where the $\Phi_i(q)$'s are modular and depend only on the spin and the $\Psi_{i}(q)$'s are (mixed) mock modular and depend only on the quantum number.  For levels $1/2$, $1/3$, and $2/3$, we shall also see that the $\Psi_{i}(q)$'s give us families of mock theta conjecture-like identities for symmetric Hecke-type double-sums similar to a result in the initial paper of Borozenets and Mortenson \cite[Theorem $2.1$]{BoMo2026}.  Our results are in terms of Ramanujan's second-order and third-order mock theta functions found in his last letter to G. H. Hardy \cite[Chapter 14]{AB2018} and in his lost notebook \cite{RLN}.


 \subsection{On string functions for $A_{1}^{(1)}$}

We will focus on the case of the affine Kac--Moody algebra $A_1^{(1)}$.  For integral-level, string functions and branching coefficients have been well-studied.  For negative admissible-level, a good deal of research has been done, but for positive admissible-level, very little is known.  We will continue the focus on positive admissible-level as in \cite{BoMo2026, BoMo2025, KM25}, but first we review the notation and basic tools.

\smallskip
  We begin by following the terminology in Schilling and Warnaar \cite[Section 3]{SW} and introduce the string function as follows.  We let $p^{\prime}\ge 2$, $p \ge 1$ be coprime integers, and we define the admissible-level to be
\begin{equation} \label{equation:admlevel}
N:=\frac{p^{\prime}}{p}-2.
\end{equation}
Let $q := e^{2\pi i \tau}$ with $\im(\tau) >0$ and $z$ a non-zero complex number. For an admissible character of level $N$ and spin $\ell$, we define the string function for level $N$, quantum number $m$, and spin $\ell$ with $m+\ell \equiv 0 \pmod 2$, through
\begin{equation} \label{equation:fourcoefexp}
\chi_{\ell}^N (z;q)=\sum_{m\in 2\mathbb{Z}+\ell}
C_{m,\ell}^{N}(q) q^{\frac{m^2}{4N}}z^{-\frac{1}{2}m}.
\end{equation}

For the purpose of the Weyl--Kac formula, we will define the theta function as
\begin{equation}\label{equation:SW-thetaDef}
\Theta_{n,m}(z;q):=\sum_{j\in\mathbb{Z}+n/2m}q^{mj^2}z^{-mj};
\end{equation}
however, we will soon use a different definition.   Using the Weyl--Kac formula, one can express the admissible $A_{1}^{(1)}$ character as
\begin{equation}\label{equation:WK-formula}
\chi_{\ell}^N(z;q)=\frac{\sum_{\sigma=\pm 1}\sigma \Theta_{\sigma (\ell+1),p^{\prime}}(z;q^{p})}
{\sum_{\sigma=\pm 1}\sigma\Theta_{\sigma,2}(z;q)}.
\end{equation}

\smallskip

String functions enjoy many symmetries \cite[(3.4), (3.5)]{SW}, \cite[(2.40)]{ACT}
\begin{equation*}
C_{m,\ell}^{N}(q) = C_{-m,\ell}^{N}(q), \ C_{m,\ell}^{N}(q) = C_{N-m,N-\ell}^{N}(q).
\end{equation*}
For the integral-level $N$ we have the important periodicity property \cite[(3.5)]{SW}
\begin{equation} \label{eq:inglevelperiod}
C_{m,\ell}^{N}(q) = C_{m+2N,\ell}^{N}(q),
\end{equation}
and hence from \eqref{equation:fourcoefexp} the theta-expansion
\begin{equation}\label{eq:intlevelthetadecomp}
\chi_{\ell}^N(z;q)=\sum_{\substack{0\le m <2N\\m \in 2\Z + \ell}}C_{m,l}^{N}(q)\Theta_{m,N}(z,q).
\end{equation}
In  \cite{BoMo2025}, Borozenets and Mortenson developed quasi-periodicity \cite[Theorem 2.1]{BoMo2025} and applied it to \eqref{equation:fourcoefexp}.  They obtained a polar-finite expansion for characters of even-spin.  Their result \cite[Theorem 2.3]{BoMo2025} was similar to (\ref{eq:intlevelthetadecomp}) but with an extra term that was a finite sum of binary products of theta functions and Appell functions.   Appell functions are the building blocks of Ramanujan's mock theta functions \cite{AB2018, HM, RLN}.

\smallskip
To state some examples of integral-level string functions, let us first recall the $q$-Pochhammer notation
\begin{equation*}
(x)_n=(x;q)_n:=\prod_{i=0}^{n-1}(1-q^ix), \ \ (x)_{\infty}=(x;q)_{\infty}:=\prod_{i\ge 0}(1-q^ix),
\end{equation*}
and the theta function
\begin{equation}\label{equation:JTPid}
j(x;q):=(x)_{\infty}(q/x)_{\infty}(q)_{\infty}=\sum_{n=-\infty}^{\infty}(-1)^nq^{\binom{n}{2}}x^n,
\end{equation}
where the last equality is the Jacobi triple product identity.   We will frequently use the notation,
\begin{equation*} 
J_{a,b}:=j(q^a;q^b),
 \ \overline{J}_{a,b}:=j(-q^a;q^b), \ {\text{and }}J_a:=J_{a,3a}=\prod_{i\ge 1}(1-q^{ai}),
\end{equation*}
where $a,b$ are positive integers.    One notes that the two definitions for theta functions are equivalent via the identity
\begin{equation*}
\Theta_{n,m}(z;q)=z^{-\frac{n}{2}}q^{\frac{n^2}{4m}}j\big( {-}q^{n+m}z^{-m};q^{2m}\big ).
\end{equation*}

An important example of a theta functions is the Dedekind eta-function.  We define the Dedekind eta-function as follows
\begin{equation*} 
\eta(q) :=q^{1/24}\prod_{n\ge 1}(1-q^{n}).
\end{equation*}
Among more general results, Kac and Peterson \cite[p. 220]{KP84} showed
{\allowdisplaybreaks \begin{gather*}
C_{0,0}^{1}(q) = \eta(q)^{-1},\\
C_{1,1}^{2}(q)= \eta(q)^{-2}\eta(q^2),\\
C_{1,1}^{3}(q)= \eta(q)^{-2} q^{3/40} J_{6,15},\\
C_{2,0}^{4}(q) = \eta(q)^{-2} \eta(q^6)^{-1} \eta(q^{12})^{2}.
\end{gather*}}%

To make our subsequent results easier to state, we introduce the additional notation
\begin{equation}
\mathcal{C}_{m,\ell}^{N}(q) := q^{-s_{\lambda}}C_{m,\ell}^{N}(q) \in \Z[[q]],\label{equation:mathCalCtoStringC}
\end{equation} 
where
\begin{equation*}
s_{\lambda} := -\frac{1}{8}+\frac{(\ell+1)^2}{4(N+2)} - \frac{m^2}{4N}.
\end{equation*}

For our purposes, writing string functions in terms of Hecke-type double-sum notation will be crucial.  Andrews famously related Hecke-type double-sums to Ramanujan's mock theta functions  \cite{And1986}.  We will use the following definition for Hecke-type double-sums:  let $x,y$ be non-zero complex numbers, then
\begin{equation} \label{equation:fabc-def2}
f_{a,b,c}(x,y;q):=\left( \sum_{r,s\ge 0 }-\sum_{r,s<0}\right)(-1)^{r+s}x^ry^sq^{a\binom{r}{2}+brs+c\binom{s}{2}}.
\end{equation}

One way to obtain a Hecke-type double-sum form for admissible-level string functions is by using the classical partial fraction expansion for the reciprocal of Jacobi's theta product
\begin{equation*}
\frac{1}{j(z;q)}=\frac{1}{(q;q)_{\infty}^3}\sum_{n\in\mathbb{Z}}\frac{(-1)^nq^{\binom{n+1}{2}}}{1-q^{n}z}
\end{equation*}
and the Weyl--Kac formula \eqref{equation:WK-formula} as in for example \cite[Section 2.4]{ACT}, \cite[Proposition 3]{L},  \cite[(3.8)]{SW}:  for $p'\geq 2$, $p\geq 1$ coprime integers, $0\leq \ell \leq p'-2$ and $m\in 2\Z +\ell$, then
\begin{equation}
\mathcal{C}_{m,\ell}^{N}(q)
 =\frac{1}{(q)_{\infty}^3}\left ( f_{1,p^{\prime},2pp^{\prime}}(q^{1+\frac{m+\ell}{2}},-q^{p(p^{\prime}+\ell+1)};q)
 -f_{1,p^{\prime},2pp^{\prime}}(q^{\frac{m-\ell}{2}},-q^{p(p^{\prime}-(\ell+1))};q)\right).
 \label{equation:modStringFnHeckeForm}
\end{equation}      
In the case of positive integer level $N>0$, we have the compact form \cite[Example $1.3$]{HM}, \cite{SW}
\begin{equation*}
\mathcal{C}_{m,\ell}^{N}(q)=\frac{1}{(q)_{\infty}^3}
f_{1,1+N,1}(q^{1+\tfrac{1}{2}(m+\ell)},q^{1-\tfrac{1}{2}(m-\ell)};q).
\end{equation*}

Hickerson and Mortenson \cite{HM} and Mortenson and Zwegers \cite{MZ2023} obtained formulas expanding Hecke-type double-sums in terms of Appell functions and theta functions.   We will use the Hecke-type double-sum formula found in \cite[Theorem 1.4]{HM}.  Following Hickerson and Mortenson \cite{HM} we define the Appell function as
\begin{equation}
m(x,z;q):=\frac{1}{j(z;q)}\sum_{r\in\Z}\frac{(-1)^rq^{\binom{r}{2}}z^r}{1-q^{r-1}xz}.\label{equation:m-def}
\end{equation}
Appell functions are building blocks for Ramanujan's mock theta functions \cite[Section 5]{HM}. 
\subsection{On Ramanujan's mock theta functions}
In Ramanujan's last letter to Hardy \cite[Chapter 14]{AB2018}, he gave a list of seventeen so-called mock theta functions.   Each function was defined as a $q$-series convergent for $|q|<1$.  For an example of a mock theta function in terms of Appell functions, Ramanujan's second-order mock theta function $\mu_{2}(q)$ can be written \cite[p. 8]{RLN}, \cite[Section 5]{HM}: 
\begin{equation}
\mu_2(q)
:=\sum_{n\ge 0}\frac{(-1)^nq^{n^2}(q;q^2)_n}{(-q^2;q^2)_{n}^2}
=2m(-q,-1;q^4)+2m(-q,q;q^4),\label{equation:2nd-mu(q)Alt}
\end{equation}
Mock theta functions have certain asymptotic properties similar to those of ordinary theta functions, but they are not theta functions.  In the letter, one finds four third-order mock theta functions, ten fifth-order functions, and three seventh-order functions, as well as several identities relating the mock theta functions to each other.   The notion of order is a mess.

\smallskip
For decades, mock theta functions had been shrouded in mystery and their development had been somewhat sleepy, but they still attracted the attention of luminaries such as Watson \cite{W3}, Selberg \cite{Selb1938}, Dyson \cite{Dyson1944}, and Andrews \cite{And1981, And1986}.   With the advent of the discovery of the lost notebook \cite{RLN}, Hickerson's proof of the mock theta conjectures \cite{H1,H2},  Zwegers breakthrough work on mock modularity \cite{Zw1, Zw2}, and subsequent work of Bringmann and Ono \cite{BrO1, BrO2}, see also Zagier \cite{Zag2007}, mock theta functions became more central to mathematics.

\smallskip
While on a visit to Trinity College Library at Cambridge University, George Andrews discovered a collection of papers of Ramanujan now known as the lost notebook \cite{RLN}.   The collection of papers contained more mock theta functions and mock theta function identities.  In particular, we find the mock theta conjectures, which are ten identities in which each identity expresses a different fifth-order mock theta function in terms of a building-block mock theta function and a simple quotient of theta functions.  This particular building-block is the so-called universal mock theta function $g_3(x;q)$, which is defined
\begin{equation*}
g_3(x;q):=x^{-1}\Big ( -1 +\sum_{n=0}^{\infty}\frac{q^{n^2}}{(x)_{n+1}(q/x)_{n}} \Big ).
\end{equation*}
Using our notation, two of the ten mock theta conjectures read \cite{AG, H1}
{\allowdisplaybreaks \begin{align*}
f_0(q)&:=\sum_{n= 0}^{\infty}\frac{q^{n^2}}{(-q;q)_n}=\frac{J_{5,10}J_{2,5}}{J_1}-2q^2g_3(q^2;q^{10}),\\
f_1(q)&:=\sum_{n= 0}^{\infty}\frac{q^{n(n+1)}}{(-q;q)_n}=\frac{J_{5,10}J_{1,5}}{J_1}-2q^3g_3(q^4;q^{10}).
\end{align*}}%

The universal mock theta function $g_3(x;q)$ is only related to the odd-ordered mock theta functions; however,  both even and odd-ordered mock theta functions can be expressed in terms of Appell functions \cite[Section 5]{HM}, \cite{Zw2}.  For another type of even-ordered building-blocks see \cite{GM, M2}.

\subsection{On admissible-level string functions and mock theta functions}
 In \cite{BoMo2026, BoMo2025}, Borozenets and Mortenson argued that in principle, positive admissible-level string functions can be expressed in terms of Appell functions.  In particular, they used a Hecke-type double-sum expansion \cite[Theorem $1.4$]{HM} and the theory of mock modularity \cite{Zw2} to show that in the case of $1/2$-level string functions:

 \begin{theorem} \label{theorem:fractionalLevelPlus12alt}\cite[Theorem 2.1]{BoMo2026}  Let $(p,p')=(2,5),$ $0\le \ell\le 3$ and $m\in2\mathbb{Z}+\ell$.  We have that
\begin{align*}
 (q)_{\infty}^3 \mathcal{C}_{m,\ell}^{1/2}(q) 
 & = q^{\frac{1}{2}(m-\ell)}  j(q^{1+\ell};q^5)
 \left ( (-1)^{m} q^{\binom{m}{2}} \frac{1}{2}\mu_2(q) 
 +\sum_{k=0}^{m-1}(-1)^kq^{mk-\binom{k+1}{2}}\right )\\
 &\qquad +(-1)^{m} q^{\frac{1}{2}(m-\ell)} q^{\binom{m}{2}} \Theta_{\ell}(q),
\end{align*}
where
\begin{equation*}
\Theta_{\ell}:=
\begin{cases}
\frac{1}{2}\cdot \frac{J_{1}^3}{J_{2}J_{4}}\cdot j(-q;-q^5), & \textup{for} \ \ell \in \{0,3 \},\\
-\frac{1}{2}\cdot \frac{J_{1}^3}{J_{2}J_{4}}\cdot j(-q^3;-q^5), & \textup{for} \ \ell \in \{1,2 \},
\end{cases}
\end{equation*} 
where $\mu_2(q)$ is Ramanujan's second-order mock theta function (\ref{equation:2nd-mu(q)Alt}).
\end{theorem}

As corollaries, they found mock theta conjecture-like identities \cite[Corollary $2.6$]{BoMo2026}.  For example,  for $(p,p^{\prime})=(2,5)$, $r\in\{0,1\}$, it holds
\begin{gather}
(q)_{\infty}^3\mathcal{C}_{0,2r}^{1/2}(q)
=(-q)^{-r}\frac{1}{2}
\frac{J_{1}^3}{J_{2}J_{4}}j(-q^{2r+1};-q^{5})
+ q^{-r} \frac{1}{2}  j(q^{1+2r};q^{5}) \mu_2(q).
 \label{equation:mockThetaConj2502r-2ndmu}
\end{gather}

It appeared that the methods of \cite{BoMo2026} did not readily work for other positive admissible-levels.  One could find results, but nothing like Theorem \ref{theorem:fractionalLevelPlus12alt}.  This led to the search for alternate proofs of Theorem \ref{theorem:fractionalLevelPlus12alt} and thus \cite{BoMo2025}, where the authors developed quasi-periodicity, i.e. they developed a positive admissible-level analog of (\ref{eq:inglevelperiod}).  For even-spin they found 
\begin{theorem}\label{theorem:generalQuasiPeriodicityEvenSpin}
\cite[Theorem 2.1]{BoMo2025} For $(p,p^{\prime})=(p,2p+j)$, we have the quasi-periodic relation for even-spin 
\begin{align*}
& (q)_{\infty}^{3}C_{2jt+2s,2r}^{(p,2p+j)}(q)
 -(q)_{\infty}^{3}C_{2s,2r}^{(p,2p+j)}(q)\\
& \  = (-1)^{p}q^{-\frac{1}{8}+\frac{p(2r+1)^2}{4(2p+j)}}q^{\binom{p}{2}-p(r-s)-\frac{p}{j}s^2}\sum_{i=1}^{t}q^{-2pj\binom{i}{2}-2psi} \\
&\quad  \times 
\sum_{m=1}^{p-1}(-1)^{m}q^{\binom{m+1}{2}+m(r-p)}
 \left ( q^{m(ji+s-j)}-q^{-m(ji+s)}\right )\\
&\quad   \times 
\Big (  j(-q^{m(2p+j)+p(2r+1)};q^{2p(2p+j)} )
  -  q^{m(2p+j)-m(2r+1)}j(-q^{-m(2p+j)+p(2r+1)};q^{2p(2p+j)})\Big ).
\end{align*}
\end{theorem}

Borozenets and Mortenson then applied the notion of quasi-periodicity in Theorem \ref{theorem:generalQuasiPeriodicityEvenSpin} to the Weyl--Kac formula in (\ref{equation:fourcoefexp}) and  (\ref{equation:WK-formula}) to obtain polar-finite decompositions after Zwegers \cite{Zw2} and Dabholkar, Murthy, and Zagier \cite{DMZ}.  From such polar-finite decompositions, they were also able to write characters for levels $1/2$, $1/3$, $2/3$ and $1/5$ in terms of mock theta conjecture-like identities.  In order state some of their results, we recall two of the third-order mock theta functions found in Ramanujan's last letter to Hardy \cite[Chapter 14]{AB2018}:
\begin{gather*}
f_3(q):=\sum_{n\ge 0}\frac{q^{n^2}}{(-q)_n^2}, \ \ 
\omega_3(q)
:=\sum_{n\ge 0}\frac{q^{2n(n+1)}}{(q;q^2)_{n+1}^2}.
\end{gather*}

 For the even-spin, $1/3$-level string functions Borozenets and Mortenson obtained
\begin{theorem}\label{theorem:newMockThetaIdentitiespP37m0ell2r}\cite[Theorem 2.6]{BoMo2025}
For $(p,p^{\prime})=(3,7)$, $r\in \{0,1,2\}$ we have
\begin{align*}
(q)_{\infty}^3&\mathcal{C}_{0,2r}^{1/3}(q)
=(-q)^{-r}\frac{(q)_{\infty}^3}{J_{2}}
\frac{j(-q^{1+2r};q^{14})j(q^{16+4r};q^{28})}
{j(-1;q)J_{28}} \\
&\qquad   -q^{2-2r}\frac{j(q^{6-2r};q^{14})j(q^{26-4r};q^{28})}{J_{28}} \omega_3(-q)
+  \frac{q^{-r}}{2} \frac{j(q^{1+2r};q^{14})j(q^{16+4r};q^{28})}{J_{28}} 
f_{3}(q^2).
\end{align*}
\end{theorem}

We recall the floor function $\lfloor \cdot \rfloor$.  For the even-spin, $2/3$-level string functions, they found
\begin{theorem}\label{theorem:newMockThetaIdentitiespP38m0ell2r}
\cite[Theorem 2.7]{BoMo2025}
 For $(p,p^{\prime})=(3,8)$, $r\in\{0,1,2,3\}$, we have
\begin{align*}
(q)_{\infty}^3&\mathcal{C}_{0,2r}^{2/3}(q) 
=(-1)^{\lfloor (r+1)/2\rfloor}\cdot
\frac{q^{-r}}{2}\frac{J_{1}^2J_{2}}{J_{4}^2J_{8}}
j(-q^{7-2r};q^{16})j(q^{1+2r};q^{8})\\
&\qquad 
- q^{3-2r}    
\frac{j(q^{7-2r};q^{16})j(q^{30-4r};q^{32})}{J_{32}}
\omega_{3}(-q^{2}) 
 + q^{-r} 
\frac{j(q^{1+2r};q^{16})j(q^{18+4r};q^{32})}{J_{32}} 
 \frac{1}{2}f_{3}(q^4).
 \end{align*}
\end{theorem}

\begin{theorem}\label{theorem:newMockThetaIdentitiespP38m2ell2r} 
\cite[Theorem 2.8]{BoMo2025}
For $(p,p^{\prime})=(3,8)$, $r\in\{0,1,2,3\}$, we have
\begin{align*}
(q)_{\infty}^3\mathcal{C}_{2,2r}^{2/3}(q) 
&=(-1)^{\lfloor (r+1)/2\rfloor}\cdot 
\frac{q^{3-2r}}{2}\frac{J_{1}^2J_{2}}{J_{4}^2J_{32}}j(q^{2+4r};q^{32})j(q^{7-2r};q^{16})\\
&\qquad 
- q^{3-2r}   
\frac{j(q^{7-2r};q^{16})j(q^{30-4r};q^{32})}{J_{32}}
\left (1- \frac{1}{2}f_{3}(q^4) \right )  \\
&\qquad + q^{1-r} 
\frac{j(q^{1+2r};q^{16})j(q^{18+4r};q^{32})}{J_{32}} 
 \left (1-q^2\omega_{3}(-q^{2}) \right ).
\end{align*}
\end{theorem}

Borozenets and Mortenson \cite{BoMo2025} and then Konenkov and Mortenson \cite{KM25} also found mock theta conjecture-like identities for $1/5$-level string functions \cite[Theorem $2.9$]{BoMo2025} and $2/5$-level string functions \cite[Theorem $2.5$, $2.6$]{KM25} but in terms of Ramanujan's four tenth-order mock theta functions, where the tenth-order functions were first found in the lost notebook \cite{RLN}.

\smallskip
In this work, we will first obtain the complimentary formula for quasi-periodicity for odd-spin admissible-level string functions.  We will then apply the formulas for quasi-peridocity for both even and odd-spin string functions to a generalized Euler identity found in Schilling and Warnaar \cite[Proposition 8.2]{SW}.  We obtain formulas similar to Theorem \ref{theorem:fractionalLevelPlus12alt}, i.e. \cite[Theorem $2.1$]{BoMo2026}, and we discover new mock theta conjecture-like identities.  In addition to the previously encountered mock theta functions $\mu_{2}(q)$, $f_{3}(q)$, and $\omega_{3}(q)$ we encounter two more of Ramanujan's classical mock theta functions in $\psi_{3}(q)$ and $\chi_{3}(q)$!

\subsection{On a generalized Euler identity}
We define $\mathbb{Z}_{p}:=\{ 0,1,2,\dots,p-1\}$.  We state Schilling and Warnaar's generalized Euler identity for modified string functions:
  \begin{proposition}\cite[Proposition 8.2]{SW}\label{propostion:genEulerIdentitySW}  For $1\le p < p^{\prime}$, $\ell\in\mathbb{Z}_{p^{\prime}-1}$, $\eta\in \mathbb{Z}_{p^{\prime}}$ such that $\ell+\eta$ is even.  Then
  \begin{equation}
  \sum_{L=-\infty}^{\infty}(-1)^{L}q^{\binom{L}{2}}\mathcal{C}_{2L+\eta,\ell}^{p,p^{\prime}}(q)=\delta_{\ell,\eta}.
  \label{equation:genEulerIdentitySW}
  \end{equation}
  \end{proposition}
  \begin{remark}  Schilling and Warnaar point out that this is the classical Euler identity for $(p,p^{\prime})=(1,3)$, and that for $p=1$ and arbitrary $p^{\prime}$ this is the case $A_{1}^{(1)}$ of equation $(2.1.16)$ of \cite{KW90}, which is a corollary of Kac and Wakimoto \cite[Theorem $2.1$]{KW90}.
  \end{remark}
  
  For even-spin, integral-level string functions, it is straight-forward to use the classical periodicity to obtain
\begin{align}
 \delta_{2r,2s}
  &=\sum_{a=0}^{N-1}(-1)^{a}q^{\binom{a}{2}} 
j\big ({-}(-1)^{N}q^{Na+\binom{N+1}{2}+2(a+s)};q^{N(N+2)}\big )
\mathcal{C}_{2(a+s),2r}^{N}(q).\label{equation:intLevelGenEulerEvenSpin}
\end{align}
For $N=1$ the above reduces to the elegant
\begin{align*}
 \delta_{2r,2s}
  &=
j\big (q^{1+2s};q^{3}\big )
\mathcal{C}_{2s,2r}^{1}(q),
\end{align*}
where $r=0$ and $s\in\{0,1\}$.  We remind the reader that $j(q^n;q)=0$ for $n\in\mathbb{Z}$.

\smallskip
When the string functions are of positive fractional-level, the results on quasi-periodicity yield very different results.  We do obtain a finite sum of admissible-level string functions similar to what is found in (\ref{equation:intLevelGenEulerEvenSpin}); however, we also obtain an additional finite sum of the form
\begin{equation*}
\sum_{i}\Phi_{i}(q)\Psi_{i}(q),
\end{equation*}
where the $\Phi_i(q)$'s are modular and depend only on the spin and the $\Psi_{i}(q)$'s are (mixed) mock modular Hecke-type double-sums and depend only on the quantum number.  For levels $1/2$, $1/3$, and $2/3$, we shall also see that the $\Psi_{i}(q)$'s give us families of mock theta conjecture-like identities for symmetric Hecke-type double-sums. Our expressions are in terms of Ramanujan's second-order mock theta function $\mu_2(q)$ and third-order mock theta functions $f_{3}(q)$, $\omega_3(q)$, $\psi_{3}(q)$, and $\chi_3(q)$.  Obtaining the mock theta conjecture-like identities for the $\Psi_{i}(q)$'s will be the focus of our paper.


\section{Results: On the generalized Euler identity for even-spin}\label{section:resultsEvenSpin}

We apply the quasi-periodicity for even-spin in Theorem \ref{theorem:generalQuasiPeriodicityEvenSpin} to the generalized Euler identity (\ref{equation:genEulerIdentitySW}).  This gives
\begin{theorem} \label{theorem:genEulerIdentityEvenSpin} For $2r\in\mathbb{Z}_{2p+j-1}$ and $2s\in\mathbb{Z}_{2p+j}$, we have
\begin{align*}
(q)_{\infty}^{3}\delta_{2r,2s}
&=(q)_{\infty}^{3}\sum_{a=0}^{j-1}(-1)^{a} q^{\binom{a}{2}} j(-(-1)^{j}q^{j(a+p)+\binom{j}{2}+2p(a+s)};q^{j(2p+j)})
\mathcal{C}_{2(a+s),2r}^{(p,2p+j)}(q)\\
&\qquad 
+(-1)^{p}q^{\binom{p}{2}-p(r-s)} \sum_{a=0}^{j-1} (-1)^{a}q^{pa}\sum_{m=1}^{p-1}(-1)^{m}q^{\binom{m+1}{2}+m(r-p)}
\Phi_{m}^{r}(q)\Psi_{a,m}^{s}(q),
\end{align*}
where
\begin{equation*}
\Phi_{m}^{r}(q):= j(-q^{m(2p+j)+p(2r+1)};q^{2p(2p+j)} )   -  q^{m(2p+j)-m(2r+1)}j(-q^{-m(2p+j)+p(2r+1)};q^{2p(2p+j)}),
\end{equation*}
\noindent and
\begin{align*}
\Psi_{a,m}^{s}(q)&:=(-1)^{j}q^{\binom{a}{2}+\binom{j}{2}+j(a+p)+m(a+s)}\\
&\qquad \times f_{2p+j,2p+j,j}(-(-1)^{j}q^{a(2p+j)+2ps+3pj+j+3\binom{j}{2}},-(-1)^{j}q^{3\binom{j}{2}+j(a+m+p+1)};q^j)\\
&\quad - (-1)^{j}q^{\binom{a}{2}+\binom{j}{2}+j(a+p)-m(a+s+j)}\\
&\qquad \times f_{2p+j,2p+j,j}(-(-1)^{j}q^{a(2p+j)+2ps+3pj+j+3\binom{j}{2}},-(-1)^{j}q^{3\binom{j}{2}+j(a-m+p+1)};q^j).
\end{align*}
\end{theorem}

\noindent We note that the variable $j$ determining the level in $(p,p^{\prime})=(p,2p+j)$ should not be confused with the notation for the theta function $j(z;q)$.  This will be clear from the usage.

We first state a useful corollary for the case when the variable $j=1$ and we have even-spin:

\begin{corollary}\label{corollary:genEulerIdentityEvenSpinCorollary} For $2r\in\mathbb{Z}_{2p}$ and $2s\in\mathbb{Z}_{2p+1}$, we have
{\allowdisplaybreaks \begin{align*}
&(q)_{\infty}^{3}\delta_{2r,2s}\\
&\quad =(q)_{\infty}^{3}j(q^{p+2ps};q^{(2p+1)})
\mathcal{C}_{2s,2r}^{(p,2p+1)}(q)\\
&\qquad 
+(-1)^{p}q^{\binom{p}{2}-p(r-s)} \sum_{m=1}^{p-1}(-1)^{m}q^{\binom{m+1}{2}+m(r-p)}\\
&\qquad   \times 
\Big (  j(-q^{m(2p+1)+p(2r+1)};q^{2p(2p+1)} )
  -  q^{m(2p+1)-m(2r+1)}j(-q^{-m(2p+1)+p(2r+1)};q^{2p(2p+1)})\Big )\\
&\qquad \times \Big ( -q^{p+ms} f_{2p+1,2p+1,1}(q^{2ps+3p+1},q^{m+p+1};q)\\
&\qquad \qquad   +q^{p-m(s+1)} f_{2p+1,2p+1,1}(q^{2ps+3p+1},q^{-m+p+1};q)\Big ). 
\end{align*}}%
\end{corollary}
We provide some examples so that we can compare with Theorems \ref{theorem:fractionalLevelPlus12alt},   
\ref{theorem:newMockThetaIdentitiespP37m0ell2r}, \ref{theorem:newMockThetaIdentitiespP38m0ell2r}, \ref{theorem:newMockThetaIdentitiespP38m2ell2r}.

\subsection{On string functions of admissible-level $1/2$ and the function $\mu_{2}(q)$}

We start off with the easiest example.  We have
\begin{corollary}\label{corollary:levelOneHalfEvenSpin}  For $p=2$, $j=1$, $r\in\{0,1\}$ and $s\in\{0,1,2\}$ we have
\begin{align*}
(q)_{\infty}^{3}\delta_{2r,2s}
&=(q)_{\infty}^{3} j(q^{2+4s};q^{5})
\mathcal{C}_{2s,2r}^{1/2}(q)\\
&\qquad    -q^{-r} j(q^{1+2r};q^5)  \times  \Big ( -q^{2+3s} f_{5,5,1}(q^{4s+7},q^{4};q) +q^{1+s} f_{5,5,1}(q^{4s+7},q^{2};q)\Big ).
\end{align*}
\end{corollary}
\noindent For $s=2$, this implies that the double-sum coefficient vanishes, see also Lemma \ref{lemma:degenerateDoubleSumCoeffs}.   

\smallskip
Throughout the paper, we will be focusing on evaluating the Hecke-type double-sum coefficients of the theta functions, and we will be expressing them in terms of mock theta function identities.  The next theorem is probably the closest to the previous result in Theorem \ref{theorem:fractionalLevelPlus12alt}, \cite[Theorem 2.1]{BoMo2026}.
\begin{theorem}\label{theorem:genEulerOneHalfEvenSpin} It holds
\begin{align*}
-q^{2+3s}&f_{5,5,1}(q^{4s+7},q^4;q)+q^{1+s}f_{5,5,1}(q^{4s+7},q^2;q)\\
& =(-1)^{s}j(q^{s+3};q^5)
\left (q^{\binom{s+1}{2}}\frac{1}{2}\mu_2(q)+ q^{-3\binom{s}{2}}\sum_{k=0}^{2s-1}( -1)^{k}q^{2sk-\binom{k+1}{2}}\right )\\
&\qquad \qquad -(-1)^s(-q)^{\binom{s+1}{2}}\frac{1}{2}\frac{J_{1}^3}{J_{2}J_{4}}j((-q)^{s+3};-q^5).
\end{align*}
\end{theorem}
For the case $s=0$, the Hecke-type double-sum in Theorem \ref{theorem:genEulerOneHalfEvenSpin}  gives us the following mock theta conjecture-like identity:
\begin{corollary}\label{corollary:genEulerOneHalfEvenSpin} It holds
\begin{equation}
-q^{2} f_{5,5,1}(q^{7},q^{4};q) +q f_{5,5,1}(q^{7},q^{2};q)
=\frac{1}{2}j(q^2;q^5)\mu_2(q)-\frac{1}{2}\frac{J_{1}^3}{J_{2}J_{4}}j({-}q^{3};-q^5).
\label{equation:genEulerOneHalfEvenSpinMockThetaId}
\end{equation}
\end{corollary}


\subsection{On string functions of admissible-level $1/3$ and the functions $f_{3}(q)$ and $\omega_{3}(q)$}
The examples in this section have the same form as Theorem \ref{theorem:fractionalLevelPlus12alt}, i.e. \cite[Theorem 2.1]{BoMo2026}.   The examples also share the same mock theta functions as in Theorem \ref{theorem:newMockThetaIdentitiespP37m0ell2r}, i.e. \cite[Theorem 2.6]{BoMo2025}, but the identities have a different form.  Instead of Appell functions, the theta function coefficients are Hecke-type double-sums which evaluate to the same mock theta functions $f_{3}(q)$ and $\omega_{3}(q)$ plus a simple quotient of theta functions.

\begin{corollary}\label{corollary:levelOneThirdEvenSpin}  For $p=3$, $j=1$, $r\in\{ 0,1,2\}$ and $s\in\{ 0,1,2,3\}$ we have
{\allowdisplaybreaks \begin{align*}
(q)_{\infty}^{3}\delta_{2r,2s}
&=(q)_{\infty}^{3}j(q^{3+6s};q^{7})
\mathcal{C}_{2s,2r}^{1/3}(q)\\
&\qquad 
+q^{1-2r} \frac{j(q^{8+2r};q^{14})j(q^{2+4r};q^{28})}{J_{28}}\\
&\qquad \qquad  \times \Big ( -q^{3+4s}
 f_{7,7,1}(q^{6s+10},q^{5};q) +q^{2+2s} f_{7,7,1}(q^{6s+10},q^{3};q)\Big )\\
&\qquad  -q^{-r}
\frac{j(q^{1+2r};q^{14})j(q^{16+4r};q^{28})}{J_{28}}\\
&\qquad \qquad \times \Big ( -q^{3+5s}
 f_{7,7,1}(q^{6s+10},q^{6};q) +q^{1+s} f_{7,7,1}(q^{6s+10},q^{2};q)\Big ).
\end{align*}}%
\end{corollary}
\noindent For the case $s=3$ of Corollary \ref{corollary:levelOneThirdEvenSpin}, we have that the two double-sum coefficients vanish, see also Lemma \ref{lemma:degenerateDoubleSumCoeffs}.  

We evaluate the two Hecke-type double-sum coefficients separately.  For the first Hecke-type double-sum coefficient, we have
\begin{theorem} \label{theorem:genEulerOneThirdFirstPairEvenSpin} It holds
\begin{align*}
 -q^{3+4s}&
 f_{7,7,1}(q^{6s+10},q^{5};q) +q^{2+2s} f_{7,7,1}(q^{6s+10},q^{3};q)\\
 &=(-1)^sj(q^{s+4};q^{7})
 \left ( q^{\binom{s+1}{2}}\cdot q\omega_3(-q)
 +q^{-5\binom{s}{2}}\sum_{k=0}^{s-1}q^{6ks-6\binom{k+1}{2}}\left( q^{k}-q^{-k+2s-1}\right ) \right ) \\
 &\qquad \qquad -q^{5s^2-3s+1}\frac{J_{1}^3J_{4}}{J_{2}^3}j(q^{8+16s};q^{28}).
\end{align*}
\end{theorem}
\noindent For the second Hecke-type double-sum coefficient, we have
\begin{theorem} \label{theorem:genEulerOneThirdSecondPairEvenSpin} It holds
\begin{align*}
 -q^{3+5s}&
 f_{7,7,1}(q^{6s+10},q^{6};q) +q^{1+s} f_{7,7,1}(q^{6s+10},q^{2};q)\\
 &=(-1)^sj(q^{s+4};q^{7})
 \left ( q^{\binom{s+1}{2}}\frac{1}{2}f_{3}(q^2)
 +q^{-5\binom{s}{2}-s}\sum_{k=0}^{s-1}q^{6ks-6\binom{k+1}{2}}\left( q^{2k}-q^{-2k+4s-2}\right ) \right ) \\
 &\qquad \qquad -(-1)^{s}\frac{(-q)^{\binom{s+1}{2}}}{2}\frac{J_{1}^2}{J_{4}}j((-q)^{s+4};-q^7).
\end{align*}
\end{theorem}
 
The cases $s=0$ of Theorems  \ref{theorem:genEulerOneThirdFirstPairEvenSpin} and \ref{theorem:genEulerOneThirdSecondPairEvenSpin} give us the following mock theta conjecture-like identities:
\begin{corollary}\label{corollary:genEulerOneThirdEvenSpin} It holds
{\allowdisplaybreaks \begin{gather*}
-q^3f_{7,7,1}(q^{10},q^{5};q)+q^2f_{7,7,1}(q^{10},q^{3};q)
=j(q^3;q^7)\cdot q\omega_3(-q)
-q\frac{J_{1}^3J_{4}}{J_{2}^3}\cdot J_{8,28},\\
-q^3f_{7,7,1}(q^{10},q^{6};q)+qf_{7,7,1}(q^{10},q^{2};q)
=j(q^3;q^7)\frac{1}{2}f_{3}(q^2)
-\frac{1}{2}\frac{J_{1}^2}{J_{4}}\cdot j(-q^3;-q^7).
\end{gather*}}%
\end{corollary}

We note that the simple quotients of theta functions can often be written over a common base, where all of the information is located in the numerator.  In particular, the identities in Corollary \ref{corollary:genEulerOneThirdEvenSpin} can be rewritten as follows:
\begin{align}
-&q^3f_{7,7,1}(q^{10},q^{5};q)+q^2f_{7,7,1}(q^{10},q^{3};q)\\
&=j(q^3;q^7)\cdot q\omega_3(-q)
-q\cdot J_{1,28}^3J_{3,28}^3J_{4,28}J_{5,28}^3J_{7,28}^3J_{8,28}^2J_{9,28}^3J_{11,28}^3J_{12,28}J_{13,28}^3
\frac{1}{J_{28}^{23}},\notag
\end{align}
and
\begin{align}
-&q^3f_{7,7,1}(q^{10},q^{6};q)+qf_{7,7,1}(q^{10},q^{2};q)\\
&=\frac{1}{2}j(q^3;q^7)f_3(q^2)
-\frac{1}{2}
J_{1,28}^2J_{2,28}^2J_{3,28}J_{4,28}^2J_{5,28}^2J_{6,28}^3J_{7,28}J_{8,28}J_{9,28}^2J_{10,28}^3J_{11,28}J_{12,28}J_{13,28}^2\frac{J_{14}^4}{J_{28}^{25}}.\notag
\end{align}

\subsection{On string functions of admissible-level $2/3$ and the functions $f_{3}(q)$ and $\omega_{3}(q)$}

The examples for even-spin, $2/3$-level again have the same form as Theorem \ref{theorem:fractionalLevelPlus12alt}, i.e. \cite[Theorem 2.1]{BoMo2026}.   The examples also share the same mock theta functions as in Theorems \ref{theorem:newMockThetaIdentitiespP38m0ell2r}, \ref {theorem:newMockThetaIdentitiespP38m2ell2r} , i.e. \cite[Theorems 2.7, 2.8]{BoMo2025}, but the identities again have a different form.  Instead of Appell functions, the theta function coefficients are Hecke-type double-sums which again evaluate to the same mock theta functions $f_{3}(q)$ and $\omega_{3}(q)$ plus a simple quotient of theta functions.

\begin{corollary}\label{corollary:levelTwoThirdsEvenSpin} For $p=3$ and $j=2$, $r\in\{0,1,2,3\}$ and $s\in\{0,1,2,3\}$ we have
{\allowdisplaybreaks \begin{align*}
(q)_{\infty}^{3}\delta_{2r,2s}
&=(q)_{\infty}^{3}  j(-q^{7+6s};q^{16})
\mathcal{C}_{2s,2r}^{(3,8)}(q)
-(q)_{\infty}^{3} j(-q^{15+6s};q^{16})
\mathcal{C}_{2(1+s),2r}^{(3,8)}(q)\\
&\qquad + q^{1-2r}
\frac{j(q^{9+2r};q^{16})j(q^{2+4r};q^{32})}{J_{32}}\\
&\qquad   \times\Big ( q^{7+4s}
 f_{4,4,1}(-q^{6s+23},-q^{13};q^4)
   - q^{5+2s} f_{4,4,1}(-q^{6s+23},-q^{9};q^4)\\
&\qquad \qquad -q^{13+4s}
 f_{4,4,1}(-q^{31+6s},-q^{15};q^4)
  + q^{9+2s} f_{4,4,1}(-q^{31+6s},-q^{11};q^4)\Big ) \\
&\qquad -q^{-r}
\frac{j(q^{1+2r};q^{16})j(q^{18+4r};q^{32})}{J_{32}}\\
&\qquad   \times\Big ( q^{7+5s}
 f_{4,4,1}(-q^{6s+23},-q^{15};q^4)
  - q^{3+s} f_{4,4,1}(-q^{6s+23},-q^{7};q^4)\\
&\qquad \qquad - q^{14+5s}
 f_{4,4,1}(-q^{31+6s},-q^{17};q^4)
   + q^{6+s} f_{4,4,1}(-q^{31+6s},-q^{9};q^4)
\Big ).
\end{align*}}%
\end{corollary}
In order to obtain our mock theta conjecture-like identities from the two Hecke-type double-sum coefficients, we need to consider the cases $s$ even and $s$ odd.   For the first set of four double-sums, we have for $s\to 2s$ even that
\begin{theorem}\label{theorem:level23EvenSpinFirstQuad-sEven}   It holds
{\allowdisplaybreaks \begin{align*}
&q^{7+8s}f_{4,4,1}(-q^{23+12s},-q^{13};q^4)-q^{5+4s}f_{4,4,1}(-q^{23+12s},-q^9;q^4)\\
&\quad -q^{13+8s}f_{4,4,1}(-q^{31+12s},-q^{15};q^4)+q^{9+4s}f_{4,4,1}(-q^{31+12s},-q^{11};q^{4})\\
&=q^{7+16s}j(-q^{23+12s};q^{16})
\left(q^{2s+12\binom{s}{2}}\cdot q^{2}\omega_{3}(-q^2)
+\sum_{k=0}^{s-1}q^{12sk-12\binom{k+1}{2}}\left ( q^{2k}-q^{-2k+4s-2}\right )  \right)\\
&\quad -q^{17+16s}j(-q^{31+12s};q^{16})\\
&\qquad \times \left ( 1-q^{4s}\left( q^{4s+12\binom{s}{2}}\frac{1}{2}f_{3}(q^4)
+\sum_{k=0}^{s-1}q^{12sk-12\binom{k+1}{2}}\left ( q^{4k}-q^{-4k+8s-4}\right )\right) \right )\\
&\quad -(-1)^{s}\frac{q^{\binom{2s}{2}+2}}{2}\frac{J_{1}^2J_{2}}{J_{4}^2}j(-q^{1+4s};q^{16}).
\end{align*}}%
\end{theorem}
For the first set of four double-sums, we have for $s\to 2s+1$ odd that
\begin{theorem} \label{theorem:level23EvenSpinFirstQuad-sOdd} It holds
{\allowdisplaybreaks \begin{align*}
&q^{11+8s}f_{4,4,1}(-q^{29+12s},-q^{13};q^4)-q^{7+4s}f_{4,4,1}(-q^{29+12s},-q^9;q^4)\\
&\quad -q^{17+8s}f_{4,4,1}(-q^{37+12s},-q^{15};q^4)+q^{11+4s}f_{4,4,1}(-q^{37+12s},-q^{11};q^{4})\\
&=q^{15+16s}j(-q^{29+12s};q^{16})\\
&\qquad \times \left(1-q^{4s}\left ( \sum_{k=0}^{s-1}q^{12sk-12\binom{k+1}{2}}\left (q^{4k} -q^{-4k+8s-4}\right ) 
 +q^{4s+12\binom{s}{2}}\frac{1}{2}f_3(q^4)\right ) \right)\\
&\quad -q^{25+16s}j(-q^{37+12s};q^{16})\\
&\qquad \times \left ( 1-q^{4s+2}\left ( 1-q^{8s} \left ( \sum_{k=0}^{s-1}q^{12sk-12\binom{k+1}{2}}\left (q^{2k} -q^{-2k+4s-2}\right ) 
 +q^{2s+12\binom{s}{2}}\cdot q^2\omega_{3}(-q^2)\right )\right )   \right )\\
&\quad -(-1)^{s}\frac{q^{\binom{2s+1}{2}+2}}{2}\frac{J_{1}^2J_{2}}{J_{4}^2}j(-q^{3+4s};q^{16}).
\end{align*}}%
\end{theorem}

For the second set of four double-sums, we have for $s\to 2s$ even that
\begin{theorem} \label{theorem:level23EvenSpinSecondQuad-sEven} For $s$ even, i.e. we replace $s$ by $2s$, we have
{\allowdisplaybreaks \begin{align*}
&q^{7+10s}
 f_{4,4,1}(-q^{12s+23},-q^{15};q^4)
   - q^{3+2s} f_{4,4,1}(-q^{12s+23},-q^{7};q^4)\\
&\qquad  -q^{14+10s}
 f_{4,4,1}(-q^{31+12s},-q^{17};q^4)
  + q^{6+2s} f_{4,4,1}(-q^{31+12s},-q^{9};q^4)\\
  &= q^{7+14s}  j(-q^{12s+23};q^{16})\left ( \sum_{k=0}^{s-1}q^{12sk-12\binom{k+1}{2}}\left (q^{4k} -q^{-4k+8s-4}\right ) 
 +q^{4s+12\binom{s}{2}}\frac{1}{2}f_{3}(q^4) \right ) \\
 &\qquad - q^{16+14s}  j(-q^{31+12s};q^{16})\\
 &\qquad \qquad \times  \left ( 1-q^{8s}\left (\sum_{k=0}^{s-1}q^{12sk-12\binom{k+1}{2}}\left (q^{2k} -q^{-2k+4s-2}\right ) 
 +q^{2s+12\binom{s}{2}}\cdot q^2\omega_{3}(-q^2) \right ) \right ) \\
 &\qquad -(-1)^{s}\frac{q^{\binom{2s+1}{2}}}{2}\frac{J_{1}^2J_{2}}{J_{4}^2}j(-q^{9+4s};q^{16}).
\end{align*}}%
\end{theorem}

For the second set of four double-sums, we have for $s\to 2s+1$ odd that
\begin{theorem}  \label{theorem:level23EvenSpinSecondQuad-sOdd}
For $s$ odd, i.e. we replace $s$ by $2s+1$, we have
{\allowdisplaybreaks \begin{align*}
&q^{12+10s}
 f_{4,4,1}(-q^{12s+29},-q^{15};q^4)
   - q^{4+2s} f_{4,4,1}(-q^{12s+29},-q^{7};q^4)\\
&\qquad  -q^{19+10s}
 f_{4,4,1}(-q^{37+12s},-q^{17};q^4)
  + q^{7+2s} f_{4,4,1}(-q^{37+12s},-q^{9};q^4)\\
&=
 q^{14+14s}  j(-q^{12s+29};q^{16})\\
 &\qquad \times \left ( 1- q^{8s}\left (\sum_{k=0}^{s-1}q^{12sk-12\binom{k+1}{2}}\left (q^{2k} -q^{-2k+4s-2}\right ) 
 +q^{2s+12\binom{s}{2}}\cdot q^{2}\omega_{3}(-q^2)\right ) \right ) \\
&\quad  
 - q^{23+14s}  j(-q^{37+12s};q^{16})\\
 &\qquad \times \left ( 1-q^{8s+4} \left ( 1-q^{4s}\left ( \sum_{k=0}^{s-1}q^{12sk-12\binom{k+1}{2}}\left (q^{4k} -q^{-4k+8s-4}\right ) 
 +q^{4s+12\binom{s}{2}}\frac{1}{2}f_{3}(q^4) \right ) \right ) \right )\\
 & \quad - (-1)^{s}\frac{q^{\binom{2s+2}{2}}}{2}\frac{J_{1}^2J_{2}}{J_{4}^2}j(-q^{11+4s};q^{16}).
\end{align*}}%
\end{theorem}

For the cases $s=0$, the Hecke-type double-sums of Theorems \ref{theorem:level23EvenSpinFirstQuad-sEven} and  \ref{theorem:level23EvenSpinSecondQuad-sEven} give us the following mock theta conjecture-like identities:
\begin{corollary} \label{corollary:genEulerTwoThirdsEvenSpin} It holds 
{\allowdisplaybreaks \begin{align}
q^{7} &f_{4,4,1}(-q^{23},-q^{13};q^4)
   - q^{5} f_{4,4,1}(-q^{23},-q^{9};q^4)\\
&\quad  -q^{13} f_{4,4,1}(-q^{31},-q^{15};q^4)
  + q^{9} f_{4,4,1}(-q^{31},-q^{11};q^4)\notag \\
&=j(-q^7;q^{16})\cdot q^2\omega_3(-q^2)-q^2j(-q;q^{16})\left ( 1- \frac{1}{2}\left ( f_3(q^4)- \frac{J_{1}^2J_{2}}{J_{4}^2}\right )\right ),
\notag\\
q^{7}& f_{4,4,1}(-q^{23},-q^{15};q^4)
  - q^{3} f_{4,4,1}(-q^{23},-q^{7};q^4)\\
&\quad - q^{14} f_{4,4,1}(-q^{31},-q^{17};q^4)
   + q^{6} f_{4,4,1}(-q^{31},-q^{9};q^4)\notag\\
&=-qj(-q;q^{16})\left ( 1- q^2\omega_3(-q^2)\right)+  \frac{1}{2}j(-q^{7};q^{16})\left ( f_3(q^4)- \frac{J_{1}^2J_{2}}{J_{4}^2}\right ).
\notag
\end{align}}%
\end{corollary}

The reader will notice that we have combined the simple quotients of theta functions with the term involving $f_{3}(q^4)$.  For the four even-spin, $2/3$-level Theorems \ref{theorem:level23EvenSpinFirstQuad-sEven}, \ref{theorem:level23EvenSpinFirstQuad-sOdd}, \ref{theorem:level23EvenSpinSecondQuad-sEven}, \ref{theorem:level23EvenSpinSecondQuad-sOdd}, the same is possible.  In all four theorems the simple quotient of theta functions can be combined with $f_{3}(q^4)$.  In Theorems \ref{theorem:level23EvenSpinFirstQuad-sEven} and \ref{theorem:level23EvenSpinSecondQuad-sEven}, one then replaces $f_{3}(q^4)$ with
\begin{equation*}
f_{3}(q^4)-(-1)^{s}\frac{J_{1}^2J_{2}}{J_{4}^2};
\end{equation*}
however, in Theorems \ref{theorem:level23EvenSpinFirstQuad-sOdd} and \ref{theorem:level23EvenSpinSecondQuad-sOdd}, one  replaces $f_{3}(q^4)$ with
\begin{equation*}
f_{3}(q^4)+(-1)^{s}\frac{J_{1}^2J_{2}}{J_{4}^2}.
\end{equation*}
For all four theorems this is easy to see using the quasi-elliptic transformation property of the theta function (\ref{equation:j-elliptic}).  For example, for Theorem \ref{theorem:level23EvenSpinFirstQuad-sEven} one notes that
\begin{align*}
q^{17+16s}j(-q^{31+12s};q^{16})q^{4s}q^{4s+12\binom{s}{2}}&=q^{17+16s}j(-q^{15+16(s+1)-4s};q^{16})q^{4s}q^{4s+12\binom{s}{2}}\\
&=q^{17+16s}q^{-16\binom{s+1}{2}-(s+1)(15-4s)}j(-q^{15-4s};q^{16})q^{4s}q^{4s+12\binom{s}{2}}\\
&=q^{\binom{2s}{2}+2}j(-q^{1+4s};q^{16}).
\end{align*}

The identities in Corollary \ref{corollary:genEulerTwoThirdsEvenSpin} can also be rewritten over a common base as follows:
\begin{align}
q^{7} &f_{4,4,1}(-q^{23},-q^{13};q^4)
   - q^{5} f_{4,4,1}(-q^{23},-q^{9};q^4)\\
&\qquad  -q^{13} f_{4,4,1}(-q^{31},-q^{15};q^4)
  + q^{9} f_{4,4,1}(-q^{31},-q^{11};q^4)\notag\\
&=q^2j(-q^7;q^{16})\omega_3(-q^2)-q^2j(-q^{15};q^{16})\left ( 1- \frac{1}{2}f_3(q^4)\right )
\notag\\
&\quad -\frac{q^2}{2}\cdot J_{1,32}J_{2,32}^4J_{3,32}^2J_{4,32}J_{5,32}^2J_{6,32}^3J_{7,32}^2J_{8,32}J_{9,32}^2J_{10,32}^3J_{11,32}^2J_{12,32}J_{13,32}^2J_{14,32}^3J_{15,32}
\cdot \frac{J_{16}^2}{J_{32}^{30}},\notag
\end{align}
and
\begin{align}
q^{7}& f_{4,4,1}(-q^{23},-q^{15};q^4)
  - q^{3} f_{4,4,1}(-q^{23},-q^{7};q^4)\\
&\qquad - q^{14} f_{4,4,1}(-q^{31},-q^{17};q^4)
   + q^{6} f_{4,4,1}(-q^{31},-q^{9};q^4)
   \notag\\
&=  \frac{1}{2}j(-q^{7};q^{16})f_3(q^4)-qj(-q^{15};q^{16})\left ( 1- q^2\omega_3(-q^2)\right)
\notag\\
&\quad  -\frac{1}{2}\cdot J_{1,32}^2J_{2,32}^{3}J_{3,32}^{2}J_{4,32}J_{5,32}^2J_{6,32}^{3}J_{7,32}J_{8,32}J_{9,32}J_{10,32}^3J_{11,32}^2J_{12,32}J_{13,32}^2J_{14,32}^4J_{15,32}^2
\cdot \frac{J_{16}^2}{J_{32}^{30}}.\notag
\end{align}

We close the section by pointing out that the identities in Corollary \ref{corollary:genEulerTwoThirdsEvenSpin} are reminiscent of an identity found in Mortenson and Sahu \cite[(2.18)]{MoSa}.  The identity expresses a double-sum in terms of two mock theta functions  and a sum of theta quotients. The $q$-hypergeometric form of the double-sum was originally found in Andrews \cite[(1.11)]{And2012}.

\section{Results: On the generalized Euler identity for odd-spin}\label{section:resultsOddSpin}
We first obtain the quasi-periodic relation for odd-spin admissible-level string functions.  The proof and statement are analogous to Theorem \ref{theorem:generalQuasiPeriodicityEvenSpin}.

\begin{theorem}\label{theorem:generalQuasiPeriodicityOddSpin}
For $(p,p^{\prime})=(p,2p+j)$, we have the quasi-periodic relation for odd-spin 
{\allowdisplaybreaks \begin{align*}
& (q)_{\infty}^{3}C_{2jt+2s+1,2r+1}^{(p,2p+j)}(q)
 -(q)_{\infty}^{3}C_{2s+1,2r+1}^{(p,2p+j)}(q)\\
& \ = (-1)^{p}q^{-\frac{1}{8}+\frac{p(2r+2)^2}{4(2p+j)}}q^{\binom{p+1}{2}-p(r+1-s)-\frac{p}{4j}(2s+1)^2}
\sum_{i=1}^{t}q^{-2pj\binom{i}{2}-p(2s+1)i}\\
&\quad \times \sum_{m=1}^{p-1}(-1)^{m}q^{\binom{m+1}{2}+m(r-p)}
\left (q^{m(ji+s-j+1)} -q^{-m(ji+s)}\right ) \\
&\quad  \times\Big (  j(-q^{m(2p+j)+p(2r+2)};q^{2p(2p+j)}) 
 -q^{m(2p+j)-m(2r+2)}j(-q^{-m(2p+j)+p(2r+2)};q^{2p(2p+j)})\Big ).
\end{align*}}
\end{theorem}

Using our new quasi-periodic relation for odd-spin, we obtain the odd-spin companion to our new even-spin Theorem  \ref{theorem:genEulerIdentityEvenSpin}.

\begin{theorem} \label{theorem:genEulerIdentityOddSpin} For $2r+1\in\mathbb{Z}_{2p+j-1}$ and $2s+1\in\mathbb{Z}_{2p+j}$, we have
{\allowdisplaybreaks \begin{align*}
(q)_{\infty}^{3}\delta_{2r+1,2s+1}
& =(q)_{\infty}^{3}\sum_{a=0}^{j-1} (-1)^{a}q^{\binom{a}{2}}j(-(-1)^jq^{j(a+p)+\binom{j}{2}+p(2(a+s)+1)};q^{j(2p+j)})
\mathcal{C}_{2(a+s)+1,2r+1}^{(p,2p+j)}(q)\\
&\qquad 
+(-1)^{p}q^{\binom{p}{2}-p(r-s)}\sum_{a=0}^{j-1}(-1)^{a} q^{pa}
 \sum_{m=1}^{p-1}(-1)^{m}q^{\binom{m+1}{2}+m(r-p)}\Phi_{m}^{r}(q)\Psi_{a,m}^{s}(q),
\end{align*}}%
where
\begin{equation*}
\Phi_{m}^{r}(q):= j(-q^{m(2p+j)+p(2r+2)};q^{2p(2p+j)})   -q^{m(2p+j)-m(2r+2)}j(-q^{-m(2p+j)+p(2r+2)};q^{2p(2p+j)})
\end{equation*}
\noindent and
{\allowdisplaybreaks \begin{align*}
\Psi_{a,m}^{s}(q)&=(-1)^{j}q^{\binom{a}{2}+\binom{j}{2}+j(a+p)+m(a+s+1)}\\
&\qquad \times f_{2p+j,2p+j,j}(-(-1)^{j}q^{p(2(a+s)+1)+aj+3pj+j+3\binom{j}{2}},-(-1)^{j}q^{3\binom{j}{2}+j(a+m+p+1)};q^j)\\
&\quad - (-1)^{j}q^{\binom{a}{2}+\binom{j}{2}+j(a+p)-m(a+s+j)}\\
&\qquad \times f_{2p+j,2p+j,j}(-(-1)^{j}q^{p(2(a+s)+1)+aj+3pj+j+3\binom{j}{2}},-(-1)^{j}q^{3\binom{j}{2}+j(a-m+p+1)};q^j).
\end{align*}}%
\end{theorem}

We present a useful corollary for the case when the variable $j=1$ and we have odd-spin:
\begin{corollary}\label{corollary:genEulerIdentityOddSpinCorollary} For $2r+1\in\mathbb{Z}_{2p+j-1}$ and $2s+1\in\mathbb{Z}_{2p+j}$, we have
{\allowdisplaybreaks \begin{align*}
&(q)_{\infty}^{3}\delta_{2r+1,2s+1}\\
&\quad =(q)_{\infty}^{3}j(q^{p+p(2s+1)};q^{(2p+1)})
\mathcal{C}_{2s+1,2r+1}^{(p,2p+1)}(q)\\
&\qquad 
+(-1)^{p}q^{\binom{p}{2}-p(r-s)}
 \sum_{m=1}^{p-1}(-1)^{m}q^{\binom{m+1}{2}+m(r-p)}\\
&\qquad  \times\Big (  j(-q^{m(2p+1)+p(2r+2)};q^{2p(2p+1)}) 
  -q^{m(2p+1)-m(2r+2)}j(-q^{-m(2p+1)+p(2r+2)};q^{2p(2p+1)})\Big )\\
&\qquad \times \Big (  -q^{p+m(s+1)} f_{2p+1,2p+1,1}(q^{p(2s+1)+3p+1},q^{m+p+1};q)\\
&\qquad \qquad +q^{p-m(s+1)} f_{2p+1,2p+1,1}(q^{p(2s+1)+3p+1},q^{-m+p+1};q)\Big ).
\end{align*}}%
\end{corollary}

For odd-spin, we will only obtain the analogs of Corollaries \ref{corollary:levelOneHalfEvenSpin},  \ref{corollary:levelOneThirdEvenSpin},  and \ref{corollary:levelTwoThirdsEvenSpin}.  We will not obtain any of analogs of Theorem \ref{theorem:fractionalLevelPlus12alt}.  Instead, we will obtain the mock theta function identities for the cases $s=0$ of the Hecke-type double-sum coefficients in Corollaries  \ref{corollary:levelOneHalfOddSpin},  \ref{corollary:levelOneThirdOddSpin},  and \ref{corollary:levelTwoThirdsOddSpin}.  These are the analogs of the cases $s=0$ of Theorems \ref{theorem:genEulerIdentityEvenSpin}, \ref{theorem:genEulerOneThirdFirstPairEvenSpin}, \ref{theorem:genEulerOneThirdSecondPairEvenSpin}, \ref{theorem:level23EvenSpinFirstQuad-sEven}, and \ref{theorem:level23EvenSpinSecondQuad-sEven}.

\subsection{On string functions of admissible-level $1/2$ and the function $\mu_{2}(q)$}

We begin with the odd-spin analog of Corollary \ref{corollary:levelOneHalfEvenSpin}.
 
\begin{corollary}\label{corollary:levelOneHalfOddSpin}  For $p=2$, $j=1$, $r\in\{0,1\}$ and $s\in\{0,1\}$ we have
\begin{align*}
(q)_{\infty}^{3}\delta_{2r+1,2s+1}
& =(q)_{\infty}^{3}j(q^{4+4s};q^{5})
\mathcal{C}_{2s+1,2r+1}^{(2,5)}(q)\\
&\qquad 
-q^{-r}  j(q^{2+2r};q^{5}) 
 \Big (  -q^{3+3s} f_{5,5,1}(q^{4s+9},q^{4};q)
 +q^{1+s} f_{5,5,1}(q^{4s+9},q^{2};q)\Big ).
\end{align*}
\end{corollary}

For the case $s=0$, the Hecke-type double-sum coefficient in Corollary \ref{corollary:levelOneHalfOddSpin} evaluates to the following:
\begin{theorem}\label{theorem:genEulerOneHalfOddSpin} It holds
\begin{equation}
 -q^{3} f_{5,5,1}(q^{9},q^{4};q)
 +q f_{5,5,1}(q^{9},q^{2};q)=j(q;q^5)\left ( 1-\frac{1}{2}\mu_2(q)\right ) -\frac{1}{2}\frac{J_{1}^3}{J_{2}J_{4}}j(-q;-q^5).
 \label{equation:genEulerOneHalfOddSpinMockThetaId}
\end{equation}
\end{theorem}

\subsection{On string functions of admissible-level $1/3$ and the functions $f_{3}(q)$ and $\omega_{3}(q)$}

Next we have the odd-spin analog of Corollary \ref{corollary:levelOneThirdEvenSpin}. 
\begin{corollary}\label{corollary:levelOneThirdOddSpin}  For $p=3$, $j=1$, $r\in\{ 0,1,2\}$ and $s\in\{ 0,1,2\}$ we have
\begin{align*}
(q)_{\infty}^{3}\delta_{2r+1,2s+1}
& =(q)_{\infty}^{3}j(q^{6s+6};q^{7})
\mathcal{C}_{2s+1,2r+1}^{(3,7)}(q)\\
&\qquad 
+q^{1-2r}
\frac{j(q^{9+2r};q^{14})j(q^{4+4r};q^{28})}{J_{28}}\\
&\qquad \times \Big (  -q^{4+4s} f_{7,7,1}(q^{6s+13},q^{5};q)
 +q^{2+2s} f_{7,7,1}(q^{6s+13},q^{3};q)\Big ) \\
 &\qquad 
-q^{-r}  \frac{j(q^{2+2r};q^{14})j(q^{18+4r};q^{28})}{J_{28}}\\
&\qquad \times \Big (  -q^{5+5s} f_{7,7,1}(q^{6s+13},q^{6};q)
 +q^{1+s} f_{7,7,1}(q^{6s+13},q^{2};q)\Big ).
\end{align*}%
\end{corollary}

For the cases $s=0$, the two Hecke-type double-sum coefficients found in Corollary \ref{corollary:levelOneThirdOddSpin} evaluate to the following mock theta conjecture-like identities:
\begin{theorem}  \label{theorem:genEulerOneThirdOddSpin} It holds
\begin{gather*}
-q^4f_{7,7,1}(q^{13},q^{5};q)+q^2f_{7,7,1}(q^{13},q^{3};q)
=j(q;q^7)\left ( 1-\frac{1}{2}f_3(q^2)\right) 
-\frac{1}{2}\frac{J_{1}^2}{J_{4}}\cdot j(-q;-q^7),\\
-q^5f_{7,7,1}(q^{13},q^{6};q)+qf_{7,7,1}(q^{13},q^{2};q)
=j(q;q^7)\left ( 1-q\omega_{3}(-q)\right ) 
-\frac{J_{1}^3J_{4}}{J_{2}^3}\cdot J_{12,28}.
\end{gather*}
\end{theorem}
\begin{remark}  Using the Hecke-type double-sum functional equation property (\ref{equation:fabc-fnq-1}), we see that the above two identities are equivalent to
{\allowdisplaybreaks \begin{gather}
q^4f_{7,7,1}(q^{13},q^{5};q)+f_{7,7,1}(q^{6},q^{2};q)
=j(q;q^7)\frac{1}{2}f_3(q^2)
+\frac{1}{2}\frac{J_{1}^2}{J_{4}}\cdot j(-q;-q^7),\\
q^5f_{7,7,1}(q^{13},q^{6};q)+f_{7,7,1}(q^{6},q;q)
=j(q;q^7)\cdot q\omega_{3}(-q) 
+\frac{J_{1}^3J_{4}}{J_{2}^3}\cdot J_{12,28}.
\end{gather}}%
All of our Hecke-type double-sum identities can be adjusted with the functional equation properties in (\ref{equation:H7eq1.14}), (\ref{equation:fabc-fnq-1}), and (\ref{equation:fabc-fnq-2}), but we choose to keep the double-sums in the form given by the generalized Euler identity.
\end{remark}

\subsection{On string functions of admissible-level $2/3$ and the functions $\psi_{3}(q)$ and $\chi_{3}(q)$}

Finally we have the odd-spin analog of Corollary \ref{corollary:levelTwoThirdsEvenSpin}. 

\begin{corollary}\label{corollary:levelTwoThirdsOddSpin} For $p=3$ and $j=2$, $r\in\{0,1,2\}$ and $s\in\{0,1,2,3\}$ we have
{\allowdisplaybreaks \begin{align*}
(q)_{\infty}^{3}\delta_{2r+1,2s+1}
& =(q)_{\infty}^{3}j(-q^{10+6s};q^{16})
\mathcal{C}_{2s+1,2r+1}^{(3,8)}(q)
-(q)_{\infty}^{3}j(-q^{18+6s};q^{16})
\mathcal{C}_{2s+3,2r+1}^{(3,8)}(q)\\
&\qquad + q^{1-2r}
\frac{j(q^{10+2r};q^{16})j(q^{4+4r};q^{32})}{J_{32}}\\,
&\qquad  \times \Big ( q^{8+4s}
 f_{4,4,1}(-q^{6s+26},-q^{13};q^4)
  - q^{5+2s}
 f_{4,4,1}(-q^{6s+26},-q^{9};q^4)\\
 &\qquad \qquad -q^{14+4s}
 f_{4,4,1}(-q^{6s+34},-q^{15};q^4)
  +q^{9+2s}
 f_{4,4,1}(-q^{6s+34},-q^{11};q^4)\Big ) \\
 &\qquad -q^{-r}
\frac{j(q^{2+2r};q^{16})j(q^{20+4r};q^{32})}{J_{32}}\\
&\qquad  \times \Big ( q^{9+5s}
 f_{4,4,1}(-q^{6s+26},-q^{15};q^4)
 - q^{3+s}
 f_{4,4,1}(-q^{6s+26},-q^{7};q^4)\\
 &\qquad \qquad  -q^{16+5s}
 f_{4,4,1}(-q^{6s+34},-q^{17};q^4)
   +q^{6+s}
 f_{4,4,1}(-q^{6s+34},-q^{9};q^4)\Big ) .
\end{align*}}%
\end{corollary}

For the cases $s=0$, the Hecke-type double-sum coefficients of Corollary \ref{corollary:levelTwoThirdsOddSpin}  give us mock theta conjecture-like identities involving two different third-order mock theta functions of Ramanujan.  They read
\begin{equation*}
\psi_3(q):=\sum_{n\ge 1}\frac{q^{n^2}}{(q;q^2)_n}, \  
\chi_3(q):=\sum_{n\ge 0}\frac{q^{n^2}(-q)_n}{(-q^3;q^3)_n}.
\end{equation*}

\begin{theorem} \label{theorem:genEulerTwoThirdsOddSpin} It holds
\begin{align}
 q^{8}& f_{4,4,1}(-q^{26},-q^{13};q^4)
  - q^{5}f_{4,4,1}(-q^{26},-q^{9};q^4)\label{equation:level23OddSpinFirstQuadMockTheta}\\
 &\qquad  -q^{14} f_{4,4,1}(-q^{34},-q^{15};q^4)
  +q^{9}f_{4,4,1}(-q^{34},-q^{11};q^4)\notag \\
 &= j(-q^{10};q^{16})\left (-\psi_{3}(-q) \right ) 
   -  j(-q^{2};q^{16})\left ( 1-q\left (1- \chi_{3}(q)\right )  \right )+3q^2\frac{J_{3}J_{12}^3}{J_{4}J_{6}^2}j(-q^{2};q^{16}),\notag
   \end{align}
\begin{align}
    q^{9}&
 f_{4,4,1}(-q^{26},-q^{15};q^4)
 - q^{3}
 f_{4,4,1}(-q^{26},-q^{7};q^4)\label{equation:level23OddSpinSecondQuadMockTheta}\\
 &\qquad   -q^{16}
 f_{4,4,1}(-q^{34},-q^{17};q^4)
   +q^{6}
 f_{4,4,1}(-q^{34},-q^{9};q^4)\notag\\
&=  j(-q^{10};q^{16})\left (1 - \chi_{3}(q)\right ) 
 -q j(-q^{2};q^{16})\left (q^{-2} + \psi_{3}(-q)\right )+3q\frac{J_{3}J_{12}^3}{J_{4}J_{6}^2}j(-q^{10};q^{16}).\notag
 \end{align}
\end{theorem}
\begin{remark}
Similar to the situation we had with the even-spin, $2/3$-level string functions, we can combine the simple quotient of theta functions with the term involving $\chi_3(q)$.
\end{remark}


 \section{Concluding remarks and an overview of the paper}
A natural question to ask is, ``Can one recover the results of \cite{BoMo2025}?''  We do lose the ``$z$'' from the Weyl--Kac formula and its polar-finite decomposition \cite[Theorem 2.3]{BoMo2025}.  For $1/2$-level, we easily recover  (\ref{equation:mockThetaConj2502r-2ndmu}), and for $1/3$-level one would need \cite{FG} to  recover Theorem (\ref{theorem:newMockThetaIdentitiespP37m0ell2r}).  Beyond that, it seems unlikely to recover the other results of \cite{BoMo2025, KM25}.  However, we do gain new a new type of polar-finite forms as well as new mock theta conjecture-like identities for Hecke-type double-sums, but there is some overlap with our new $1/2$-level polar-finite forms.

 \smallskip
 Using functional equation properties of Hecke-type double-sums, see (\ref{equation:H7eq1.14}), (\ref{equation:fabc-fnq-1}), (\ref{equation:fabc-fnq-2}), the two identities (\ref{equation:genEulerOneHalfEvenSpinMockThetaId}) and (\ref{equation:genEulerOneHalfOddSpinMockThetaId})  are easily seen to be equivalent to the following two identities from \cite{BoMo2026}:
\begin{gather*}
f_{5,5,1}(q^3,1;q)-q^2f_{5,5,1}(q^7,q^4;q)=\frac{1}{2}j(q^2;q^5)\mu_2(q)-\frac{1}{2}\frac{J_{1}^3}{J_{2}J_{4}}j(-q^3;-q^{5}),\\
f_{5,5,1}(q^4,q;q)-qf_{5,5,1}(q^6,q^3;q)=\frac{1}{2}j(q;q^5)\mu_2(q)+\frac{1}{2}\frac{J_{1}^3}{J_{2}J_{4}}j(-q;-q^{5}).
\end{gather*}
One could cobble together proofs for Theorems \ref{theorem:genEulerOneHalfEvenSpin} and \ref{theorem:genEulerOneHalfOddSpin} from the results in \cite{BoMo2026}, but we would like to keep our proofs in this paper to be uniform.
 
 \smallskip
We begin with the preliminaries.  In Section \ref{section:prelim} we review basic facts about theta functions, Hecke-type double-sums, Appell functions, and mock theta functions.  In Section \ref{section:FryeGarvan}, we list the theta function identities needed to prove the mock theta conjecture-like identities found in Sections \ref{section:resultsEvenSpin} and  \ref{section:resultsOddSpin}.  We have already proved the identities using Maple packages of Frye and Garvan \cite{FG}.

\smallskip
In Section \ref{section:quasiPeriodicOddSpin}, we prove the odd-spin quasi-periodic relation found in Theorem \ref{theorem:generalQuasiPeriodicityOddSpin}.   In Section \ref{section:genEulerTheorems}, we apply our quasi-periodicity relation for even-spin, respectively odd-spin, to the generalized Euler identity and prove Theorem  \ref{theorem:genEulerIdentityEvenSpin}, respectively Theorem \ref{theorem:genEulerIdentityOddSpin}.

\smallskip
 In Section \ref{section:genEulerCorollaries}, we specialize our new modified general Euler identity theorems for even and odd-spin to obtain the results for admissible level $1/2$ in Corollaries \ref{corollary:levelOneHalfEvenSpin} and \ref{corollary:levelOneHalfOddSpin}, admissible level $1/3$ in Corollaries \ref{corollary:levelOneThirdEvenSpin} and \ref{corollary:levelOneThirdOddSpin}, and admissible level $2/3$ in Corollaries \ref{corollary:levelTwoThirdsEvenSpin} and \ref{corollary:levelTwoThirdsOddSpin}.

\smallskip
In Section \ref{section:mockThetaTheoremsEvenSpin}, we prove the new mock theta function identities for even-spin, $1/2$-level string functions in Theorem \ref{theorem:genEulerOneHalfEvenSpin}, for $1/3$-level string functions in Theorems \ref{theorem:genEulerOneThirdFirstPairEvenSpin} and \ref{theorem:genEulerOneThirdSecondPairEvenSpin}, and for $2/3$-level string functions in Theorem \ref{theorem:level23EvenSpinFirstQuad-sEven}, \ref{theorem:level23EvenSpinFirstQuad-sOdd}, \ref{theorem:level23EvenSpinSecondQuad-sEven}, and \ref{theorem:level23EvenSpinSecondQuad-sOdd}.

\smallskip
In Section \ref{section:mockThetaTheoremsOddSpin}, we prove the new mock theta function identities for odd-spin, $1/2$-level string functions in Theorem \ref{theorem:genEulerOneHalfOddSpin}, for $1/3$-level string functions in Theorems \ref{theorem:genEulerOneThirdOddSpin}, and for $2/3$-level string functions in Theorem \ref{theorem:genEulerTwoThirdsOddSpin}.


 \section{Technical preliminaries}\label{section:prelim}
 In this section we recall some basic properties of theta functions, Appell functions, Hecke-type double-sums, and mock theta functions.  We recall the formulas that are necessary to convert from Hecke-type double-sums to Appell functions.  We close by giving the technical lemmas needed to convert the Appell function expressions to mock theta functions.
 
 \smallskip
We begin with theta functions.  We recall useful product rearrangements, that may be used without mention.  We have
\begin{subequations}
\begin{gather}
\overline{J}_{0,1}=2\overline{J}_{1,4}=\frac{2J_2^2}{J_1},  \overline{J}_{1,2}=\frac{J_2^5}{J_1^2J_4^2},   J_{1,2}=\frac{J_1^2}{J_2},   \overline{J}_{1,3}=\frac{J_2J_3^2}{J_1J_6}, \notag\\
J_{1,4}=\frac{J_1J_4}{J_2},   J_{1,6}=\frac{J_1J_6^2}{J_2J_3},   \overline{J}_{1,6}=\frac{J_2^2J_3J_{12}}{J_1J_4J_6}.\notag
\end{gather}
\end{subequations}
We list several useful theta function properties:
\begin{subequations}
{\allowdisplaybreaks \begin{gather}
j(q^n x;q)=(-1)^nq^{-\binom{n}{2}}x^{-n}j(x;q), \ \ n\in\mathbb{Z},\label{equation:j-elliptic}\\
j(x;q)=j(q/x;q)\label{equation:j-flip},\\
j(x;q)={J_1}j(x,qx,\dots,q^{n-1}x;q^n)/{J_n^n} \ \ {\text{if $n\ge 1$,}}\label{equation:1.10}\\
j(x;-q)={j(x;q^2)j(-qx;q^2)}/{J_{1,4}},\label{equation:1.11}\\
j(z;q)=\sum_{k=0}^{m-1}(-1)^k q^{\binom{k}{2}}z^k
j\big ((-1)^{m+1}q^{\binom{m}{2}+mk}z^m;q^{m^2}\big ),\label{equation:j-split}\\
j(x^n;q^n)={J_n}j(x,\zeta_nx,\dots,\zeta_n^{n-1}x;q^n)/{J_1^n} \ \ {\text{if $n\ge 1$.}}\label{equation:1.12}
\end{gather}}%
\end{subequations}
A useful specialization of (\ref{equation:j-split}) reads
\begin{equation}
j(z;q)=j(-qz^2;q^4)-zj(-q^3z^2;q^4).\label{equation:j-split-m2}
\end{equation}

We recall the quintuple product identity:
\begin{proposition}\cite[Theorem 1.0]{H1}   For generic $x,y,z\in \mathbb{C}^*$ 
 \begin{subequations}
{\allowdisplaybreaks \begin{gather}
j(qx^3;q^3)+xj(q^2x^3;q^3)=j(-x;q)j(qx^2;q^2)/J_2={J_1j(x^2;q)}/{j(x;q)}.\label{equation:H1Thm1.0}
\end{gather}}
\end{subequations}
\end{proposition}

We recall some functional equation properties for Hecke-type double-sums.
\begin{proposition} \cite[Proposition 6.2]{HM} \label{proposition:H7eq1.14}For $x,y\in\mathbb{C}^*$
\begin{equation}
f_{a,b,c}(x,y;q)=-\frac{q^{a+b+c}}{xy}f_{a,b,c}(q^{2a+b}/x,q^{2c+b}/y;q).\label{equation:H7eq1.14}
\end{equation}
\end{proposition}

\begin{proposition} \cite[Proposition 6.3]{HM}  \label{proposition:f-functionaleqn} For $x,y\in\mathbb{C}^*$ and $\ell, k \in \mathbb{Z}$
\begin{align}
f_{a,b,c}(x,y;q)&=(-x)^{\ell}(-y)^kq^{a\binom{\ell}{2}+b\ell k+c\binom{k}{2}}f_{a,b,c}(q^{a\ell+bk}x,q^{b\ell+ck}y;q) \notag\\
&\ \ \ \ +\sum_{m=0}^{\ell-1}(-x)^mq^{a\binom{m}{2}}j(q^{mb}y;q^c)+\sum_{m=0}^{k-1}(-y)^mq^{c\binom{m}{2}}j(q^{mb}x;q^a),\label{equation:Gen1}
\end{align}
where when $b<a$, we follow the usual convention:
\begin{equation}
\sum_{r=a}^{b} c_r:=-\sum_{r=b+1}^{a-1} c_r.\label{equation:sumconvention}
\end{equation}
\end{proposition}

\begin{corollary}\cite[Corollary 6.4]{HM} \label{corollary:fabc-funceqnspecial}We have two simple specializations:
\begin{align}
f_{a,b,c}(x,y;q) =&-yf_{a,b,c}(q^bx,q^cy;q)+j(x;q^a),\label{equation:fabc-fnq-1}\\
f_{a,b,c}(x,y;q) =&-xf_{a,b,c}(q^ax,q^by;q)+j(y;q^c).\label{equation:fabc-fnq-2}
\end{align}
\end{corollary}

We have the following the degenerate cases for Hecke-type double-sum coefficients found in Corollaries \ref{corollary:genEulerOneHalfEvenSpin} and \ref{corollary:genEulerOneThirdEvenSpin}
\begin{lemma} \label{lemma:degenerateDoubleSumCoeffs} We have
\begin{gather}
 -q^{8} f_{5,5,1}(q^{15},q^{4};q) +q^{3} f_{5,5,1}(q^{15},q^{2};q)=0,\label{equation:degenOneHalfEvenSpin}\\
-q^{15}
 f_{7,7,1}(q^{28},q^{5};q) +q^{8} f_{7,7,1}(q^{28},q^{3};q)=0\label{equation:degenOneThirdFirstPairEvenSpin},\\
-q^{18}
 f_{7,7,1}(q^{28},q^{6};q) +q^{4} f_{7,7,1}(q^{28},q^{2};q)=0.\label{equation:degenOneThirdSecondPairEvenSpin}
\end{gather}
\end{lemma}
\begin{proof} [Proof of Lemma \ref{lemma:degenerateDoubleSumCoeffs}] The three proofs are just repeated use of identities (\ref{equation:H7eq1.14}), (\ref{equation:fabc-fnq-1}), and (\ref{equation:fabc-fnq-2}).
\end{proof}
Expressions found in \cite{HM,MZ2023} allow one to express Hecke-type double-sums in terms of Appell functions.  We list some Appell function properties.   The Appell function properties can be found in both \cite{HM, Zw2}, but we will use the notation found in \cite{HM}.  We have
\begin{proposition} \cite[Proposition 3.1]{HM} For generic $x,z\in \mathbb{C}^*$
{\allowdisplaybreaks \begin{subequations}
{\allowdisplaybreaks \begin{gather}
m(x,z;q)=m(x,qz;q),\label{equation:mxqz-fnq-z}\\
m(x,z;q)=x^{-1}m(x^{-1},z^{-1};q),\label{equation:mxqz-flip}\\
m(qx,z;q)=1-xm(x,z;q).\label{equation:mxqz-fnq-x}
\end{gather}}%
\end{subequations}}
\end{proposition}
\noindent We also have
\begin{proposition}\cite[Theorem 3.3]{HM} \label{proposition:changing-z-theorem}  For generic $x,z_0,z_1\in \mathbb{C}^*$
\begin{equation}
m(x,z_1;q)-m(x,z_0;q)=\frac{z_0J_1^3j(z_1/z_0;q)j(xz_0z_1;q)}{j(z_0;q)j(z_1;q)j(xz_0;q)j(xz_1;q)}.
\end{equation}
\end{proposition}

We recall an Appell function property unique to work of \cite{HM}.  We take the specialization $n=2$ of Theorem $3.5$ \cite{HM}.  This gives
\begin{corollary} \label{corollary:msplit-n2} For generic $x,z,z'\in \mathbb{C}^*$ 
{\allowdisplaybreaks \begin{align*}
m(&x,z;q) =  m\big({-}q x^2, -1;q^{4} \big) - q^{-1} x m\big({-}q^{-1} x^2, -1;q^{4} \big)\\
& - \frac{J_2^3}{j(xz;q) j(-1;q^{4})} \left [  
\frac{ j\big(q x^2 z;q^2\big) j(-z^2;q^{4})}
{ j\big(q x^2;q^{2}\big ) j\big ( z;q^2\big )}
-\frac{ xz j\big(q^{2} x^2 z;q^2\big)j(-q^{2} z^2;q^{4})}
{ j\big(q x^2;q^{2}\big ) j\big ( q z;q^2\big )}
\right ].
\end{align*}}%
\end{corollary}

In order to evaluate the Hecke-type double-sums found in the coefficients, we will use specializations of \cite[Theorem 1.4]{HM}.  For our purposes, we first specialize \cite[Theorem 1.4]{HM} to $a=b=n$, $c=1$.  
\begin{theorem} \label{theorem:main-acdivb} Let $n$ be a positive integer.  Then
\begin{align*}
& f_{n,n,1}(x,y;q)=h_{n,n,1}(x,y;q)-\frac{1}{\overline{J}_{0,n-1}\overline{J}_{0,n(n-1)}}\cdot \theta_{n,n,1}(x,y;q),
\end{align*}
where 
\begin{align*}
h_{n,n,1}(x,y;q)&:=j(x;q^n)m\big({ -}q^{n-1}yx^{-1},-1;q^{n-1} \big )\\
&\qquad  +j(y;q)m\big( q^{\binom{n}{2}}x(-y)^{-n},-1;q^{n^2-n} \big ),
\end{align*}
and
\begin{align*}
\theta_{n,n,1}(x,y;q):=&\sum_{d=0}^{n-1}
q^{(n-1)\binom{d+1}{2}}j\big (q^{(n-1)(d+1)}y;q^{n}\big )  j\big ({-}q^{n(n-1)-(n-1)(d+1)}xy^{-1};q^{n(n-1)}\big ) \\
& \ \ \ \ \ \cdot \frac{J_{n(n-1)}^3j\big (q^{ \binom{n}{2}+(n-1)(d+1)}(-y)^{1-n};q^{n(n-1)}\big )}
{j\big ({-}q^{\binom{n}{2}}x(-y)^{-n};q^{n(n-1)})j(q^{(n-1)(d+1)}x^{-1}y;q^{n(n-1)}\big )}.
\end{align*}
\end{theorem}

We now list additional specializations for $n=4,5,7$.  For the level $2/3$ mock theta function identities we will use the following specialization for $n=4$. 
\begin{corollary} \label{corollary:f441-HeckeExpansion} We have
\begin{align*}
& f_{4,4,1}(x,y;q)=h_{4,4,1}(x,y;q)-\frac{1}{\overline{J}_{0,3}\overline{J}_{0,12}}\cdot \theta_{4,4,1}(x,y;q),
\end{align*}
where 
\begin{align*}
h_{4,4,1}(x,y;q)&:=j(x;q^4)m\big( {-}q^{3}yx^{-1},-1;q^{3} \big )
+j(y;q)m\big( q^{6}xy^{-4},-1;q^{12} \big ),
\end{align*}
and
\begin{align*}
\theta_{4,4,1}(x,y;q):=&\sum_{d=0}^{3}
q^{3\binom{d+1}{2}}j\big (q^{3(d+1)}y;q^{4}\big )  j\big ({-}q^{12-3(d+1)}xy^{-1};q^{12}\big )\\
&\qquad  \cdot \frac{J_{12}^3j\big ({-}q^{ 6+3(d+1)}y^{-3};q^{12}\big )}
{j\big ({-}q^{6}xy^{-4};q^{12})j(q^{3(d+1)}x^{-1}y;q^{12}\big )}.
\end{align*}
\end{corollary}

For the level $1/2$ mock theta function identities we will use the following specialization for $n=5$. 
\begin{corollary} \label{corollary:f551-HeckeExpansion} We have
\begin{equation*}
f_{5,5,1}(x,y;q)=h_{5,5,1}(x,y;q)-\frac{1}{\overline{J}_{0,4}\overline{J}_{0,20}}\theta_{5,5,1}(x,y;q),
\end{equation*}
where
\begin{equation*}
h_{5,5,1}(x,y;q)
:=j(x;q^5)m(-q^4x^{-1}y,-1;q^4)+j(y;q)m(-q^{10}xy^{-5},-1;q^{20}),
\end{equation*}
and
\begin{equation*}
\theta_{5,5,1}(x,y;q):= \sum_{d=0}^{4}
q^{2d(d+1)}\frac{J_{20}^3j\big (q^{4+4d}y;q^{5}\big )  j\big ({-}q^{16-4d}xy^{-1};q^{20}\big )j\big (q^{14+4d}y^{-4};q^{20}\big )}
{j\big (q^{10}xy^{-5};q^{20})j(q^{4+4d}x^{-1}y;q^{20}\big )}.
\end{equation*}
\end{corollary}

For the level $1/3$ mock theta function identities we will use the following specialization for $n=7$. 
\begin{corollary}  \label{corollary:f771-HeckeExpansion} We have
\begin{align*}
& f_{7,7,1}(x,y,q)=h_{7,7,1}(x,y;q)-\frac{1}{\overline{J}_{0,6}\overline{J}_{0,42}}\cdot \theta_{7,7,1}(x,y;q),
\end{align*}
where 
\begin{align*}
h_{7,7,1}(x,y;q):&=j(x;q^7)m\big({ -}q^{6}yx^{-1},-1;q^{6} \big )
 +j(y;q)m\big({ -}q^{21}xy^{-7},-1;q^{42} \big ),
\end{align*}
and
\begin{align*}
\theta_{7,7,1}(x,y;q):=&\sum_{d=0}^{6}
q^{6\binom{d+1}{2}}j\big (q^{6(d+1)}y;q^{7}\big )  j\big ({-}q^{42-6(d+1)}xy^{-1};q^{42}\big ) \\
& \ \ \ \ \ \cdot \frac{J_{42}^3j\big (q^{ 21+6(d+1)}y^{-6};q^{42}\big )}
{j\big (q^{21}xy^{-7};q^{42}\big )j\big (q^{6(d+1)}x^{-1}y;q^{42}\big )}.
\end{align*}
\end{corollary}

Once we have rewritten the double-sums using Corollaries \ref{corollary:f441-HeckeExpansion}, \ref{corollary:f551-HeckeExpansion}, and \ref{corollary:f771-HeckeExpansion}, we will need to rewrite the expression involving Appell functions in terms of mock theta functions.  The first step is to get to within a theta function of the mock theta function.  For simple results for odd-spin this step will be short, but for the general results for even-spin we will need the following:
\begin{lemma}\label{lemma:generalAppellSums}
We have
\begin{align}
&m ( -q^{1+4s},-1;q^{4}  )  -q^{2s-1}m ( -q^{-1+4s},-1;q^{4}  )
\label{equation:generalAppellSumOneModFour}\\
&\quad =\sum_{k=0}^{2s-1}(-1)^{k}q^{2sk-\binom{k+1}{2}}
 +2q^{s+4\binom{s}{2}}m ( -q,-1;q^{4}  ),\notag\\
&m ( -q^{1+6s},-1;q^{6}  )  -q^{2s-1}m ( -q^{-1+6s},-1;q^{6}  )
\label{equation:generalAppellSumOneModSix}\\
&\quad =\sum_{k=0}^{s-1}q^{6sk-6\binom{k+1}{2}}\left (q^{k} -q^{-k+2s-1}\right ) 
 +2q^{s+6\binom{s}{2}}m ( -q,-1;q^{6}  ),\notag\\
& m ( -q^{2+6s},-1;q^{6} ) - q^{4s-2}m ( -q^{-2+6s},-1;q^{6}  )
\label{equation:generalAppellSumTwoModSix}\\
&\quad =\sum_{k=0}^{s-1}q^{6sk-6\binom{k+1}{2}}\left (q^{2k} -q^{-2k+4s-2}\right ) 
 +2q^{2s+6\binom{s}{2}}m ( -q^{2},-1;q^{6}  ).\notag
\end{align}
\end{lemma}
\begin{proof}[Proof of Lemma \ref{lemma:generalAppellSums}]  This is a straightforward exercise in induction.
\end{proof}

The mock theta functions that we will be using are all from Ramanujan's classical list.  We use the forms as found in \cite[Section 5]{HM}.  We have both second-order and third-order functions.  For the second order functions, we will only use a single mock theta function:
\begin{equation}
\mu_2(q):=\sum_{n\ge 0}\frac{(-1)^nq^{n^2}(q;q^2)_n}{(-q^2;q^2)_{n}^2}
=4m(-q,-1;q^4)-\frac{J_{2,4}^4}{J_1^3}.\label{equation:2nd-mu(q)}
\end{equation}
For the third order-functions, we will have four mock theta functions:
{\allowdisplaybreaks \begin{align}
f_3(q)&:=\sum_{n\ge 0}\frac{q^{n^2}}{(-q)_n^2}
=4m(-q,q;q^3)+\frac{J_{3,6}^2}{J_1},\label{equation:3rd-f(q)}\\
 \omega_3(q)&:=\sum_{n\ge 0}\frac{q^{2n(n+1)}}{(q;q^2)_{n+1}^2}
 =-2q^{-1}m(q,q^2;q^6)+\frac{J_6^3}{J_2 J_{3,6}},\label{equation:3rd-omega(q)}\\
\psi_3(q)&:=\sum_{n\ge 1}\frac{q^{n^2}}{(q;q^2)_n}
=-m(q,-q;-q^3)+\frac{qJ_{12}^3}{J_4 J_{3,12}},\label{equation:3rd-psi(q)}\\
\chi_3(q)&:=\sum_{n\ge 0}\frac{q^{n^2}(-q)_n}{(-q^3;q^3)_n}=m(-q,q;q^3)+\frac{J_{3,6}^2}{J_1}.\label{equation:3rd-chi(q)}
\end{align}}%

We end this section with the lemmas that allow us to convert from the Appell function expressions to the mock theta functions.
\begin{lemma}\label{lemma:alternateAppellForms}  We have
\begin{align} 
m(-q,-1;q^{6})
&= \frac{1}{2}q\omega_{3}(-q)-\frac{1}{2}q\frac{J_{6}^3}{J_{2}\overline{J}_{3,6}}
+\frac{J_{6}^3\overline{J}_{2,6}J_{3,6}}{\overline{J}_{0,6}J_{1,6}J_{2,6}\overline{J}_{3,6}},
\label{equation:alternateAppellForm3rd-w}
\\
m(-q^2,-1;q^{6})
&=\frac{1}{4}f_3(q^2)-\frac{1}{4}\frac{J_{6,12}^2}{J_{2}}
+\frac{J_{6}^3}{\overline{J}_{0,6}J_{2,6}}.
\label{equation:alternateAppellForm3rd-f}
\end{align}
\end{lemma}
\begin{proof}[Proof of Lemma \ref{lemma:alternateAppellForms}]  This follows from the Appell function forms (\ref{equation:3rd-omega(q)}) and (\ref{equation:3rd-f(q)}) and Proposition \ref{proposition:changing-z-theorem}.
\end{proof}

\begin{lemma} \label{lemma:alternateAppellFormsChiPsi} We have
{\allowdisplaybreaks \begin{align}
m(-q^5, -1;q^{12}) - q^{-1} m(-q, -1;q^{12})
&=\chi_3(q)-\frac{J_{3,6}^2}{J_{1}}-\frac{q^{-1}}{2} \frac{J_{3}J_{6}J_{8}J_{12}^3}{J_{2}J_{4}J_{24}^3},\\
m(-q^5, -1;q^{12}) - q^{-1} m(-q, -1;q^{12})
&=-\psi_3(-q)-q\frac{J_{12}^3}{\overline{J}_{3,12}J_{4}}-\frac{q^{-1}}{2} \frac{J_{3}J_{6}J_{8}J_{12}^3}{J_{2}J_{4}J_{24}^3}.
\end{align}}%
\end{lemma}
\begin{remark}
There is not a mistake in the above lemma.  It is indeed the case that the left-hand sides of both identities are the same!
\end{remark}
\begin{proof}[Proof of Lemma \ref{lemma:alternateAppellFormsChiPsi}] 
Using Corollary \ref{corollary:msplit-n2}, we have 
{\allowdisplaybreaks \begin{align*}
m(-q,q;q^3) &=  m (-q^5, -1;q^{12} ) + q^{-2} m (-q^{-1}, -1;q^{12} )\\
&\qquad  - \frac{J_6^3}{j(-q^2;q^3) j(-1;q^{12})} 
\frac{ q^2
j\big(q^{9};q^6\big)
j(-q^{8};q^{12})}
{ j\big(q^5;q^{6}\big ) j\big ( q^4 ;q^6\big )}.
\end{align*}}%
Using (\ref{equation:mxqz-flip}) and (\ref{equation:j-elliptic}) gives
{\allowdisplaybreaks \begin{align*}
m(-q,q;q^3) &=  m(-q^5, -1;q^{12}) - q^{-1} m(-q, -1;q^{12})
  +q^{-1} \frac{J_6^3}{\overline{J}_{1,3} \overline{J}_{0,12}} 
\frac{J_{3,6}\overline{J}_{4,12}}
{ J_{1,6} J_{2}}.
\end{align*}}%
Elementary product rearrangements give
{\allowdisplaybreaks \begin{align*}
m(-q,q;q^3) &=  m(-q^5, -1;q^{12}) - q^{-1} m(-q, -1;q^{12})
  +\frac{q^{-1}}{2} \frac{J_{3}J_{6}J_{8}J_{12}^3}{J_{2}J_{4}J_{24}^3}.
\end{align*}}%
The results then follow from identities (\ref{equation:3rd-psi(q)}) and (\ref{equation:3rd-chi(q)})
\end{proof}

\section{Families of theta function identities through Frye and Garvan}\label{section:FryeGarvan}
In this section, we list the identities that we proved by using Frye and Garvan's Maple packages {\em qseries} and {\em thetaids}  \cite{FG}.  We will not sketch the proofs here, but the interested reader can find sketches of how the two Maple packages prove identities in \cite[Section 4]{BoMo2026}, \cite[Section 5]{BoMo2025}.

\smallskip
For even-spin, $1/2$-level, we will need the following identity to prove Theorem \ref{theorem:genEulerOneHalfEvenSpin}.  We recall the notation for $\theta_{5,5,1}(x,y;q)$ from Corollary \ref{corollary:f551-HeckeExpansion}.
 \begin{proposition} \label{proposition:level12EvenSpinThetaId} For $s\in\{0,1,3,4\}$ we have that
 \begin{align*}
 (-1)^{s}&q^{\binom{s+1}{2}}j(q^{s+3};q^5)\frac{1}{2}\frac{J_{2,4}^4}{J_{1}^3}\\
&\qquad -\frac{1}{\overline{J}_{0,4}\overline{J}_{0,20}}
\left ( -q^{2+3s} \theta_{5,5,1}(q^{4s+7},q^{4};q) +q^{1+s} \theta_{5,5,1}(q^{4s+7},q^{2};q)\right )\\
 &= -(-1)^s(-q)^{\binom{s+1}{2}}\frac{1}{2}\frac{J_{1}^3}{J_{2}J_{4}}j((-q)^{s+3};-q^5).
 \end{align*}
 \end{proposition}
\noindent For the case $s=2$ in Proposition \ref{proposition:level12EvenSpinThetaId}, the Hecke-type double-sum coefficient in Corollary  \ref{corollary:genEulerOneHalfEvenSpin} evaluates to zero, see Lemma \ref{lemma:degenerateDoubleSumCoeffs}.

For even-spin, $1/3$-level, we will need the following two identities to prove Theorems \ref{theorem:genEulerOneThirdFirstPairEvenSpin} and \ref{theorem:genEulerOneThirdSecondPairEvenSpin}, respectively.  We recall the notation for $\theta_{7,7,1}(x,y;q)$ from Corollary \ref{corollary:f771-HeckeExpansion}.
\begin{proposition}\label{proposition:level13EvenSpinFirstPairThetaId} For $s\in\{0,1,2,4,5,6\}$ we have that
{\allowdisplaybreaks \begin{align*}
(-1)^s&2q^{\binom{s+1}{2}}j(q^{s+4};q^7)\left ( -\frac{1}{2}q\frac{J_{6}^3}{J_{2}\overline{J}_{3,6}}
+\frac{J_{6}^3\overline{J}_{2,6}J_{3,6}}{\overline{J}_{0,6}J_{1,6}J_{2,6}\overline{J}_{3,6}}\right )\\
&\qquad -\frac{1}{\overline{J}_{0,6}\overline{J}_{0,42}}
\left ( -q^{3+4s}\theta_{7,7,1}(q^{6s+10},q^5;q) +q^{2+2s}\theta_{7,7,1}(q^{6s+10},q^3;q)\right ) \\
&=-q^{5s^2-3s+1}\frac{J_{1}^3J_{4}}{J_{2}^3}j(q^{8+16s};q^{28}).
\end{align*}}%
\end{proposition}

\begin{proposition}\label{proposition:level13EvenSpinSecondPairThetaId} For $s\in\{0,1,2,4,5,6\}$ we have that
{\allowdisplaybreaks \begin{align*}
(-1)^{s}&2q^{\binom{s+1}{2}}j(q^{s+4};q^7)\left ( -\frac{1}{4}\frac{J_{6,12}^2}{J_{2}}
+\frac{J_{6}^3}{\overline{J}_{0,6}J_{2,6}}\right ) \\
&\qquad -\frac{1}{\overline{J}_{0,6}\overline{J}_{0,42}}
\left ( -q^{3+5s}\theta_{7,7,1}(q^{6s+10},q^6;q) +q^{1+s}\theta_{7,7,1}(q^{6s+10},q^2;q)\right ) \\
&=-(-1)^{s}\frac{(-q)^{\binom{s+1}{2}}}{2}\frac{J_{1}^2}{J_{4}}j((-q)^{s+4};-q^7).
\end{align*}}%
\end{proposition}
\noindent For the cases $s=3$ in Propositions \ref{proposition:level13EvenSpinFirstPairThetaId} and \ref{proposition:level13EvenSpinSecondPairThetaId}, we see that the double-sum coefficients in Corollary \ref{corollary:genEulerOneThirdEvenSpin} evaluate to zero, see Lemma \ref{lemma:degenerateDoubleSumCoeffs}.

For even-spin, $2/3$-level, we will need the following two identities to prove Theorems \ref{theorem:level23EvenSpinFirstQuad-sEven}  and \ref{theorem:level23EvenSpinFirstQuad-sOdd}, respectively.   We recall the notation for $\theta_{4,4,1}(x,y;q)$ from Corollary \ref{corollary:f441-HeckeExpansion}.
\begin{proposition}\label{proposition:level23EvenSpinFirstQuad-sEvenThetaId} For $s\in\{0,1,2,3\}$ we have
{\allowdisplaybreaks \begin{align*}
&2q^{7+16s}  q^{2s+12\binom{s}{2}} j(-q^{12s+23};q^{16})\left (  -\frac{1}{2}q^2\frac{J_{12}^3}{J_{4}\overline{J}_{6,12}}
+\frac{J_{12}^3\overline{J}_{4,12}J_{6,12}}{\overline{J}_{0,12}J_{2,12}J_{4,12}\overline{J}_{6,12}}\right )\\
&\quad +2q^{17+16s}q^{4s}q^{4s+12\binom{s}{2}} j(-q^{31+12s};q^{16}) 
 \left ( -\frac{1}{4}\frac{J_{12,24}^2}{J_{4}}
+\frac{J_{12}^3}{\overline{J}_{0,12}J_{4,12}}\right )\\
&\quad-\frac{1}{\overline{J}_{0,12}\overline{J}_{0,48}}
\Big (q^{7+8s}\theta_{4,4,1}(-q^{12s+23},-q^{13};q^4)
   - q^{5+4s}\theta_{4,4,1}(-q^{12s+23},-q^{9};q^4)\\
& \qquad \qquad -q^{13+8s}\theta_{4,4,1} (-q^{31+12s},-q^{15};q^4)
  + q^{9+4s} \theta_{4,4,1} (-q^{31+12s},-q^{11};q^4)\Big ) \\ 
&=  -(-1)^{s}\frac{q^{\binom{2s}{2}+2}}{2}\frac{J_{1}^2J_{2}}{J_{4}^2}j(-q^{1+4s};q^{16}).
\end{align*}}%
\end{proposition}

\begin{proposition}\label{proposition:level23EvenSpinFirstQuad-sOddThetaId} For $s\in\{0,1,2,3\}$ we have
{\allowdisplaybreaks \begin{align*}
&-2q^{4s}q^{4s+12\binom{s}{2}}q^{15+16s}j(-q^{29+12s};q^{16})\left (-\frac{1}{4}\frac{J_{12,24}^2}{J_{4}}
+\frac{J_{12}^3}{\overline{J}_{0,12}J_{4,12}}\right ) \\
&\quad -2q^{25+16s}q^{4s+2}q^{8s} q^{2s+12\binom{s}{2}}j(-q^{37+12s};q^{16})
\left ( -\frac{1}{2}q^2\frac{J_{12}^3}{J_{4}\overline{J}_{6,12}}
+\frac{J_{12}^3\overline{J}_{4,12}J_{6,12}}{\overline{J}_{0,12}J_{2,12}J_{4,12}\overline{J}_{6,12}}\right )\\
&\quad-\frac{1}{\overline{J}_{0,12}\overline{J}_{0,48}}
\Big (q^{11+8s}\theta_{4,4,1}(-q^{29+12s},-q^{13};q^4)-q^{7+4s}\theta_{4,4,1}(-q^{29+12s},-q^9;q^4)\\
&\qquad \qquad -q^{17+8s}\theta_{4,4,1}(-q^{37+12s},-q^{15};q^4)+q^{11+4s}\theta_{4,4,1}(-q^{37+12s},-q^{11};q^{4})\Big ) \\ 
&\qquad = -(-1)^{s}\frac{q^{\binom{2s+1}{2}+2}}{2}\frac{J_{1}^2J_{2}}{J_{4}^2}j(-q^{3+4s};q^{16}).
\end{align*}}%
\end{proposition}

For even-spin, $2/3$-level, we will need the following two identities to prove Theorems \ref{theorem:level23EvenSpinSecondQuad-sEven}  and \ref{theorem:level23EvenSpinSecondQuad-sOdd}, respectively.   We recall the notation for $\theta_{4,4,1}(x,y;q)$ from Corollary \ref{corollary:f441-HeckeExpansion}.
\begin{proposition}\label{proposition:level23EvenSpinSecondQuad-sEvenThetaId} For $s\in\{0,1,2,3\}$ we have
{\allowdisplaybreaks \begin{align*}
&2q^{7+14s} q^{4s+12\binom{s}{2}} j(-q^{12s+23};q^{16})
 \left (-\frac{1}{4}\frac{J_{12,24}^2}{J_{4}}
+\frac{J_{12}^3}{\overline{J}_{0,12}J_{4,12}}\right ) \\
&\qquad +2q^{8s}q^{16+14s} q^{2s+12\binom{s}{2}} j(-q^{31+12s};q^{16})\left ( -\frac{1}{2}q^2\frac{J_{12}^3}{J_{4}\overline{J}_{6,12}}
+\frac{J_{12}^3\overline{J}_{4,12}J_{6,12}}{\overline{J}_{0,12}J_{2,12}J_{4,12}\overline{J}_{6,12}} \right )\\
&\qquad -\frac{1}{\overline{J}_{12}\overline{J}_{0,48}}\Big ( q^{7+10s}
 \theta_{4,4,1}(-q^{12s+23},-q^{15};q^4)
   - q^{3+2s} \theta_{4,4,1}(-q^{12s+23},-q^{7};q^4)\\
&\qquad \qquad  -q^{14+10s}
 \theta_{4,4,1}(-q^{31+12s},-q^{17};q^4)
  + q^{6+2s} \theta_{4,4,1}(-q^{31+12s},-q^{9};q^4) \Big ) \\
 &\quad = -(-1)^{s}\frac{q^{\binom{2s+1}{2}}}{2}\frac{J_{1}^2J_{2}}{J_{4}^2}j(-q^{9+4s};q^{16}).
\end{align*}}%
\end{proposition}

\begin{proposition}\label{proposition:level23EvenSpinSecondQuad-sOddThetaId} For $s\in\{0,1,2,3\}$ we have
{\allowdisplaybreaks 
\begin{align*}
-2&q^{8s}q^{2s+12\binom{s}{2}} q^{14+14s}  j(-q^{12s+29};q^{16})\left ( -\frac{1}{2}q^2\frac{J_{12}^3}{J_{4}\overline{J}_{6,12}}
+\frac{J_{12}^3\overline{J}_{4,12}J_{6,12}}{\overline{J}_{0,12}J_{2,12}J_{4,12}\overline{J}_{6,12}}   \right ) \\
&\quad - 2q^{23+14s}  q^{8s+4} q^{4s}q^{4s+12\binom{s}{2}}j(-q^{37+12s};q^{16})\left (-\frac{1}{4}\frac{J_{12,24}^2}{J_{4}}
+\frac{J_{12}^3}{\overline{J}_{0,12}J_{4,12}}\right )\\
&\quad -\frac{1}{\overline{J}_{12}\overline{J}_{0,48}}
\Big ( q^{12+10s}
 \theta_{4,4,1}(-q^{12s+29},-q^{15};q^4)
   - q^{4+2s} \theta_{4,4,1}(-q^{12s+29},-q^{7};q^4)\\
&\qquad  -q^{19+10s}
 \theta_{4,4,1}(-q^{37+12s},-q^{17};q^4)
  + q^{7+2s} \theta_{4,4,1}(-q^{37+12s},-q^{9};q^4) \Big ) \\
  &= - (-1)^{s}\frac{q^{\binom{2s+2}{2}}}{2}\frac{J_{1}^2J_{2}}{J_{4}^2}j(-q^{11+4s};q^{16}).
\end{align*}}%
\end{proposition}


For odd-spin, $1/2$-level, we will need the following identity to prove Theorem \ref{theorem:genEulerOneHalfOddSpin}.  We recall the notation for $\theta_{5,5,1}(x,y;q)$ from Corollary \ref{corollary:f551-HeckeExpansion}.
\begin{proposition} \label{proposition:level12OddSpinThetaId} We have that
\begin{align*}
-\frac{1}{2}&j(q;q^5)\frac{J_{2,4}^4}{J_{1}^3}
 -\frac{1}{\overline{J}_{0,4}\overline{J}_{0,20}}
\left ( -q^{3}\theta_{5,5,1}(q^9,q^4;q)+q\theta_{5,5,1}(q^9,q^2;q)\right )\\
&=-\frac{1}{2}\frac{J_{1}^3}{J_{2}J_{4}}j(-q;-q^5).
\end{align*}
\end{proposition}

For odd-spin, $1/3$-level, we will need the following two identities to prove Theorem \ref{theorem:genEulerOneThirdOddSpin}.  We recall the notation for $\theta_{7,7,1}(x,y;q)$ from Corollary \ref{corollary:f771-HeckeExpansion}.
\begin{proposition}\label{proposition:level13OddSpinFirstPairThetaId}  We have that
\begin{align*}
\frac{1}{2}&\frac{J_{6,12}^2J_{1,7}}{J_{2}}
-2\frac{J_{6}^3J_{1,7}}{\overline{J}_{0,6}J_{2,6}}
 -\frac{1}{\overline{J}_{0,6}\overline{J}_{0,42}}\left ( -q^4\theta_{7,7,1}(q^{13},q^5;q) +q^2\theta_{7,7,1}(q^{13},q^3;q)\right ) \\
 &=-\frac{1}{2}\frac{J_{1}^2}{J_{4}}\cdot j(-q;-q^7).
\end{align*}
\end{proposition}

\begin{proposition}\label{proposition:level13OddSpinSecondPairThetaId}  We have that
\begin{align*}
q&\frac{J_{6}^3J_{1,7}}{J_{2}\overline{J}_{3,6}}
-2\frac{J_{6}^3\overline{J}_{2,6}J_{3,6}J_{1,7}}{\overline{J}_{0,6}J_{1,6}J_{2,6}\overline{J}_{3,6}} 
 -\frac{1}{\overline{J}_{0,6}\overline{J}_{0,42}}\left ( -q^5\theta_{7,7,1}(q^{13},q^6;q) +q\theta_{7,7,1}(q^{13},q^2;q)\right )\\
 &=-\frac{J_{1}^3J_{4}}{J_{2}^3}\cdot J_{12,28}.
\end{align*}
\end{proposition}

For odd-spin, $2/3$-level, we will need the following two identities to prove Theorem \ref{theorem:genEulerTwoThirdsOddSpin}.  We recall the notation for $\theta_{4,4,1}(x,y;q)$ from Corollary \ref{corollary:f441-HeckeExpansion}.
\begin{proposition} \label{proposition:level23OddSpinFirstQuadThetaId}  We have
\begin{align*}
 -& j(-q^{10};q^{16}) \left ( q\frac{J_{12}^3}{\overline{J}_{3,12}J_{4}}+\frac{q^{-1}}{2} \frac{J_{3}J_{6}J_{8}J_{12}^3}{J_{2}J_{4}J_{24}^3} \right ) 
+  j(-q^{2};q^{16})\left (\frac{J_{3,6}^2}{J_{1}}+\frac{q^{-1}}{2} \frac{J_{3}J_{6}J_{8}J_{12}^3}{J_{2}J_{4}J_{24}^3} \right )\\
 &\qquad -\frac{1}{\overline{J}_{0,3}\overline{J}_{0,12}}
 \Big ( q^{8} \theta_{4,4,1}(-q^{26},-q^{13};q^4)
  - q^{5}\theta_{4,4,1}(-q^{26},-q^{9};q^4)\\
 &\qquad \qquad -q^{14} \theta_{4,4,1}(-q^{34},-q^{15};q^4)
  +q^{9}\theta_{4,4,1}(-q^{34},-q^{11};q^4)\Big)\\
&=3q^2\frac{J_{3}J_{12}^3}{J_{4}J_{6}^2}j(-q^{2};q^{16}).
\end{align*}
\end{proposition}

\begin{proposition} \label{proposition:level23OddSpinSecondQuadThetaId}  We have
\begin{align*}
 j&(-q^{10};q^{16})\left ( \frac{J_{3,6}^2}{J_{1}}+\frac{q^{-1}}{2} \frac{J_{3}J_{6}J_{8}J_{12}^3}{J_{2}J_{4}J_{24}^3}\right ) 
-q j(-q^{2};q^{16})\left ( q\frac{J_{12}^3}{\overline{J}_{3,12}J_{4}}+\frac{q^{-1}}{2} \frac{J_{3}J_{6}J_{8}J_{12}^3}{J_{2}J_{4}J_{24}^3} \right )\\
&\qquad -\frac{1}{\overline{J}_{0,3}\overline{J}_{0,12}}\Big ( 
q^{9} \theta_{4,4,1}(-q^{26},-q^{15};q^4)
 - q^{3}\theta_{4,4,1}(-q^{26},-q^{7};q^4)\\
 &\qquad \qquad  -q^{16}\theta_{4,4,1}(-q^{34},-q^{17};q^4)
   +q^{6} \theta_{4,4,1}(-q^{34},-q^{9};q^4)\Big )\\
&=3q\frac{J_{3}J_{12}^3}{J_{4}J_{6}^2}j(-q^{10};q^{16}).
\end{align*}
\end{proposition}
\section{The quasi-periodic relations for odd-spin}
\label{section:quasiPeriodicOddSpin}

We prove Theorem \ref{theorem:generalQuasiPeriodicityOddSpin}, but we will do so in a series of steps.  We first recall the Hecke-type double-sum form for our string function.  Let $p'\geq 2$, $p\geq 1$ be co-prime integers, $0\leq \ell \leq p'-2$ and $m\in 2\Z +\ell$. We specialize  (\ref{equation:modStringFnHeckeForm}) with $(m,\ell)=(2k+1,2r+1)$, $(p,p^{\prime})=(p,2p+j)$.  This gives
\begin{align}
(q)_{\infty}^{3}\mathcal{C}_{2k+1,2r+1}^{(p,2p+j)}(q)
 &= f_{1,2p+j,2p(2p+j)}(q^{2+k+r},-q^{p(2p+j+2r+2)};q)
 \label{equation:stringQuasiHecke}\\
&\qquad  -f_{1,2p+j,2p(2p+j)}(q^{k-r},-q^{p(2p+j-2r-2)};q).
\notag
\end{align}  

Using the functional equation Proposition \ref{proposition:f-functionaleqn} for Hecke-type double-sums, we are able to establish the following relation:

\begin{proposition} \label{proposition:quasiPeriodicity-step1} We have the following expression
{\allowdisplaybreaks \begin{align}
 (q)_{\infty}^{3}&\mathcal{C}_{2k+1,2r+1}^{(p,2p+j)}(q)
 -(q)_{\infty}^{3}q^{-p(2k+1+j)} \mathcal{C}_{2k+2j+1,2r+1}^{(p,2p+j)}(q)
 \label{equation:quasiPeriodString1}\\
&\qquad   = - \sum_{m=1}^{2p}(-1)^{m}q^{-m(2+k+r)}q^{\binom{m+1}{2}}
j(-q^{(p-m)(2p+j)+p(2r+2)};q^{2p(2p+j)})
\notag \\
&\qquad  \qquad  + \sum_{m=1}^{2p}(-1)^{m}q^{-m(k-r)}q^{\binom{m+1}{2}}j(-q^{(p+m)(2p+j)+p(2r+2)};q^{2p(2p+j)}).
\notag
\end{align}}%
\end{proposition}%

Taking advantage of some cancellation as well as some symmetries brings us to a more compact expression:

\begin{proposition}\label{proposition:quasiPeriodicity-step2} We have
{\allowdisplaybreaks \begin{align*}
 (q)_{\infty}^{3}&\mathcal{C}_{2k+1,2r+1}^{(p,2p+j)}(q)
 -(q)_{\infty}^{3}q^{-p(2k+1+j)} \mathcal{C}_{2k+2j+1,2r+1}^{(p,2p+j)}(q)\\
&  = -(-1)^{p}q^{\binom{p+1}{2}-p(k+r+2)} \sum_{m=1}^{p-1}(-1)^{m}q^{\binom{m+1}{2}+m(r-p)}
\left (q^{m(k+1)} -q^{-m(k+j)}\right ) \\
&\qquad  \times\Big (  j(-q^{m(2p+j)+p(2r+2)};q^{2p(2p+j)}) \\
&\qquad \qquad  -q^{m(2p+j)-m(2r+2)}j(-q^{-m(2p+j)+p(2r+2)};q^{2p(2p+j)})\Big ) .
\end{align*}}%
\end{proposition}

Now we can start the proof of the Theorem \ref{theorem:generalQuasiPeriodicityOddSpin}.  The proofs of the two propositions follow.

\begin{proof}[Proof of Theorem \ref{theorem:generalQuasiPeriodicityOddSpin}]
Changing the string function notation with (\ref{equation:mathCalCtoStringC}), we have
\begin{equation*}
    q^{-\frac{1}{8}+\frac{p(2r+2)^2}{4(2p+j)}-\frac{p}{4j}(2k+1)^2}
    \mathcal{C}_{2k+1,2r+1}^{(p,2p+j)}(q)
    =C_{2k+1,2r+1}^{(p,2p+j)}(q).
\end{equation*}
Hence we have
{\allowdisplaybreaks \begin{align*}
& (q)_{\infty}^{3}C_{2k+1,2r+1}^{(p,2p+j)}(q)
 -(q)_{\infty}^{3}C_{2k+2j+1,2r+1}^{(p,2p+j)}(q)\\
&  = -(-1)^{p}q^{-\frac{1}{8}+\frac{p(2r+2)^2}{4(2p+j)}-\frac{p}{4j}(2k+1)^2}q^{\binom{p+1}{2}-p(k+r+2)} \sum_{m=1}^{p-1}(-1)^{m}q^{\binom{m+1}{2}+m(r-p)}
\left (q^{m(k+1)} -q^{-m(k+j)}\right ) \\
&\qquad  \times\Big (  j(-q^{m(2p+j)+p(2r+2)};q^{2p(2p+j)}) \\
&\qquad \qquad  -q^{m(2p+j)-m(2r+2)}j(-q^{-m(2p+j)+p(2r+2)};q^{2p(2p+j)})\Big ) .
\end{align*}}%

Rearranging terms and replacing $k\to k-j$ gives
{\allowdisplaybreaks \begin{align*}
(q)_{\infty}^{3}&C_{2k+1,2r+1}^{(p,2p+j)}(q)
- (q)_{\infty}^{3}C_{2k-2j+1,2r+1}^{(p,2p+j)}(q)\\
&  = (-1)^{p}q^{-\frac{1}{8}+\frac{p(2r+2)^2}{4(2p+j)}-\frac{p}{4j}(2k-2j+1)^2}q^{\binom{p+1}{2}-p(k-j+r+2)}\\
&\qquad \times \sum_{m=1}^{p-1}(-1)^{m}q^{\binom{m+1}{2}+m(r-p)}
\left (q^{m(k-j+1)} -q^{-mk}\right ) \\
&\qquad  \times\Big (  j(-q^{m(2p+j)+p(2r+2)};q^{2p(2p+j)}) \\
&\qquad \qquad  -q^{m(2p+j)-m(2r+2)}j(-q^{-m(2p+j)+p(2r+2)};q^{2p(2p+j)})\Big ) .
\end{align*}}%
Simplifying the exponents yields
{\allowdisplaybreaks \begin{align*}
 (q)_{\infty}^{3}&C_{2k+1,2r+1}^{(p,2p+j)}(q)
 -(q)_{\infty}^{3}C_{2k-2j+1,2r+1}^{(p,2p+j)}(q)\\
&  = (-1)^{p}q^{-\frac{1}{8}+\frac{p(2r+2)^2}{4(2p+j)}-\frac{p}{4j}(2k+1)^2}q^{\binom{p+1}{2}+p(k-r-1)}\\
&\qquad \times \sum_{m=1}^{p-1}(-1)^{m}q^{\binom{m+1}{2}+m(r-p)}
\left (q^{m(k-j+1)} -q^{-mk}\right ) \\
&\qquad  \times\Big (  j(-q^{m(2p+j)+p(2r+2)};q^{2p(2p+j)}) \\
&\qquad \qquad  -q^{m(2p+j)-m(2r+2)}j(-q^{-m(2p+j)+p(2r+2)};q^{2p(2p+j)})\Big ) .
\end{align*}}%

\noindent Now, we consider $m=2k$, $k=jt+s$, $0\le s\le j-1$.  This allows us to write
{\allowdisplaybreaks \begin{align*}
 (q)_{\infty}^{3}&C_{2jt+2s+1,2r+1}^{(p,2p+j)}(q)
 -(q)_{\infty}^{3}C_{2j(t-1)+2s+1,2r+1}^{(p,2p+j)}(q)\\
&  = (-1)^{p}q^{-\frac{1}{8}+\frac{p(2r+2)^2}{4(2p+j)}-\frac{p}{4j}(2(jt+s)+1)^2}q^{\binom{p+1}{2}+p(jt+s-r-1)}\\
&\qquad \times \sum_{m=1}^{p-1}(-1)^{m}q^{\binom{m+1}{2}+m(r-p)}
\left (q^{m(jt+s-j+1)} -q^{-m(jt+s)}\right ) \\
&\qquad  \times\Big (  j(-q^{m(2p+j)+p(2r+2)};q^{2p(2p+j)}) \\
&\qquad \qquad  -q^{m(2p+j)-m(2r+2)}j(-q^{-m(2p+j)+p(2r+2)};q^{2p(2p+j)})\Big )\\
&  = (-1)^{p}q^{-\frac{1}{8}+\frac{p(2r+2)^2}{4(2p+j)}}q^{\binom{p+1}{2}-2pj\binom{t}{2}-p(r+1-s+(2s+1)t)-\frac{p}{4j}(2s+1)^2}\\
&\qquad \times \sum_{m=1}^{p-1}(-1)^{m}q^{\binom{m+1}{2}+m(r-p)}
\left (q^{m(jt+s-j+1)} -q^{-m(jt+s)}\right ) \\
&\qquad  \times\Big (  j(-q^{m(2p+j)+p(2r+2)};q^{2p(2p+j)}) \\
&\qquad \qquad  -q^{m(2p+j)-m(2r+2)}j(-q^{-m(2p+j)+p(2r+2)};q^{2p(2p+j)})\Big ).
\end{align*}}%
Iterating gives us
{\allowdisplaybreaks \begin{align*}
 (q)_{\infty}^{3}&C_{2jt+2s+1,2r+1}^{(p,2p+j)}(q)\\
 &=(q)_{\infty}^{3}C_{2j(t-1)+2s+1,2r+1}^{(p,2p+j)}(q)\\
& \qquad + (-1)^{p}q^{-\frac{1}{8}+\frac{p(2r+2)^2}{4(2p+j)}}
\sum_{i=1}^{t}q^{\binom{p+1}{2}-2pj\binom{i}{2}-p(r+1-s+(2s+1)i)-\frac{p}{4j}(2s+1)^2}\\
&\qquad \times \sum_{m=1}^{p-1}(-1)^{m}q^{\binom{m+1}{2}+m(r-p)}
\left (q^{m(ji+s-j+1)} -q^{-m(ji+s)}\right ) \\
&\qquad  \times\Big (  j(-q^{m(2p+j)+p(2r+2)};q^{2p(2p+j)}) \\
&\qquad \qquad  -q^{m(2p+j)-m(2r+2)}j(-q^{-m(2p+j)+p(2r+2)};q^{2p(2p+j)})\Big ).
\end{align*}}
Rewriting the expression gives the result.
\end{proof}


\begin{proof}[Proof of Proposition \ref{proposition:quasiPeriodicity-step1}] In Proposition \ref{proposition:f-functionaleqn}, we set $(\ell,k)=(-2p,1)$.  For the first double-sum in (\ref{equation:stringQuasiHecke}), this gives
{\allowdisplaybreaks \begin{align*}
 &f_{1,2p+j,2p(2p+j)}(q^{2+k+r},-q^{p(2p+j+2r+2)};q)\\
 &\qquad =(-q^{2+k+r})^{(-2p)}q^{p(2p+j+2r+2)}q^{\binom{-2p}{2}-2p(2p+j)}
 f_{1,2p+j,2p(2p+j)}(q^{j+2+k+r},-q^{p(2p+j+2r+2)};q)\\
&\qquad \qquad +\sum_{m=0}^{-2p-1}(-q^{2+k+r})^mq^{\binom{m}{2}}j(-q^{m(2p+j)}q^{p(2p+j+2r+2)};q^{2p(2p+j)})
+j(q^{2+k+r};q)\\
 &\qquad =q^{-p(2k+1+j)}
 f_{1,2p+j,2p(2p+j)}(q^{2+j+k+r},-q^{p(2p+j+2r+2)};q)\\
&\qquad \qquad -\sum_{m=-2p}^{-1}(-q^{2+k+r})^mq^{\binom{m}{2}}j(-q^{m(2p+j)}q^{p(2p+j+2r+2)};q^{2p(2p+j)}),
\end{align*}}%
where in the last equality we have used the summation convention (\ref{equation:sumconvention}) and the fact that $j(q^{n};q)=~0$ for $n\in\mathbb{Z}$.  Simplifying some more yields ($m\to -m$)
{\allowdisplaybreaks \begin{align}
 &f_{1,2p+j,2p(2p+j)}(q^{2+k+r},-q^{p(2p+j+2r+2)};q) \label{equation:quasi-preSum1}\\
 &\qquad  =q^{-p(2k+1+j)}
 f_{1,2p+j,2p(2p+j)}(q^{2+j+k+r},-q^{p(2p+j+2r+2)};q)\notag \\
&\qquad \qquad -\sum_{m=1}^{2p}(-1)^{m}q^{-m(2+k+r})q^{\binom{m+1}{2}}j(-q^{(p-m)(2p+j)+p(2r+2)};q^{2p(2p+j)})\notag.
\end{align}}%

\noindent For the second double-sum in (\ref{equation:stringQuasiHecke}), the same approach yields
{\allowdisplaybreaks \begin{align*}
&f_{1,2p+j,2p(2p+j)}(q^{k-r},-q^{p(2p+j-2r-2)};q)\\
&\qquad =(-q^{k-r})^{-2p}q^{p(2p+j-2r-2)}
q^{\binom{-2p}{2}-2p(2p+j)}f_{1,2p+j,2p(2p+j)}(q^{j+k-r},-q^{p(2p+j-2r-2)};q)\\
&\qquad \qquad  +\sum_{m=0}^{-2p-1}(-q^{k-r})^mq^{\binom{m}{2}}j(-q^{m(2p+j)}q^{p(2p+j-2r-2)};q^{2p(2p+j)})
+j(q^{k-r};q)\\
&\qquad =q^{-p(2k+1+j)}f_{1,2p+j,2p(2p+j)}(q^{j+k-r},-q^{p(2p+j-2r-2)};q)\\
&\qquad \qquad  -\sum_{m=-2p}^{-1}(-q^{k-r})^mq^{\binom{m}{2}}j(-q^{m(2p+j)}q^{p(2p+j-2r-2)};q^{2p(2p+j)})\\
&\qquad =q^{-p(2k+1+j)}f_{1,2p+j,2p(2p+j)}(q^{j+k-r},-q^{p(2p+j-2r-2)};q)\\
&\qquad \qquad  -\sum_{m=1}^{2p}(-1)q^{-m(k-r)}q^{\binom{m+1}{2}}j(-q^{(p-m)(2p+j)-p(2r+2)};q^{2p(2p+j)}).
\end{align*}}%

\noindent Applying (\ref{equation:j-flip}) to the theta function yields
\begin{align}
&f_{1,2p+j,2p(2p+j)}(q^{k-r},-q^{p(2p+j-(2r+2))};q)\label{equation:quasi-preSum2}\\
&\qquad =q^{-p(2k+1+j)}f_{1,2p+j,2p(2p+j)}(q^{j+k-r},-q^{p(2p+j-(2r+2))};q)\notag \\
&\qquad \qquad  -\sum_{m=1}^{2p}(-1)q^{-m(k-r)}q^{\binom{m+1}{2}}j(-q^{(p+m)(2p+j)+p(2r+2)};q^{2p(2p+j)}).\notag
\end{align}

\noindent Subtracting (\ref{equation:quasi-preSum2}) from (\ref{equation:quasi-preSum1}) gives
{\allowdisplaybreaks \begin{align*}
 &f_{1,2p+j,2p(2p+j)}(q^{2+k+r},-q^{p(2p+j+2r+2)};q) -f_{1,2p+j,2p(2p+j)}(q^{k-r},-q^{p(2p+j-(2r+2))};q)\\
 &\qquad  =q^{-p(2k+1+j)}
\Big (  f_{1,2p+j,2p(2p+j)}(q^{2+j+k+r},-q^{p(2p+j+2r+2)};q)\\
&\qquad \qquad \qquad - f_{1,2p+j,2p(2p+j)}(q^{j+k-r},-q^{p(2p+j-(2r+2))};q)  \Big )  \\
&\qquad \qquad -\sum_{m=1}^{2p}(-1)^{m}q^{-m(2+k+r})q^{\binom{m+1}{2}}j(-q^{(p-m)(2p+j)+p(2r+2)};q^{2p(2p+j)})\\
&\qquad \qquad  +\sum_{m=1}^{2p}(-1)q^{-m(k-r)}q^{\binom{m+1}{2}}j(-q^{(p+m)(2p+j)+p(2r+2)};q^{2p(2p+j)}).
\end{align*}}%
Using the string function notation (\ref{equation:stringQuasiHecke}), we obtain the stated result.
\end{proof}

\begin{proof}[Proof of Proposition \ref{proposition:quasiPeriodicity-step2}]  We first note that the terms for $m=p$ in the first and second sums in (\ref{equation:quasiPeriodString1}) cancel.  This follows from (\ref{equation:j-elliptic}) and simplifying:
{\allowdisplaybreaks \begin{align*}
& - (-1)^{p}q^{-p(2+k+r)}q^{\binom{p+1}{2}} j(-q^{p(2r+2)};q^{2p(2p+j)}) \\
&\qquad  \qquad  + (-1)^{p}q^{-p(k-r)}q^{\binom{p+1}{2}}j(-q^{2p(2p+j)+p(2r+2)};q^{2p(2p+j)})\\
&\qquad =  - (-1)^{p}q^{-p(2+k+r)}q^{\binom{p+1}{2}} j(-q^{p(2r+2)};q^{2p(2p+j)}) \\
&\qquad  \qquad  + (-1)^{p}q^{-p(k-r)}q^{\binom{p+1}{2}}q^{-(2r+2)}j(-q^{p(2r+2)};q^{2p(2p+j)})=0.
\end{align*}}%

\noindent So we can omit the two terms for $m=p$ and break up the two summations in (\ref{equation:quasiPeriodString1}) to write
{\allowdisplaybreaks \begin{align}
 (q)_{\infty}^{3}&\mathcal{C}_{2k+1,2r+1}^{(p,2p+j)}(q)
 -(q)_{\infty}^{3}q^{-p(2k+1+j)} \mathcal{C}_{2k+2j+1,2r+1}^{(p,2p+j)}(q)
 \label{equation:quasiPeriodicInitBreak}\\
&\qquad   = - \sum_{m=1}^{p-1}(-1)^{m}q^{-m(2+k+r)}q^{\binom{m+1}{2}}
j(-q^{(p-m)(2p+j)+p(2r+2)};q^{2p(2p+j)})
\notag \\
&\qquad   \qquad  - \sum_{m=p+1}^{2p}(-1)^{m}q^{-m(2+k+r)}q^{\binom{m+1}{2}}
j(-q^{(p-m)(2p+j)+p(2r+2)};q^{2p(2p+j)})
\notag \\
&\qquad  \qquad  + \sum_{m=1}^{p-1}(-1)^{m}q^{-m(k-r)}q^{\binom{m+1}{2}}j(-q^{(p+m)(2p+j)+p(2r+2)};q^{2p(2p+j)})
\notag\\
&\qquad  \qquad  + \sum_{m=p+1}^{2p}(-1)^{m}q^{-m(k-r)}q^{\binom{m+1}{2}}j(-q^{(p+m)(2p+j)+p(2r+2)};q^{2p(2p+j)}).
\notag
\end{align}}%

\noindent In the first and third summation symbols of (\ref{equation:quasiPeriodicInitBreak}), we make the substitution $m\to p-m$: 
{\allowdisplaybreaks \begin{align}
 (q)_{\infty}^{3}&\mathcal{C}_{2k+1,2r+1}^{(p,2p+j)}(q)
 -(q)_{\infty}^{3}q^{-p(2k+1+j)} \mathcal{C}_{2k+2j+1,2r+1}^{(p,2p+j)}(q)
 \label{equation:quasiPeriodicInitBreak2}\\
&\qquad   = - \sum_{m=1}^{p-1}(-1)^{p-m}q^{(m-p)(2+k+r)}q^{\binom{p-m+1}{2}}
j(-q^{m(2p+j)+p(2r+2)};q^{2p(2p+j)})
\notag \\
&\qquad   \qquad  - \sum_{m=p+1}^{2p}(-1)^{m}q^{-m(2+k+r)}q^{\binom{m+1}{2}}
j(-q^{(p-m)(2p+j)+p(2r+2)};q^{2p(2p+j)})
\notag \\
&\qquad  \qquad  + \sum_{m=1}^{p-1}(-1)^{p-m}q^{(m-p)(k-r)}q^{\binom{p-m+1}{2}}j(-q^{(2p-m)(2p+j)+p(2r+2)};q^{2p(2p+j)})
\notag\\
&\qquad  \qquad  + \sum_{m=p+1}^{2p}(-1)^{m}q^{-m(k-r)}q^{\binom{m+1}{2}}j(-q^{(p+m)(2p+j)+p(2r+2)};q^{2p(2p+j)}).
\notag
\end{align}}%
\noindent In the third summation of (\ref{equation:quasiPeriodicInitBreak2}), we use (\ref{equation:j-elliptic}) and simplify.  This gives
{\allowdisplaybreaks \begin{align}
& (q)_{\infty}^{3}\mathcal{C}_{2k+1,2r+1}^{(p,2p+j)}(q)
 -(q)_{\infty}^{3}q^{-p(2k+1+j)} \mathcal{C}_{2k+2j+1,2r+1}^{(p,2p+j)}(q)
 \label{equation:quasiPeriodicInitBreak3}\\
&\quad   = - (-1)^{p}q^{\binom{p+1}{2}-p(k+r+2)}
\sum_{m=1}^{p-1}(-1)^{m}q^{\binom{m+1}{2}+m(k+r+1)-mp}
j(-q^{m(2p+j)+p(2r+2)};q^{2p(2p+j)})
\notag \\
&\quad   \quad  - \sum_{m=p+1}^{2p}(-1)^{m}q^{-m(2+k+r)}q^{\binom{m+1}{2}}
j(-q^{(p-m)(2p+j)+p(2r+2)};q^{2p(2p+j)})
\notag \\
&\quad  \quad  +(-1)^{p}q^{\binom{p+1}{2}-p(k+r+2)}
 \sum_{m=1}^{p-1}(-1)^{m}q^{\binom{m+1}{2}+m(k-r+p+j-1)}
 j(-q^{-m(2p+j)+p(2r+2)};q^{2p(2p+j)})
\notag\\
&\quad  \quad  + \sum_{m=p+1}^{2p}(-1)^{m}q^{-m(k-r)}q^{\binom{m+1}{2}}j(-q^{(p+m)(2p+j)+p(2r+2)};q^{2p(2p+j)}).
\notag
\end{align}}%
\noindent In the second and fourth summations of (\ref{equation:quasiPeriodicInitBreak3}) we make the substitution $m\to p+m$:
{\allowdisplaybreaks \begin{align}
 &(q)_{\infty}^{3}\mathcal{C}_{2k+1,2r+1}^{(p,2p+j)}(q)
 -(q)_{\infty}^{3}q^{-p(2k+1+j)} \mathcal{C}_{2k+2j+1,2r+1}^{(p,2p+j)}(q)
 \label{equation:quasiPeriodicInitBreak4}\\
&\quad   = - (-1)^{p}q^{\binom{p+1}{2}-p(k+r+2)}\sum_{m=1}^{p-1}(-1)^{m}q^{\binom{m+1}{2}+m(k+r+1)-mp}
j(-q^{m(2p+j)+p(2r+2)};q^{2p(2p+j)})
\notag \\
&\quad   \quad  - \sum_{m=1}^{p}(-1)^{p+m}q^{-(p+m)(2+k+r)}q^{\binom{p+m+1}{2}}
j(-q^{-m(2p+j)+p(2r+2)};q^{2p(2p+j)})
\notag \\
&\quad  \quad  +(-1)^{p}q^{\binom{p+1}{2}-p(k+r+2)}
 \sum_{m=1}^{p-1}(-1)^{m}q^{\binom{m+1}{2}+m(k-r+p+j-1)}
 j(-q^{-m(2p+j)+p(2r+2)};q^{2p(2p+j)})
\notag\\
&\quad  \quad  + \sum_{m=1}^{p}(-1)^{p+m}q^{-(p+m)(k-r)}q^{\binom{p+m+1}{2}}j(-q^{(2p+m)(2p+j)+p(2r+2)};q^{2p(2p+j)}).
\notag
\end{align}}%

\noindent In the fourth summation of (\ref{equation:quasiPeriodicInitBreak4}) we use (\ref{equation:j-elliptic}).  This brings us to
{\allowdisplaybreaks \begin{align}
& (q)_{\infty}^{3}\mathcal{C}_{2k+1,2r+1}^{(p,2p+j)}(q)
 -(q)_{\infty}^{3}q^{-p(2k+1+j)} \mathcal{C}_{2k+2j+1,2r+1}^{(p,2p+j)}(q)
 \label{equation:quasiPeriodicInitBreak5}\\
&\quad   = - (-1)^{p}q^{\binom{p+1}{2}-p(k+r+2)}\sum_{m=1}^{p-1}(-1)^{m}q^{\binom{m+1}{2}+m(k+r+1)-mp}
j(-q^{m(2p+j)+p(2r+2)};q^{2p(2p+j)})
\notag \\
&\qquad   - (-1)^{p}q^{\binom{p+1}{2}-p(k+r+2)}\sum_{m=1}^{p}(-1)^{m}q^{\binom{m+1}{2}-m(k+r+2-p)}
j(-q^{-m(2p+j)+p(2r+2)};q^{2p(2p+j)})
\notag \\
&\qquad   +(-1)^{p}q^{\binom{p+1}{2}-p(k+r+2)}
 \sum_{m=1}^{p-1}(-1)^{m}q^{\binom{m+1}{2}+m(k-r+p+j-1)}
 j(-q^{-m(2p+j)+p(2r+2)};q^{2p(2p+j)})
\notag\\
&\qquad   +(-1)^{p}q^{\binom{p+1}{2}-p(k+r+2)} \sum_{m=1}^{p}(-1)^{m}q^{\binom{m+1}{2}+m(r-k-j-p)}
j(-q^{m(2p+j)+p(2r+2)};q^{2p(2p+j)}).
\notag
\end{align}}%

\noindent Now, in the second and fourth summations in (\ref{equation:quasiPeriodicInitBreak5}), we see that the terms for $m=p$ cancel:
\begin{align*}
  - &q^{\binom{p+1}{2}-p(k+r+2)}q^{\binom{p+1}{2}-p(k+r+2-p)}
j(-q^{-p(2p+j)+p(2r+2)};q^{2p(2p+j)})\\
&\qquad    +q^{\binom{p+1}{2}-p(k+r+2)} q^{\binom{p+1}{2}+p(r-k-j-p)}
j(-q^{p(2p+j)+p(2r+2)};q^{2p(2p+j)})\\
&=  - q^{\binom{p+1}{2}-p(k+r+2)}q^{\binom{p+1}{2}-p(k+r+2-p)}
j(-q^{-p(2p+j)+p(2r+2)};q^{2p(2p+j)})\\
&\qquad    +q^{\binom{p+1}{2}-p(k+r+2)} q^{\binom{p+1}{2}+p(r-k-j-p)}
q^{p(2p+j)-p(2r+2)}j(-q^{-p(2p+j)+p(2r+2)};q^{2p(2p+j)})=0,
\end{align*}
where we used (\ref{equation:j-elliptic}) and simplified.  In (\ref{equation:quasiPeriodicInitBreak5}), we group together the first and fourth summations as well as the second and third summations to obtain
{\allowdisplaybreaks \begin{align*}
 (q)_{\infty}^{3}&\mathcal{C}_{2k+1,2r+1}^{(p,2p+j)}(q)
 -(q)_{\infty}^{3}q^{-p(2k+1+j)} \mathcal{C}_{2k+2j+1,2r+1}^{(p,2p+j)}(q)\\
&  = (-1)^{p}q^{\binom{p+1}{2}-p(k+r+2)} \sum_{m=1}^{p-1}(-1)^{m}q^{\binom{m+1}{2}+m(r-k-j-p)}
\left ( 1-q^{m(2k+1+j)}\right ) \\
&\qquad \qquad \times j(-q^{m(2p+j)+p(2r+2)};q^{2p(2p+j)}) \\
&\qquad    - (-1)^{p}q^{\binom{p+1}{2}-p(k+r+2)}\sum_{m=1}^{p-1}(-1)^{m}q^{\binom{m+1}{2}-m(k+r+2-p)}
\left ( 1-q^{m(2k+1+j)}\right ) \\
&\qquad \qquad \times j(-q^{-m(2p+j)+p(2r+2)};q^{2p(2p+j)}).
\end{align*}}%

\noindent Combining the remaining two sums yields
{\allowdisplaybreaks \begin{align*}
 &(q)_{\infty}^{3}\mathcal{C}_{2k+1,2r+1}^{(p,2p+j)}(q)
 -(q)_{\infty}^{3}q^{-p(2k+1+j)} \mathcal{C}_{2k+2j+1,2r+1}^{(p,2p+j)}(q)\\
&  = (-1)^{p}q^{\binom{p+1}{2}-p(k+r+2)} \sum_{m=1}^{p-1}(-1)^{m}q^{\binom{m+1}{2}+m(r-k-j-p)}
\left ( 1-q^{m(2k+1+j)}\right ) \\
&\qquad  \times\left (  j(-q^{m(2p+j)+p(2r+2)};q^{2p(2p+j)}) 
-q^{m(2p+j)-m(2r+2)}j(-q^{-m(2p+j)+p(2r+2)};q^{2p(2p+j)})\right ) .
\end{align*}}%

\noindent Rewriting the exponents gives the result.
\end{proof}

 
 \section{On a generalized Euler identity for modified string functions: Theorems}\label{section:genEulerTheorems}
  
In Section \ref{section:genEulerTheoremsEvenSpin}, we apply our quasi-periodicity relation for even-spin in Theorem \ref{theorem:generalQuasiPeriodicityEvenSpin} to the generalized Euler identity (\ref{equation:genEulerIdentitySW}) to obtain Theorem  \ref{theorem:genEulerIdentityEvenSpin}.  Likewise, in Section \ref{section:genEulerTheoremsOddSpin}, we apply our quasi-periodicity relation for odd-spin in Theorem \ref{theorem:generalQuasiPeriodicityOddSpin} to the generalized Euler identity (\ref{equation:genEulerIdentitySW}) to obtain Theorem  \ref{theorem:genEulerIdentityOddSpin}.


\subsection{The quasi-periodicity approach for admissible-level, even-spin}\label{section:genEulerTheoremsEvenSpin}

We first change the notation of Theorem \ref{theorem:generalQuasiPeriodicityEvenSpin}:
\begin{corollary}\label{corollary:quasiPeriodicityEvenSpinMathCal} We have
{\allowdisplaybreaks \begin{align*}
&  (q)_{\infty}^3\mathcal{C}_{2jt+2s,2r}^{(p,2p+j)}(q)
 -(q)_{\infty}^3q^{pjt^2+2pts}\mathcal{C}_{2s,2r}^{(p,2p+j)}(q)\\
& \quad  = (-1)^{p}q^{pjt^2+2pts}q^{\binom{p}{2}-p(r-s)}\sum_{i=1}^{t}q^{-2pj\binom{i}{2}-2psi} \\
&\qquad  \times 
\sum_{m=1}^{p-1}(-1)^{m}q^{\binom{m+1}{2}+m(r-p)}
 \left ( q^{m(ji+s-j)}-q^{-m(ji+s)}\right )\\
&\qquad   \times 
\Big (  j(-q^{m(2p+j)+p(2r+1)};q^{2p(2p+j)} )
  -  q^{m(2p+j)-m(2r+1)}j(-q^{-m(2p+j)+p(2r+1)};q^{2p(2p+j)})\Big ).
\end{align*}}%
\end{corollary}
\begin{proof}[Proof of Corollary \ref{corollary:quasiPeriodicityEvenSpinMathCal}]
We recall (\ref{equation:mathCalCtoStringC}): 
\begin{equation*}
\mathcal{C}_{m,\ell}^{N}(q) := q^{-s_{\lambda}}C_{m,\ell}^{N}(q), \ 
s_{\lambda} := -\frac{1}{8}+\frac{(\ell+1)^2}{4(N+2)} - \frac{m^2}{4N},
\end{equation*}
to get
{\allowdisplaybreaks \begin{align*}
C_{2jt+2s,2r}^{(p,2p+j)}(q)&
 -C_{2s,2r}^{(p,2p+j)}(q)\\
&=q^{-\frac{1}{8}+\frac{p(2r+1)^2}{4(2p+j)}-\frac{p(2jt+2s)^2}{4j}} \mathcal{C}_{2jt+2s,2r}^{(p,2p+j)}(q)
 -q^{-\frac{1}{8}+\frac{p(2r+1)^2}{4(2p+j)}-\frac{p(2s)^2}{4j}}\mathcal{C}_{2s,2r}^{(p,2p+j)}(q).
\end{align*}}%

\noindent So our quasi-periodicity relation can be rewritten
{\allowdisplaybreaks \begin{align*}
& q^{-\frac{1}{8}+\frac{p(2r+1)^2}{4(2p+j)}-\frac{p(2jt+2s)^2}{4j}} (q)_{\infty}^3\mathcal{C}_{2jt+2s,2r}^{(p,2p+j)}(q)
 -q^{-\frac{1}{8}+\frac{p(2r+1)^2}{4(2p+j)}-\frac{p(2s)^2}{4j}}(q)_{\infty}^3\mathcal{C}_{2s,2r}^{(p,2p+j)}(q)\\
& \  = (-1)^{p}q^{-\frac{1}{8}+\frac{p(2r+1)^2}{4(2p+j)}}q^{\binom{p}{2}-p(r-s)-\frac{p}{j}s^2}\sum_{i=1}^{t}q^{-2pj\binom{i}{2}-2psi} \\
&\quad  \times 
\sum_{m=1}^{p-1}(-1)^{m}q^{\binom{m+1}{2}+m(r-p)}
 \left ( q^{m(ji+s-j)}-q^{-m(ji+s)}\right )\\
&\quad   \times 
\Big (  j(-q^{m(2p+j)+p(2r+1)};q^{2p(2p+j)} )
  -  q^{m(2p+j)-m(2r+1)}j(-q^{-m(2p+j)+p(2r+1)};q^{2p(2p+j)})\Big ).
\end{align*}}%
Multiplying through by the appropriate power of $q$ gives the result.
\end{proof}

\begin{proof}[Proof of Theorem \ref{theorem:genEulerIdentityEvenSpin}]  For $\eta=2s$, $\ell=2r$, $L=jt+a$, we have
\begin{equation*}
 C_{2L+\eta,\ell}^{(p,p^{\prime})}(q)=C_{2(jt+a)+2s,2r}^{N}(q)
 = C_{2jt+2(a+s),2r}^{N}(q).
\end{equation*}
Inserting into the generalized Euler identity (\ref{equation:genEulerIdentitySW}) yields
{\allowdisplaybreaks \begin{align*}
(q)_{\infty}^{3}\delta_{2r,2s}
&= (q)_{\infty}^{3} \sum_{L=-\infty}^{\infty}(-1)^{L}q^{\binom{L}{2}}\mathcal{C}_{2L+2s,2r}^{(p,2p+j)}(q)\\
&=(q)_{\infty}^{3}\sum_{a=0}^{j-1} \sum_{t=-\infty}^{\infty}(-1)^{jt+a}q^{\binom{jt+a}{2}}\mathcal{C}_{2(jt+a)+2s,2r}^{(p,2p+j)}(q)\\
&=(q)_{\infty}^{3}\sum_{a=0}^{j-1} \sum_{t=-\infty}^{\infty}(-1)^{jt+a}q^{\binom{jt+a}{2}}\mathcal{C}_{2jt+2(a+s),2r}^{(p,2p+j)}(q).
\end{align*}}%

\noindent We then use Corollary \ref{corollary:quasiPeriodicityEvenSpinMathCal}:
{\allowdisplaybreaks \begin{align*}
(q)_{\infty}^{3}\delta_{2r,2s}
& =(q)_{\infty}^{3}\sum_{a=0}^{j-1} \sum_{t=-\infty}^{\infty}(-1)^{jt+a}q^{\binom{jt+a}{2}}
q^{pjt^2+2pt(a+s)}\mathcal{C}_{2(a+s),2r}^{(p,2p+j)}(q)\\
&\qquad 
+\sum_{a=0}^{j-1} \sum_{t=-\infty}^{\infty}(-1)^{jt+a}q^{\binom{jt+a}{2}}
 (-1)^{p}q^{pjt^2+2pt(a+s)}q^{\binom{p}{2}-p(r-a-s)}\\
&\qquad  \times 
\sum_{i=1}^{t}q^{-2pj\binom{i}{2}-2p(a+s)i} \sum_{m=1}^{p-1}(-1)^{m}q^{\binom{m+1}{2}+m(r-p)}
 \left ( q^{m(ji+a+s-j)}-q^{-m(ji+a+s)}\right )\\
&\qquad   \times 
\Big (  j(-q^{m(2p+j)+p(2r+1)};q^{2p(2p+j)} )\\
&\qquad \quad   -  q^{m(2p+j)-m(2r+1)}j(-q^{-m(2p+j)+p(2r+1)};q^{2p(2p+j)})\Big ).
\end{align*}}%

\noindent We rewrite this as
{\allowdisplaybreaks\begin{align}
(q)_{\infty}^{3}\delta_{2r,2s}
& =(q)_{\infty}^{3}\sum_{a=0}^{j-1}(-1)^{a} q^{\binom{a}{2}} j(-(-1)^{j}q^{j(a+p)+\binom{j}{2}+2p(a+s)};q^{j(2p+j)})
\mathcal{C}_{2(a+s),2r}^{(p,2p+j)}(q)
\label{equation:preLastTwoLines}\\
&\qquad 
+(-1)^{p}q^{\binom{p}{2}-p(r-s)} \sum_{a=0}^{j-1} (-1)^{a}q^{pa}\sum_{m=1}^{p-1}(-1)^{m}q^{\binom{m+1}{2}+m(r-p)}
\notag\\
&\qquad   \times 
\Big (  j(-q^{m(2p+j)+p(2r+1)};q^{2p(2p+j)} )
\notag\\
&\qquad \qquad   -  q^{m(2p+j)-m(2r+1)}j(-q^{-m(2p+j)+p(2r+1)};q^{2p(2p+j)})\Big )
\notag\\
&\qquad \times \sum_{t=-\infty}^{\infty}(-1)^{jt}q^{\binom{jt+a}{2}}
q^{pjt^2+2pt(a+s)}
\notag\\
&\qquad  \times 
\sum_{i=1}^{t}q^{-2pj\binom{i}{2}-2p(a+s)i} 
 \left ( q^{m(ji+a+s-j)}-q^{-m(ji+a+s)}\right ).\notag
\end{align}}%
We focus on the last two lines of (\ref{equation:preLastTwoLines}), but we first rewrite the exponents.
{\allowdisplaybreaks\begin{align*}
 \sum_{t=-\infty}^{\infty}&(-1)^{jt}q^{\binom{jt+a}{2}}
q^{pjt^2+2pt(a+s)}  
\sum_{i=1}^{t}q^{-2pj\binom{i}{2}-2p(a+s)i} 
 \left ( q^{m(ji+a+s-j)}-q^{-m(ji+a+s)}\right )\\
 &= \sum_{t=-\infty}^{\infty}(-1)^{jt}q^{\binom{a}{2}+j(2p+j)\binom{t}{2}+t(j(a+p)+\binom{j}{2}+2p(a+s))}\\
&\qquad \times \sum_{i=1}^{t}q^{-2pj\binom{i}{2}-2p(a+s)i} 
 \left ( q^{m(ji+a+s-j)}-q^{-m(ji+a+s)}\right )=:\Psi_{a,m}^{s}(q).
\end{align*}}%
For $t=0$, the inner sum vanishes by our summation convention, so we can write
{\allowdisplaybreaks \begin{align*}
\Psi_{a,m}^{s}(q)&=q^{\binom{a}{2}}
\sum_{t\ge1} \sum_{i=1}^{t}(-1)^{jt}q^{j(2p+j)\binom{t}{2}+t(j(a+p)+\binom{j}{2}+2p(a+s))}\\
&\qquad \qquad \times q^{-2pj\binom{i}{2}-2p(a+s)i} 
 \left ( q^{m(ji+a+s-j)}-q^{-m(ji+a+s)}\right )\\
 &\qquad +q^{\binom{a}{2}}
\sum_{t\le -1} \sum_{i=1}^{t}(-1)^{jt}q^{j(2p+j)\binom{t}{2}+t(j(a+p)+\binom{j}{2}+2p(a+s))}\\
&\qquad \qquad \times q^{-2pj\binom{i}{2}-2p(a+s)i} 
 \left ( q^{m(ji+a+s-j)}-q^{-m(ji+a+s)}\right ).
\end{align*}}%
Using the summation convention again, we can write the second double-sum to get
{\allowdisplaybreaks \begin{align*}
\Psi_{a,m}^{s}(q)&=q^{\binom{a}{2}}
\sum_{t\ge1} \sum_{i=1}^{t}(-1)^{jt}q^{j(2p+j)\binom{t}{2}+t(j(a+p)+\binom{j}{2}+2p(a+s))}\\
&\qquad \qquad \times q^{-2pj\binom{i}{2}-2p(a+s)i} 
 \left ( q^{m(ji+a+s-j)}-q^{-m(ji+a+s)}\right )\\
 &\qquad -q^{\binom{a}{2}}
\sum_{t\le -1} \sum_{i=t+1}^{0}(-1)^{jt}q^{j(2p+j)\binom{t}{2}+t(j(a+p)+\binom{j}{2}+2p(a+s))}\\
&\qquad \qquad \times q^{-2pj\binom{i}{2}-2p(a+s)i} 
 \left ( q^{m(ji+a+s-j)}-q^{-m(ji+a+s)}\right ).
\end{align*}}%
In the second double-sum  we make the substitutions $t\to -t$, $i\to -i+1$:
{\allowdisplaybreaks \begin{align*}
\Psi_{a,m}^{s}(q)&=q^{\binom{a}{2}}
\sum_{t\ge1} \sum_{i=1}^{t}(-1)^{jt}q^{j(2p+j)\binom{t}{2}+t(j(a+p)+\binom{j}{2}+2p(a+s))}\\
&\qquad \qquad \times q^{-2pj\binom{i}{2}-2p(a+s)i} 
 \left ( q^{m(ji+a+s-j)}-q^{-m(ji+a+s)}\right )\\
 &\qquad -q^{\binom{a}{2}}
\sum_{t\ge 1} \sum_{i=1}^{t}(-1)^{jt}q^{j(2p+j)\binom{-t}{2}-t(j(a+p)+\binom{j}{2}+2p(a+s))}\\
&\qquad \qquad \times q^{-2pj\binom{-i+1}{2}-2p(a+s)(-i+1)} 
 \left ( q^{m(j(-i+1)+a+s-j)}-q^{-m(j(-i+1)+a+s)}\right ).
\end{align*}}%
Simplifying yields
{\allowdisplaybreaks \begin{align*}
\Psi_{a,m}^{s}(q)&=q^{\binom{a}{2}}
\sum_{t\ge1} \sum_{i=1}^{t}(-1)^{jt}q^{j(2p+j)\binom{t}{2}+t(j(a+p)+\binom{j}{2}+2p(a+s))}\\
&\qquad \qquad \times q^{-2pj\binom{i}{2}-2p(a+s)i} 
 \left ( q^{m(ji+a+s-j)}-q^{-m(ji+a+s)}\right )\\
 &\qquad -q^{\binom{a}{2}}
\sum_{t\ge 1} \sum_{i=1}^{t}(-1)^{jt}q^{j(2p+j)\binom{t+1}{2}-t(j(a+p)+\binom{j}{2}+2p(a+s))}\\
&\qquad \qquad \times q^{-2pj\binom{i}{2}+2p(a+s)(i-1)} 
 \left ( q^{-m(j(i-1)-a-s+j)}-q^{m(j(i-1)-a-s)}\right ).
\end{align*}}%
Replacing $i$ with $i+1$ and $t$ with $t+1$ gives
{\allowdisplaybreaks \begin{align*}
\Psi_{a,m}^{s}(q)&=q^{\binom{a}{2}}
\sum_{t\ge0} \sum_{i=0}^{t}(-1)^{j(t+1)}q^{j(2p+j)\binom{t+1}{2}+(t+1)(j(a+p)+\binom{j}{2}+2p(a+s))}\\
&\qquad \qquad \times q^{-2pj\binom{i+1}{2}-2p(a+s)(i+1)} 
 \left ( q^{m(ji+a+s)}-q^{-m(j(i+1)+a+s)}\right )\\
 &\qquad -q^{\binom{a}{2}}
\sum_{t\ge 0} \sum_{i=0}^{t}(-1)^{j(t+1)}q^{j(2p+j)\binom{t+2}{2}-(t+1)(j(a+p)+\binom{j}{2}+2p(a+s))}\\
&\qquad \qquad \times q^{-2pj\binom{i+1}{2}+2p(a+s)i} 
 \left ( q^{-m(j(i+1)-a-s)}-q^{m(ji-a-s)}\right ).
\end{align*}}%
We interchange summation symbols to obtain
{\allowdisplaybreaks \begin{align*}
\Psi_{a,m}^{s}(q)&=q^{\binom{a}{2}}
\sum_{i\ge0} \sum_{t=i}^{\infty}(-1)^{j(t+1)}q^{j(2p+j)\binom{t+1}{2}+(t+1)(j(a+p)+\binom{j}{2}+2p(a+s))}\\
&\qquad \qquad \times q^{-2pj\binom{i+1}{2}-2p(a+s)(i+1)} 
 \left ( q^{m(ji+a+s)}-q^{-m(j(i+1)+a+s)}\right )\\
 &\qquad -q^{\binom{a}{2}}
\sum_{i\ge0} \sum_{t=i}^{\infty}(-1)^{j(t+1)}q^{j(2p+j)\binom{t+2}{2}-(t+1)(j(a+p)+\binom{j}{2}+2p(a+s))}\\
&\qquad \qquad \times q^{-2pj\binom{i+1}{2}+2p(a+s)i} 
 \left ( q^{-m(j(i+1)-a-s)}-q^{m(ji-a-s)}\right ).
\end{align*}}%
Replacing $t$ with $i+t$ then brings us to a familiar point:
{\allowdisplaybreaks \begin{align}
\Psi_{a,m}^{s}(q)&=q^{\binom{a}{2}}
\sum_{i=0}^{\infty} \sum_{t=0}^{\infty}(-1)^{j(t+i+1)}q^{j(2p+j)\binom{t+i+1}{2}+(t+i+1)(j(a+p)+\binom{j}{2}+2p(a+s))}
\label{equation:preHeckeTypeDoubleSum}\\
&\qquad \qquad \times q^{-2pj\binom{i+1}{2}-2p(a+s)(i+1)} 
 \left ( q^{m(ji+a+s)}-q^{-m(j(i+1)+a+s)}\right )
 \notag\\
 &\qquad -q^{\binom{a}{2}}
\sum_{i=0}^{\infty} \sum_{t=0}^{\infty}(-1)^{j(t+i+1)}q^{j(2p+j)\binom{t+i+2}{2}-(t+i+1)(j(a+p)+\binom{j}{2}+2p(a+s))}
\notag \\
&\qquad \qquad \times q^{-2pj\binom{i+1}{2}+2p(a+s)i} 
 \left ( q^{-m(j(i+1)-a-s)}-q^{m(ji-a-s)}\right ).\notag
\end{align}}%

We want to write (\ref{equation:preHeckeTypeDoubleSum}) as the sum of two double-sums of the form (\ref{equation:fabc-def2}).  In the third and fourth sums, we make the substitutions $i\to -i-1$ and $t\to -t-1$.    The first and third double-sums pair together, the second and fourth sums pair together, and we have our result.
\end{proof}


\subsection{The quasi-periodicity approach for admissible-level, odd-spin}\label{section:genEulerTheoremsOddSpin}
We first change the notation of Theorem \ref{theorem:generalQuasiPeriodicityOddSpin}:

\begin{corollary}\label{corollary:quasiPeriodicityOddSpinMathCal}
For $(p,p^{\prime})=(p,2p+j)$, we have the quasi-periodic relation for odd-spin 
{\allowdisplaybreaks \begin{align*}
& (q)_{\infty}^3 \mathcal{C}_{2jt+2s+1,2r+1}^{(p,2p+j)}(q)
 -q^{pjt^2+pt(2s+1)}(q)_{\infty}^3\mathcal{C}_{2s+1,2r+1}^{(p,2p+j)}(q)\\
& \ = (-1)^{p}q^{pjt^2+pt(2s+1)+\binom{p}{2}-p(r-s)}
\sum_{i=1}^{t}q^{-2pj\binom{i}{2}-p(2s+1)i}\\
&\quad \times \sum_{m=1}^{p-1}(-1)^{m}q^{\binom{m+1}{2}+m(r-p)}
\left (q^{m(ji+s-j+1)} -q^{-m(ji+s)}\right ) \\
&\quad  \times\Big (  j(-q^{m(2p+j)+p(2r+2)};q^{2p(2p+j)}) 
 -q^{m(2p+j)-m(2r+2)}j(-q^{-m(2p+j)+p(2r+2)};q^{2p(2p+j)})\Big ).
\end{align*}}
\end{corollary}

\begin{proof}[Proof of Corollary \ref{corollary:quasiPeriodicityOddSpinMathCal}]
We recall (\ref{equation:mathCalCtoStringC}): 
\begin{equation*}
\mathcal{C}_{m,\ell}^{N}(q) := q^{-s_{\lambda}}C_{m,\ell}^{N}(q), \ 
s_{\lambda} := -\frac{1}{8}+\frac{(\ell+1)^2}{4(N+2)} - \frac{m^2}{4N},
\end{equation*}
to get
\begin{align*}
&C_{2jt+2s+1,2r+1}^{(p,2p+j)}(q)
 -C_{2s+1,2r+1}^{(p,2p+j)}(q)\\
&\qquad =q^{-\frac{1}{8}+\frac{p(2r+2)^2}{4(2p+j)}-\frac{p(2jt+2s+1)^2}{4j}} \mathcal{C}_{2jt+2s+1,2r+1}^{(p,2p+j)}(q)
 -q^{-\frac{1}{8}+\frac{p(2r+2)^2}{4(2p+j)}-\frac{p(2s+1)^2}{4j}}\mathcal{C}_{2s+1,2r+1}^{(p,2p+j)}(q).
\end{align*}

\noindent So our quasi-periodicity relation can be rewritten
{\allowdisplaybreaks \begin{align*}
& q^{-\frac{1}{8}+\frac{p(2r+2)^2}{4(2p+j)}-\frac{p(2jt+2s+1)^2}{4j}}(q)_{\infty}^3 \mathcal{C}_{2jt+2s+1,2r+1}^{(p,2p+j)}(q)
 -q^{-\frac{1}{8}+\frac{p(2r+2)^2}{4(2p+j)}-\frac{p(2s+1)^2}{4j}}(q)_{\infty}^3\mathcal{C}_{2s+1,2r+1}^{(p,2p+j)}(q)\\
& \ = (-1)^{p}q^{-\frac{1}{8}+\frac{p(2r+2)^2}{4(2p+j)}}q^{\binom{p+1}{2}-p(r+1-s)-\frac{p}{4j}(2s+1)^2}
\sum_{i=1}^{t}q^{-2pj\binom{i}{2}-p(2s+1)i}\\
&\quad \times \sum_{m=1}^{p-1}(-1)^{m}q^{\binom{m+1}{2}+m(r-p)}
\left (q^{m(ji+s-j+1)} -q^{-m(ji+s)}\right ) \\
&\quad  \times\Big (  j(-q^{m(2p+j)+p(2r+2)};q^{2p(2p+j)}) 
 -q^{m(2p+j)-m(2r+2)}j(-q^{-m(2p+j)+p(2r+2)};q^{2p(2p+j)})\Big ).
\end{align*}}

Multiplying through by the appropriate power of $q$ gives the result.
\end{proof}

\begin{proof}[Proof of Theorem \ref{theorem:genEulerIdentityOddSpin}]  For $\eta=2s+1$, $\ell=2r+1$, $L=jt+a$, we have
\begin{equation*}
 C_{2L+\eta,\ell}^{(p,p^{\prime})}(q)=C_{2(jt+a)+2s+1,2r+1}^{N}(q)
 = C_{2jt+2(a+s)+1,2r+1}^{N}(q).
\end{equation*}
Inserting into the generalized Euler identity (\ref{equation:genEulerIdentitySW}) yields
{\allowdisplaybreaks \begin{align*}
(q)_{\infty}^{3}\delta_{2r+1,2s+1}
&= (q)_{\infty}^{3} \sum_{L=-\infty}^{\infty}(-1)^{L}q^{\binom{L}{2}}\mathcal{C}_{2L+2s+1,2r+1}^{(p,2p+j)}(q)\\
&=(q)_{\infty}^{3}\sum_{a=0}^{j-1} \sum_{t=-\infty}^{\infty}(-1)^{jt+a}q^{\binom{jt+a}{2}}\mathcal{C}_{2(jt+a)+2s+1,2r+1}^{(p,2p+j)}(q)\\
&=(q)_{\infty}^{3}\sum_{a=0}^{j-1} \sum_{t=-\infty}^{\infty}(-1)^{jt+a}q^{\binom{jt+a}{2}}\mathcal{C}_{2jt+2(a+s)+1,2r+1}^{(p,2p+j)}(q).
\end{align*}}%

\noindent We then use Corollary \ref{corollary:quasiPeriodicityOddSpinMathCal}:
{\allowdisplaybreaks \begin{align*}
&(q)_{\infty}^{3}\delta_{2r+1,2s+1}\\
&\quad =(q)_{\infty}^{3}\sum_{a=0}^{j-1} \sum_{t=-\infty}^{\infty}(-1)^{jt+a}q^{\binom{jt+a}{2}}
q^{pjt^2+pt(2(a+s)+1)}\mathcal{C}_{2(a+s)+1,2r+1}^{(p,2p+j)}(q)\\
&\qquad 
+\sum_{a=0}^{j-1} \sum_{t=-\infty}^{\infty}(-1)^{jt+a}q^{\binom{jt+a}{2}}
 (-1)^{p}q^{pjt^2+pt(2(a+s)+1)+\binom{p}{2}-p(r-a-s)}\\
&\qquad \times \sum_{i=1}^{t}q^{-2pj\binom{i}{2}-p(2(a+s)+1)i}
 \times \sum_{m=1}^{p-1}(-1)^{m}q^{\binom{m+1}{2}+m(r-p)}
\left (q^{m(ji+a+s-j+1)} -q^{-m(ji+a+s)}\right ) \\
&\qquad  \times\Big (  j(-q^{m(2p+j)+p(2r+2)};q^{2p(2p+j)}) \\
&\qquad \qquad  -q^{m(2p+j)-m(2r+2)}j(-q^{-m(2p+j)+p(2r+2)};q^{2p(2p+j)})\Big ).
\end{align*}}%

\noindent We rewrite this as
{\allowdisplaybreaks \begin{align}
(q)_{\infty}^{3}\delta_{2r+1,2s+1}
& =(q)_{\infty}^{3}\sum_{a=0}^{j-1} (-1)^{a}q^{\binom{a}{2}}j(-(-1)^jq^{j(a+p)+\binom{j}{2}+p(2(a+s)+1)})
\mathcal{C}_{2(a+s)+1,2r+1}^{(p,2p+j)}(q)
\label{equation:preLastTwoLinesOdd}\\
&\qquad 
+(-1)^{p}q^{\binom{p}{2}-p(r-s)}\sum_{a=0}^{j-1}(-1)^{a} q^{pa}
 \sum_{m=1}^{p-1}(-1)^{m}q^{\binom{m+1}{2}+m(r-p)}
 \notag\\
&\qquad  \times\Big (  j(-q^{m(2p+j)+p(2r+2)};q^{2p(2p+j)}) 
\notag\\
&\qquad \qquad  -q^{m(2p+j)-m(2r+2)}j(-q^{-m(2p+j)+p(2r+2)};q^{2p(2p+j)})\Big )
\notag\\
&\qquad \times\sum_{t=-\infty}^{\infty}(-1)^{jt}q^{\binom{jt+a}{2}}
 q^{pjt^2+pt(2(a+s)+1)}
 \notag\\
&\qquad \times \sum_{i=1}^{t}q^{-2pj\binom{i}{2}-p(2(a+s)+1)i}\left (q^{m(ji+a+s-j+1)} -q^{-m(ji+a+s)}\right ). 
\notag
\end{align}}

We focus on the last two lines of (\ref{equation:preLastTwoLinesOdd}), but we first rewrite the exponents.
{\allowdisplaybreaks\begin{align*}
\sum_{t=-\infty}^{\infty}&(-1)^{jt}q^{\binom{jt+a}{2}}
 q^{pjt^2+pt(2(a+s)+1)}
  \sum_{i=1}^{t}q^{-2pj\binom{i}{2}-p(2(a+s)+1)i}\left (q^{m(ji+a+s-j+1)} -q^{-m(ji+a+s)}\right ) \\
  &=\sum_{t=-\infty}^{\infty}(-1)^{jt}q^{\binom{a}{2}+j(2p+j)\binom{t}{2}+t(j(a+p)+\binom{j}{2}+p(2(a+s)+1))}\\
  &\qquad \times   \sum_{i=1}^{t}q^{-2pj\binom{i}{2}-p(2(a+s)+1)i}\left (q^{m(ji+a+s-j+1)} -q^{-m(ji+a+s)}\right )=:\Psi_{a,m}^{s}(q).
\end{align*}}

For $t=0$, the inner sum vanishes by our summation convention, so we can write
{\allowdisplaybreaks \begin{align*}
\Psi_{a,m}^{s}(q)&=q^{\binom{a}{2}}
\sum_{t\ge 1}\sum_{i=1}^{t}(-1)^{jt}q^{j(2p+j)\binom{t}{2}+t(j(a+p)+\binom{j}{2}+p(2(a+s)+1))}\\
  &\qquad \qquad \times   q^{-2pj\binom{i}{2}-p(2(a+s)+1)i}\left (q^{m(ji+a+s-j+1)} -q^{-m(ji+a+s)}\right )\\
&\qquad + q^{\binom{a}{2}}
\sum_{t\le -1} \sum_{i=1}^{t}(-1)^{jt}q^{j(2p+j)\binom{t}{2}+t(j(a+p)+\binom{j}{2}+p(2(a+s)+1))}\\
  &\qquad \qquad \times  q^{-2pj\binom{i}{2}-p(2(a+s)+1)i}\left (q^{m(ji+a+s-j+1)} -q^{-m(ji+a+s)}\right ).
\end{align*}}

Using the summation convention again, we can write the second double-sum to get
{\allowdisplaybreaks \begin{align*}
\Psi_{a,m}^{s}(q)&=q^{\binom{a}{2}}
\sum_{t\ge 1}\sum_{i=1}^{t}(-1)^{jt}q^{j(2p+j)\binom{t}{2}+t(j(a+p)+\binom{j}{2}+p(2(a+s)+1))}\\
  &\qquad \qquad \times   q^{-2pj\binom{i}{2}-p(2(a+s)+1)i}\left (q^{m(ji+a+s-j+1)} -q^{-m(ji+a+s)}\right )\\
&\qquad - q^{\binom{a}{2}}
\sum_{t\le -1} \sum_{i=t+1}^{0}(-1)^{jt}q^{j(2p+j)\binom{t}{2}+t(j(a+p)+\binom{j}{2}+p(2(a+s)+1))}\\
  &\qquad \qquad \times  q^{-2pj\binom{i}{2}-p(2(a+s)+1)i}\left (q^{m(ji+a+s-j+1)} -q^{-m(ji+a+s)}\right ).
\end{align*}}

In the second double-sum  we make the substitutions $t\to -t$, $i\to -i+1$:
{\allowdisplaybreaks \begin{align*}
\Psi_{a,m}^{s}(q)&=q^{\binom{a}{2}}
\sum_{t\ge 1}\sum_{i=1}^{t}(-1)^{jt}q^{j(2p+j)\binom{t}{2}+t(j(a+p)+\binom{j}{2}+p(2(a+s)+1))}\\
  &\qquad \qquad \times   q^{-2pj\binom{i}{2}-p(2(a+s)+1)i}\left (q^{m(ji+a+s-j+1)} -q^{-m(ji+a+s)}\right )\\
&\qquad - q^{\binom{a}{2}}
\sum_{t\ge 1} \sum_{i=1}^{t}(-1)^{jt}q^{j(2p+j)\binom{-t}{2}-t(j(a+p)+\binom{j}{2}+p(2(a+s)+1))}\\
  &\qquad \qquad \times  q^{-2pj\binom{-i+1}{2}-p(2(a+s)+1)(-i+1)}\left (q^{m(j(-i+1)+a+s-j+1)} -q^{-m(j(-i+1)+a+s)}\right ).
\end{align*}}

Simplifying yields
{\allowdisplaybreaks \begin{align*}
\Psi_{a,m}^{s}(q)&=q^{\binom{a}{2}}
\sum_{t\ge 1}\sum_{i=1}^{t}(-1)^{jt}q^{j(2p+j)\binom{t}{2}+t(j(a+p)+\binom{j}{2}+p(2(a+s)+1))}\\
  &\qquad \qquad \times   q^{-2pj\binom{i}{2}-p(2(a+s)+1)i}\left (q^{m(ji+a+s-j+1)} -q^{-m(ji+a+s)}\right )\\
&\qquad - q^{\binom{a}{2}}
\sum_{t\ge 1} \sum_{i=1}^{t}(-1)^{jt}q^{j(2p+j)\binom{t+1}{2}-t(j(a+p)+\binom{j}{2}+p(2(a+s)+1))}\\
  &\qquad \qquad \times  q^{-2pj\binom{i}{2}+p(2(a+s)+1)(i-1)}
  \left (q^{-m(j(i-1)-a-s+j-1)} -q^{m(j(i-1)-a-s)}\right ).
\end{align*}}

Replacing $i$ with $i+1$ and $t$ with $t+1$ gives
{\allowdisplaybreaks \begin{align*}
\Psi_{a,m}^{s}(q)&=q^{\binom{a}{2}}
\sum_{t\ge 0}\sum_{i=0}^{t}(-1)^{j(t+1)}q^{j(2p+j)\binom{t+1}{2}+(t+1)(j(a+p)+\binom{j}{2}+p(2(a+s)+1))}\\
  &\qquad \qquad \times   q^{-2pj\binom{i+1}{2}-p(2(a+s)+1)(i+1)}\left (q^{m(ji+a+s+1)} -q^{-m(j(i+1)+a+s)}\right )\\
&\qquad - q^{\binom{a}{2}}
\sum_{t\ge 0} \sum_{i=0}^{t}(-1)^{j(t+1)}q^{j(2p+j)\binom{t+2}{2}-(t+1)(j(a+p)+\binom{j}{2}+p(2(a+s)+1))}\\
  &\qquad \qquad \times  q^{-2pj\binom{i+1}{2}+p(2(a+s)+1)i}
  \left (q^{-m(j(i+1)-a-s-1)} -q^{m(ji-a-s)}\right ).
\end{align*}}

We interchange summation symbols to obtain
{\allowdisplaybreaks \begin{align*}
\Psi_{a,m}^{s}(q)&=q^{\binom{a}{2}}
\sum_{i=0}^{t}\sum_{t\ge 0}(-1)^{j(t+1)}q^{j(2p+j)\binom{t+1}{2}+(t+1)(j(a+p)+\binom{j}{2}+p(2(a+s)+1))}\\
  &\qquad \qquad \times   q^{-2pj\binom{i+1}{2}-p(2(a+s)+1)(i+1)}\left (q^{m(ji+a+s+1)} -q^{-m(j(i+1)+a+s)}\right )\\
&\qquad - q^{\binom{a}{2}}
 \sum_{i=0}^{t}\sum_{t\ge 0}(-1)^{j(t+1)}q^{j(2p+j)\binom{t+2}{2}-(t+1)(j(a+p)+\binom{j}{2}+p(2(a+s)+1))}\\
  &\qquad \qquad \times  q^{-2pj\binom{i+1}{2}+p(2(a+s)+1)i}
  \left (q^{-m(j(i+1)-a-s-1)} -q^{m(ji-a-s)}\right ).
\end{align*}}
Replacing $t$ with $i+t$ then brings us to a familiar point:
{\allowdisplaybreaks \begin{align}
\Psi_{a,m}^{s}(q)&=q^{\binom{a}{2}}
\sum_{i=0}^{\infty}\sum_{t= 0}^{\infty}(-1)^{j(t+i+1)}q^{j(2p+j)\binom{t+i+1}{2}+(t+i+1)(j(a+p)+\binom{j}{2}+p(2(a+s)+1))}
\label{equation:preHeckeTypeDoubleSumOdd}\\
  &\qquad \qquad \times   q^{-2pj\binom{i+1}{2}-p(2(a+s)+1)(i+1)}\left (q^{m(ji+a+s+1)} -q^{-m(j(i+1)+a+s)}\right )
  \notag\\
&\qquad - q^{\binom{a}{2}}
 \sum_{i=0}^{\infty}\sum_{t= 0}^{\infty}(-1)^{j(t+i+1)}q^{j(2p+j)\binom{t+i+2}{2}-(t+i+1)(j(a+p)+\binom{j}{2}+p(2(a+s)+1))}
 \notag\\
  &\qquad \qquad \times  q^{-2pj\binom{i+1}{2}+p(2(a+s)+1)i}
  \left (q^{-m(j(i+1)-a-s-1)} -q^{m(ji-a-s)}\right ).
  \notag
\end{align}}

We want to write (\ref{equation:preHeckeTypeDoubleSumOdd}) as the sum of two double-sums of the form (\ref{equation:fabc-def2}).  In the third and fourth sums, we make the substitutions $i\to -i-1$ and $t\to -t-1$.    The first and third double-sums pair together, the second and fourth sums pair together, and we are done.
\end{proof}


\section{On a generalized Euler identity for modified string functions: Corollaries}\label{section:genEulerCorollaries}

In this section we obtain the corollaries to the modified generalized Euler identities found in Theorems \ref{theorem:genEulerIdentityEvenSpin} and Theorems \ref{theorem:genEulerIdentityOddSpin}.  In Section \ref{section:genEulerCorollariesEvenSpin} we derive the Corollaries \ref{corollary:levelOneHalfEvenSpin}, \ref{corollary:levelOneThirdEvenSpin}, and \ref{corollary:levelTwoThirdsEvenSpin} for even-spin string functions.   In Section \ref{section:genEulerCorollariesOddSpin} we derive the Corollaries \ref{corollary:levelOneHalfOddSpin}, \ref{corollary:levelOneThirdOddSpin}, and \ref{corollary:levelTwoThirdsOddSpin} for odd-spin string functions.

\subsection{Even-spin string functions with admissible-levels: $1/2$, $1/3$, and $2/3$}\label{section:genEulerCorollariesEvenSpin}
\begin{proof}[Proof of Corollary \ref{corollary:levelOneHalfEvenSpin}]
We specialize Corollary \ref{corollary:genEulerIdentityEvenSpinCorollary} to $p=2$.  This yields
\begin{align*}
(q)_{\infty}^{3}\delta_{2r,2s}
&=(q)_{\infty}^{3} j(q^{2+4s};q^{5})
\mathcal{C}_{2s,2r}^{(2,5)}(q)\\
&\qquad    -q^{-r+2s}
\Big (  j(-q^{7+4r};q^{20} )  -  q^{4-2r}j(-q^{-3+4r};q^{20})\Big )\\
&\qquad   \times  \Big ( -q^{2+s} f_{5,5,1}(q^{4s+7},q^{4};q) +q^{1-s} f_{5,5,1}(q^{4s+7},q^{2};q)\Big ).
\end{align*}

We combine the two theta functions.   We first use (\ref{equation:j-flip}) and then (\ref{equation:j-elliptic})
\begin{align*}
j(-q^{7+4r};q^{20} )  -  q^{4-2r}j(-q^{-3+4r};q^{20})
&=j(-q^{7+4r};q^{20} )  -  q^{4-2r}j(-q^{23-4r};q^{20})\\
&=j(-q^{7+4r};q^{20} )  -  q^{1+2r}j(-q^{3-4r};q^{20}).
\end{align*}
We again use (\ref{equation:j-flip}) and then (\ref{equation:j-split-m2}) to get
\begin{align*}
j(-q^{7+4r};q^{20} )  -  q^{4-2r}j(-q^{-3+4r};q^{20})
&=j(-q^{7+4r};q^{20} )  -  q^{1+2r}j(-q^{17+4r};q^{20})\\
&=j(q^{1+2r};q^5).
\end{align*}
Hence
\begin{align*}
(q)_{\infty}^{3}\delta_{2r,2s}
&=(q)_{\infty}^{3} j(q^{2+4s};q^{5})
\mathcal{C}_{2s,2r}^{(2,5)}(q)\\
&\qquad    -q^{-r+2s} j(q^{1+2r};q^5)  \times  \Big ( -q^{2+s} f_{5,5,1}(q^{4s+7},q^{4};q) +q^{1-s} f_{5,5,1}(q^{4s+7},q^{2};q)\Big ).
\end{align*}
If we redistribute the exponents in the second line, we arrive at the desired result.
\end{proof}

\begin{proof}[Proof of Corollary \ref{corollary:levelOneThirdEvenSpin} ]
We specialize Corollary \ref{corollary:genEulerIdentityEvenSpinCorollary} to $p=3$.  This produces
{\allowdisplaybreaks \begin{align*}
(q)_{\infty}^{3}\delta_{2r,2s}
&=(q)_{\infty}^{3}j(q^{3+6s};q^{7})
\mathcal{C}_{2s,2r}^{(3,7)}(q)\\
&\quad 
-q^{3-3r+3s} \sum_{m=1}^{2}(-1)^{m}q^{\binom{m+1}{2}+m(r-3)}\\
&\qquad   \times 
\Big (  j(-q^{7m+6r+3};q^{42} )
  -  q^{7m-m(2r+1)}j(-q^{-7m+6r+3};q^{42})\Big )\\
&\qquad \times \Big ( -q^{3+ms}
 f_{7,7,1}(q^{6s+10},q^{m+4};q) +q^{3-m(s+1)} f_{7,7,1}(q^{6s+10},q^{-m+4};q)\Big ) .
\end{align*}}%
We combine the theta functions.  For $m=1$, we first use (\ref{equation:j-flip}) to get
\begin{align*}
 j(-q^{10+6r};q^{42} ) &-  q^{6-2r}j(-q^{-4+6r};q^{42})\\
 &= j(-q^{32-6r};q^{42} ) -  q^{6-2r}j(-q^{46-6r};q^{42})\\
 &= j(-q^{14+3(6-2r)};q^{42} ) -  q^{6-2r}j(-q^{28+3(6-2r)};q^{42}).
 \end{align*}
 Using the quintuple product identity (\ref{equation:H1Thm1.0}) and again (\ref{equation:j-flip}) brings us to
 \begin{align*}
  j(-q^{10+6r};q^{42} ) -  q^{6-2r}j(-q^{-4+6r};q^{42})
 &=\frac{j(q^{6-2r};q^{14})j(q^{14+2(6-2r)};q^{28})}{J_{28}}\\
 &=\frac{j(q^{8+2r};q^{14})j(q^{2+4r};q^{28})}{J_{28}}.
\end{align*}
For $m=2$, we argue in a similar manner to obtain
\begin{align*}
j(-q^{17+6r};q^{42})-q^{12-4r}j(-q^{-11+6r};q^{42})=\frac{j(q^{1+2r};q^{14})j(q^{16+4r};q^{28})}{J_{28}}.
\end{align*}
Hence
{\allowdisplaybreaks \begin{align*}
(q)_{\infty}^{3}\delta_{2r,2s}
&=(q)_{\infty}^{3}j(q^{3+6s};q^{7})
\mathcal{C}_{2s,2r}^{(3,7)}(q)\\
&\qquad 
-q^{3-3r+3s} \left ( -q^{r-2}\right )  \frac{j(q^{8+2r};q^{14})j(q^{2+4r};q^{28})}{J_{28}}\\
&\qquad \qquad \times \Big ( -q^{3+s}
 f_{7,7,1}(q^{6s+10},q^{5};q) +q^{2-s} f_{7,7,1}(q^{6s+10},q^{3};q)\Big )\\
&\qquad  -q^{3-3r+3s} \left ( q^{2r-3}\right ) 
\frac{j(q^{1+2r};q^{14})j(q^{16+4r};q^{28})}{J_{28}}\\
&\qquad \qquad \times \Big ( -q^{3+2s}
 f_{7,7,1}(q^{6s+10},q^{6};q) +q^{1-2s} f_{7,7,1}(q^{6s+10},q^{2};q)\Big ).
\end{align*}}%
Redistributing the exponents gives the desired result.
\end{proof}

\begin{proof}[Proof of Corollary \ref{corollary:levelTwoThirdsEvenSpin}]
We specialize Theorem  \ref{theorem:genEulerIdentityEvenSpin} to $p=3$, $j=2$.  This gives
{\allowdisplaybreaks \begin{align*}
(q)_{\infty}^{3}\delta_{2r,2s}
&=(q)_{\infty}^{3}\sum_{a=0}^{1}(-1)^{a}  j(-q^{7+8a+6s};q^{16})
\mathcal{C}_{2(a+s),2r}^{(3,8)}(q)\\
&\qquad 
-q^{3-3r+3s} \sum_{a=0}^{1} (-1)^{a}q^{3a}\sum_{m=1}^{2}(-1)^{m}q^{\binom{m+1}{2}+m(r-3)}\\
&\qquad   \times 
\Big (  j(-q^{8m+3(2r+1)};q^{48} )   -  q^{8m-m(2r+1)}j(-q^{-8m+3(2r+1)};q^{48})\Big )\\
&\qquad   \times\Big ( q^{7+2a+m(a+s)}
 f_{8,8,2}(-q^{8a+6s+23},-q^{11+2(a+m)};q^2)\\
&\qquad \qquad - q^{7+2a-m(a+s+2)} f_{8,8,2}(-q^{8a+6s+23},-q^{11+2(a-m)};q^2)\Big ). 
\end{align*}}%
We expand the sum over $m$.  This brings us to 
{\allowdisplaybreaks \begin{align*}
&(q)_{\infty}^{3}\delta_{2r,2s}\\
&\quad =(q)_{\infty}^{3}\sum_{a=0}^{1}(-1)^{a}  j(-q^{7+8a+6s};q^{16})
\mathcal{C}_{2(a+s),2r}^{(3,8)}(q)\\
&\qquad 
-q^{3-3r+3s} \sum_{a=0}^{1} (-1)^{a}q^{3a}\\
&\qquad \times \Big [ -q^{r-2}
\Big (  j(-q^{11+6r};q^{48} )   -  q^{7-2r}j(-q^{-5+6r};q^{48})\Big )\\
&\qquad  \qquad  \times\Big ( q^{7+3a+s}
 f_{8,8,2}(-q^{8a+6s+23},-q^{13+2a};q^2)
 - q^{5+a-s} f_{8,8,2}(-q^{8a+6s+23},-q^{9+2a};q^2)\Big ) \\
&\qquad +q^{2r-3}
\Big (  j(-q^{19+6r};q^{48} )   -  q^{14-4r}j(-q^{-13+6r};q^{48})\Big )\\
&\qquad   \qquad \times\Big ( q^{7+4a+2s}
 f_{8,8,2}(-q^{8a+6s+23},-q^{15+2a};q^2)
 - q^{3-2s} f_{8,8,2}(-q^{8a+6s+23},-q^{7+2a};q^2)\Big ) \Big ]. 
\end{align*}}%
Once again, we combine the theta functions.  For the first sum of theta functions, we use (\ref{equation:j-flip}), (\ref{equation:H1Thm1.0}), and (\ref{equation:j-flip}) to get
{\allowdisplaybreaks \begin{align*}
j(-q^{11+6r};q^{48} )   -  q^{7-2r}j(-q^{-5+6r};q^{48})
&= j(-q^{37-6r};q^{48} )   -  q^{7-2r}j(-q^{53-6r};q^{48})\\
&=\frac{j(q^{7-2r};q^{16})j(q^{30-4r};q^{32})}{J_{32}}\\
&=\frac{j(q^{9+2r};q^{16})j(q^{2+4r};q^{32})}{J_{32}}.
\end{align*}}%
For the second sum, we argue in a similar manner to get
\begin{equation*}
j(-q^{19+6r};q^{48} )   -  q^{14-4r}j(-q^{-13+6r};q^{48})
= \frac{j(q^{1+2r};q^{16})j(q^{18+4r};q^{32})}{J_{32}}.
\end{equation*}
Substituting back in gives
{\allowdisplaybreaks \begin{align*}
&(q)_{\infty}^{3}\delta_{2r,2s}\\
&\quad =(q)_{\infty}^{3}\sum_{a=0}^{1}(-1)^{a}  j(-q^{7+8a+6s};q^{16})
\mathcal{C}_{2(a+s),2r}^{(3,8)}(q)\\
&\qquad 
-q^{3-3r+3s} \sum_{a=0}^{1} (-1)^{a}q^{3a}\\
&\qquad \quad \times \Big [ -q^{r-2}
\frac{j(q^{9+2r};q^{16})j(q^{2+4r};q^{32})}{J_{32}}\\
&\qquad \qquad  \times\Big ( q^{7+3a+s}
 f_{8,8,2}(-q^{8a+6s+23},-q^{13+2a};q^2)
  - q^{5+a-s} f_{8,8,2}(-q^{8a+6s+23},-q^{9+2a};q^2)\Big ) \\
&\qquad \quad +q^{2r-3}
\frac{j(q^{1+2r};q^{16})j(q^{18+4r};q^{32})}{J_{32}}\\
&\qquad \qquad   \times\Big ( q^{7+4a+2s}
 f_{8,8,2}(-q^{8a+6s+23},-q^{15+2a};q^2)
  - q^{3-2s} f_{8,8,2}(-q^{8a+6s+23},-q^{7+2a};q^2)\Big ) \Big ]. 
\end{align*}}%
We expand the first sum over $a$ and distribute the second sum over $a$.  This brings us to
{\allowdisplaybreaks \begin{align*}
(q)_{\infty}^{3}\delta_{2r,2s}
& =(q)_{\infty}^{3}  j(-q^{7+6s};q^{16})
\mathcal{C}_{2s,2r}^{(3,8)}(q)
-(q)_{\infty}^{3} j(-q^{15+6s};q^{16})
\mathcal{C}_{2(1+s),2r}^{(3,8)}(q)\\
&\qquad 
-q^{3-3r+3s} 
 \times \Big [ -q^{r-2}
\frac{j(q^{9+2r};q^{16})j(q^{2+4r};q^{32})}{J_{32}}\\
&\qquad   \times\sum_{a=0}^{1} (-1)^{a}q^{3a}\Big ( q^{7+3a+s}
 f_{8,8,2}(-q^{8a+6s+23},-q^{13+2a};q^2)\\
&\qquad \qquad  - q^{5+a-s} f_{8,8,2}(-q^{8a+6s+23},-q^{9+2a};q^2)\Big ) \\
&\qquad \quad  +q^{2r-3}
\frac{j(q^{1+2r};q^{16})j(q^{18+4r};q^{32})}{J_{32}}\\
&\qquad   \times\sum_{a=0}^{1} (-1)^{a}q^{3a}\Big ( q^{7+4a+2s}
 f_{8,8,2}(-q^{8a+6s+23},-q^{15+2a};q^2)\\
&\qquad \qquad  - q^{3-2s} f_{8,8,2}(-q^{8a+6s+23},-q^{7+2a};q^2)\Big ) \Big ]. 
\end{align*}}%
We continue to expand the sums over $a$:
{\allowdisplaybreaks \begin{align*}
(q)_{\infty}^{3}\delta_{2r,2s}
&=(q)_{\infty}^{3}  j(-q^{7+6s};q^{16})
\mathcal{C}_{2s,2r}^{(3,8)}(q)
-(q)_{\infty}^{3} j(-q^{15+6s};q^{16})
\mathcal{C}_{2(1+s),2r}^{(3,8)}(q)\\
&\qquad 
-q^{3-3r+3s} 
 \times \Big [ -q^{r-2}
\frac{j(q^{9+2r};q^{16})j(q^{2+4r};q^{32})}{J_{32}}\\
&\qquad   \times\Big ( q^{7+s}
 f_{8,8,2}(-q^{6s+23},-q^{13};q^2)
   - q^{5-s} f_{8,8,2}(-q^{6s+23},-q^{9};q^2)\\
&\qquad \qquad -q^{13+s}
 f_{8,8,2}(-q^{31+6s},-q^{15};q^2)
  + q^{9-s} f_{8,8,2}(-q^{31+6s},-q^{11};q^2)\Big ) \\
&\qquad +q^{2r-3}
\frac{j(q^{1+2r};q^{16})j(q^{18+4r};q^{32})}{J_{32}}\\
&\qquad   \times\Big ( q^{7+2s}
 f_{8,8,2}(-q^{6s+23},-q^{15};q^2)
  - q^{3-2s} f_{8,8,2}(-q^{6s+23},-q^{7};q^2)\\
&\qquad \qquad - q^{14+2s}
 f_{8,8,2}(-q^{31+6s},-q^{17};q^2)
   + q^{6-2s} f_{8,8,2}(-q^{31+6s},-q^{9};q^2)
\Big ) \Big ].
\end{align*}}%
Redistributing the exponents and noting that
\begin{equation*}
f_{8,8,2}(x,y;q)=f_{4,4,1}(x,y;q^2)
\end{equation*}
brings us to the final form.
\end{proof}


\subsection{Odd-spin string functions with admissible-levels: $1/2$, $1/3$, and $2/3$}\label{section:genEulerCorollariesOddSpin}
\begin{proof}[Proof of Corollary {\ref{corollary:levelOneHalfOddSpin}}] We specialize Corollary \ref{corollary:genEulerIdentityOddSpinCorollary} to $p=2$.  This gives
{\allowdisplaybreaks \begin{align*}
(q)_{\infty}^{3}\delta_{2r+1,2s+1}
& =(q)_{\infty}^{3}j(q^{4+4s};q^{5})
\mathcal{C}_{2s+1,2r+1}^{(2,5)}(q)\\
&\qquad 
-q^{-r}\Big (  j(-q^{9+4r};q^{20}) 
  -q^{3-2r}j(-q^{-1+4r};q^{20})\Big )\\
&\qquad \qquad  \times \Big (  -q^{3+3s} f_{5,5,1}(q^{4s+9},q^{4};q)
 +q^{1+s} f_{5,5,1}(q^{4s+9},q^{2};q)\Big ).
\end{align*}}%
We combine the theta functions.  Using (\ref{equation:j-flip}) and then (\ref{equation:j-split-m2}) gives
\begin{align*}
j(-q^{9+4r};q^{20})   -q^{3-2r}j(-q^{-1+4r};q^{20})
&=j(-q^{11-4r};q^{20})   -q^{3-2r}j(-q^{21-4r};q^{20})\\
&=j(q^{3-2r};q^{5}).
\end{align*}
Using (\ref{equation:j-flip}) gives the result.
\end{proof}

\begin{proof}[Proof of Corollary \ref{corollary:levelOneThirdOddSpin}]  We specialize Corollary \ref{corollary:genEulerIdentityOddSpinCorollary} to $p=3$.  This gives
{\allowdisplaybreaks \begin{align*}
&(q)_{\infty}^{3}\delta_{2r+1,2s+1}\\
&\quad =(q)_{\infty}^{3}j(q^{6s+6};q^{7})
\mathcal{C}_{2s+1,2r+1}^{(3,7)}(q)\\
&\qquad 
-q^{3-3r+3s}
 \sum_{m=1}^{2}(-1)^{m}q^{\binom{m+1}{2}+m(r-3)}\\
&\qquad  \times\Big (  j(-q^{7m+6r+6};q^{42}) 
  -q^{5m-2mr}j(-q^{-7m+6r+6};q^{42})\Big )\\
&\qquad \times \Big (  -q^{3+m(s+1)} f_{7,7,1}(q^{6s+13},q^{m+4};q)
 +q^{3-m(s+1)} f_{7,7,1}(q^{6s+13},q^{-m+4};q)\Big ).
 \end{align*}}%
 Expanding the sum over $m$ gives
{\allowdisplaybreaks \begin{align*}
(q)_{\infty}^{3}\delta_{2r+1,2s+1}
 &=(q)_{\infty}^{3}j(q^{6s+6};q^{7})
\mathcal{C}_{2s+1,2r+1}^{(3,7)}(q)\\
&\qquad 
+q^{1-2r+3s}
\Big (  j(-q^{13+6r};q^{42}) -q^{5-2r}j(-q^{-1+6r};q^{42})\Big )\\
&\qquad \times \Big (  -q^{4+s} f_{7,7,1}(q^{6s+13},q^{5};q)
 +q^{2-s} f_{7,7,1}(q^{6s+13},q^{3};q)\Big ) \\
 &\qquad 
-q^{-r+3s}  \Big (  j(-q^{20+6r};q^{42})   -q^{10-4r}j(-q^{-8+6r};q^{42})\Big )\\
&\qquad \times \Big (  -q^{5+2s} f_{7,7,1}(q^{6s+13},q^{6};q)
 +q^{1-2s} f_{7,7,1}(q^{6s+13},q^{2};q)\Big ).
\end{align*}}%

We combine the theta functions.  Using (\ref{equation:j-flip}), the quintuple product identity (\ref{equation:H1Thm1.0}), and again (\ref{equation:j-flip}) brings us to
\begin{align*}
 j(-q^{13+6r};q^{42}) -q^{5-2r}j(-q^{-1+6r};q^{42})
&= j(-q^{29-6r};q^{42}) -q^{5-2r}j(-q^{43-6r};q^{42})\\
 &=\frac{j(q^{5-2r};q^{14})j(q^{14+2(5-2r)};q^{28})}{J_{28}}\\
 &=\frac{j(q^{9+2r};q^{14})j(q^{4+4r};q^{28})}{J_{28}}.
\end{align*}
Arguing in a similar manner brings us to 
\begin{equation*}
 j(-q^{20+6r};q^{42})   -q^{10-4r}j(-q^{-8+6r};q^{42})
=\frac{j(q^{2+2r};q^{14})j(q^{18+4r};q^{28})}{J_{28}}.
\end{equation*}
Substituting back into the original expansion yields
{\allowdisplaybreaks \begin{align*}
(q)_{\infty}^{3}\delta_{2r+1,2s+1}
&=(q)_{\infty}^{3}j(q^{6s+6};q^{7})
\mathcal{C}_{2s+1,2r+1}^{(3,7)}(q)\\
&\qquad 
+q^{1-2r+3s}
\frac{j(q^{9+2r};q^{14})j(q^{4+4r};q^{28})}{J_{28}}\\
&\qquad \qquad \times \Big (  -q^{4+s} f_{7,7,1}(q^{6s+13},q^{5};q)
 +q^{2-s} f_{7,7,1}(q^{6s+13},q^{3};q)\Big ) \\
 &\qquad 
-q^{-r+3s}  \frac{j(q^{2+2r};q^{14})j(q^{18+4r};q^{28})}{J_{28}}\\
&\qquad \qquad \times \Big (  -q^{5+2s} f_{7,7,1}(q^{6s+13},q^{6};q)
 +q^{1-2s} f_{7,7,1}(q^{6s+13},q^{2};q)\Big ).
\end{align*}}%
Redistributing the exponents gives the desired result.
\end{proof}

\begin{proof}[Proof of Corollary \ref{corollary:levelTwoThirdsOddSpin}]  We specialize Theorem \ref{theorem:genEulerIdentityOddSpin} to $p=3$, $j=2$.  This gives
{\allowdisplaybreaks \begin{align*}
(q)_{\infty}^{3}\delta_{2r+1,2s+1}
& =(q)_{\infty}^{3}\sum_{a=0}^{1} (-1)^{a}q^{\binom{a}{2}}j(-q^{8a+10+6s};q^{16})
\mathcal{C}_{2(a+s)+1,2r+1}^{(3,8)}(q)\\
&\qquad 
-q^{3-3r+3s}\sum_{a=0}^{1}(-1)^{a} q^{3a}
 \sum_{m=1}^{2}(-1)^{m}q^{\binom{m+1}{2}+m(r-3)}\\
&\qquad  \times\Big (  j(-q^{8m+6r+6};q^{48}) 
  -q^{6m-2mr}j(-q^{-8m+6r+6};q^{48})\Big )\\,
&\qquad  \times \Big ( q^{\binom{a}{2}+7+2a+m(a+s+1)}
 f_{4,4,1}(-q^{6s+8a+26},-q^{11+2a+2m};q^4)\\
&\qquad \qquad - q^{\binom{a}{2}+7+2a-m(a+s+2)}
 f_{4,4,1}(-q^{6s+8a+26},-q^{11+2a-2m};q^4)\Big ) .
\end{align*}}%
We expand the first sum over $a$, this gives
{\allowdisplaybreaks \begin{align*}
(q)_{\infty}^{3}\delta_{2r+1,2s+1}
& =(q)_{\infty}^{3}j(-q^{10+6s};q^{16})
\mathcal{C}_{2s+1,2r+1}^{(3,8)}(q)
-(q)_{\infty}^{3}j(-q^{18+6s};q^{16})
\mathcal{C}_{2s+3,2r+1}^{(3,8)}(q)\\
&\qquad 
-q^{3-3r+3s}\sum_{a=0}^{1}(-1)^{a} q^{3a}
 \sum_{m=1}^{2}(-1)^{m}q^{\binom{m+1}{2}+m(r-3)}\\
&\qquad  \times\Big (  j(-q^{8m+6r+6};q^{48}) 
  -q^{6m-2mr}j(-q^{-8m+6r+6};q^{48})\Big )\\,
&\qquad  \times \Big ( q^{\binom{a}{2}+7+2a+m(a+s+1)}
 f_{4,4,1}(-q^{6s+8a+26},-q^{11+2a+2m};q^4)\\
&\qquad \qquad - q^{\binom{a}{2}+7+2a-m(a+s+2)}
 f_{4,4,1}(-q^{6s+8a+26},-q^{11+2a-2m};q^4)\Big ) .
\end{align*}}%
Now we expand the sum over $m$.  This yields
{\allowdisplaybreaks \begin{align*}
(q)_{\infty}^{3}\delta_{2r+1,2s+1}
& =(q)_{\infty}^{3}j(-q^{10+6s};q^{16})
\mathcal{C}_{2s+1,2r+1}^{(3,8)}(q)
-(q)_{\infty}^{3}j(-q^{18+6s};q^{16})
\mathcal{C}_{2s+3,2r+1}^{(3,8)}(q)\\
&\qquad + q^{1-2r}
\Big (  j(-q^{14+6r};q^{48}) 
  -q^{6-2r}j(-q^{-2+6r};q^{48})\Big )\\,
&\qquad  \times \sum_{a=0}^{1}(-1)^{a} q^{3a}\Big ( q^{8+3a+4s}
 f_{4,4,1}(-q^{6s+8a+26},-q^{13+2a};q^4)\\
&\qquad \qquad - q^{5+a+2s}
 f_{4,4,1}(-q^{6s+8a+26},-q^{9+2a};q^4)\Big )\\
 &\qquad -q^{-r}
\Big (  j(-q^{22+6r};q^{48}) 
  -q^{12-4r}j(-q^{-10+6r};q^{48})\Big )\\,
&\qquad  \times \sum_{a=0}^{1}(-1)^{a} q^{3a}\Big ( q^{9+4a+5s}
 f_{4,4,1}(-q^{6s+8a+26},-q^{15+2a};q^4)\\
&\qquad \qquad - q^{3+s}
 f_{4,4,1}(-q^{6s+8a+26},-q^{7+2a};q^4)\Big ).
\end{align*}}%

Finally, we expand the remaining sums over $a$ bringing us to
{\allowdisplaybreaks \begin{align*}
&(q)_{\infty}^{3}\delta_{2r+1,2s+1}\\
&\quad =(q)_{\infty}^{3}j(-q^{10+6s};q^{16})
\mathcal{C}_{2s+1,2r+1}^{(3,8)}(q)
-(q)_{\infty}^{3}j(-q^{18+6s};q^{16})
\mathcal{C}_{2s+3,2r+1}^{(3,8)}(q)\\
&\qquad + q^{1-2r}
\Big (  j(-q^{14+6r};q^{48}) 
  -q^{6-2r}j(-q^{-2+6r};q^{48})\Big )\\,
&\qquad \quad  \times \Big [ q^{8+4s}
 f_{4,4,1}(-q^{6s+26},-q^{13};q^4)
  - q^{5+2s}
 f_{4,4,1}(-q^{6s+26},-q^{9};q^4)\\
 &\qquad  \qquad -q^3 \Big ( q^{11+4s}
 f_{4,4,1}(-q^{6s+34},-q^{15};q^4)
  - q^{6+2s}
 f_{4,4,1}(-q^{6s+34},-q^{11};q^4)\Big )\Big ] \\
 &\qquad -q^{-r}
\Big (  j(-q^{22+6r};q^{48}) 
  -q^{12-4r}j(-q^{-10+6r};q^{48})\Big )\\,
&\qquad \quad \times \Big [ q^{9+5s}
 f_{4,4,1}(-q^{6s+26},-q^{15};q^4)
 - q^{3+s}
 f_{4,4,1}(-q^{6s+26},-q^{7};q^4)\\
 &\qquad \qquad  -q^3 \Big ( q^{13+5s}
 f_{4,4,1}(-q^{6s+34},-q^{17};q^4)
  - q^{3+s}
 f_{4,4,1}(-q^{6s+34},-q^{9};q^4)\Big )\Big ] .
\end{align*}}%

We combine the theta functions.  Using (\ref{equation:j-flip}) and the quintuple product identity (\ref{equation:H1Thm1.0}) gives
{\allowdisplaybreaks \begin{align*}
j(-q^{14+6r};q^{48}) & -q^{6-2r}j(-q^{-2+6r};q^{48})\\
&= j(-q^{34-6r};q^{48})  -q^{6-2r}j(-q^{50-6r};q^{48})\\
&= j(-q^{16+3(6-2r)};q^{48})  -q^{6-2r}j(-q^{32+3(6-2r)};q^{48})\\
&=\frac{j(q^{6-2r};q^{16})j(q^{16+2(6-2r)};q^{32})}{J_{32}}.
\end{align*}}%
Using (\ref{equation:j-flip}) again brings us to 
\begin{align*}
j(-q^{14+6r};q^{48}) & -q^{6-2r}j(-q^{-2+6r};q^{48})=\frac{j(q^{10+2r};q^{16})j(q^{4+4r};q^{32})}{J_{32}}.
\end{align*}
Arguing in a similar manner brings us to
\begin{equation*}
 j(-q^{22+6r};q^{48})  -q^{12-4r}j(-q^{-10+6r};q^{48})=\frac{j(q^{2+2r};q^{16})j(q^{20+4r};q^{32})}{J_{32}}.
\end{equation*}
We then substitute the theta functions into the original expansion and simplify.
\end{proof}

\section{Mock theta function identities for double-sum coefficients from even-spin}\label{section:mockThetaTheoremsEvenSpin}
In this section we consider the generalized Euler identity in the case of even-spin.  We prove the general forms of the mock theta function identities found in Theorem \ref{theorem:genEulerOneHalfEvenSpin} for $1/2$-level, Theorems \ref{theorem:genEulerOneThirdFirstPairEvenSpin} and \ref{theorem:genEulerOneThirdSecondPairEvenSpin} for $1/3$-level, and Theorems \ref{theorem:level23EvenSpinFirstQuad-sEven}, \ref{theorem:level23EvenSpinFirstQuad-sOdd}, \ref{theorem:level23EvenSpinSecondQuad-sEven}, and \ref{theorem:level23EvenSpinSecondQuad-sOdd} for $2/3$-level.  The proofs in this section are all the same, but they do involve an increasing amount of book-keeping.

\subsection{The $1/2$-level string functions: Theorem \ref{theorem:genEulerOneHalfEvenSpin}}  We state the necessary propositions to prove Theorem \ref{theorem:genEulerOneHalfEvenSpin}, then we give the proof, and then we give the proofs of the propositions.  

\smallskip
Let us define the left-hand side of the identity in Theorem \ref{theorem:genEulerOneHalfEvenSpin} to be
\begin{equation*}
F(r):=-q^{2+3r}f_{5,5,1}(q^{4r+7},q^4;q)+q^{1+r}f_{5,5,1}(q^{4r+7},q^2;q).
\end{equation*}
and the right-hand side to be
{\allowdisplaybreaks  \begin{align*}
G(r)&:=(-1)^{r}j(q^{r+3};q^5)
\left (q^{\binom{r+1}{2}}\frac{1}{2}\mu_2(q)+ q^{-3\binom{r}{2}}\sum_{k=0}^{2r-1}( -1)^{k}q^{2rk-\binom{k+1}{2}}\right )\\
&\qquad -(-1)^r(-q)^{\binom{r+1}{2}}\frac{1}{2}\frac{J_{1}^3}{J_{2}J_{4}}j((-q)^{r+3};-q^5).
\end{align*}}%
 \begin{proposition}\label{proposition:level12evenSpinFuncEqn}  We have that both $F(r)$ and $G(r)$ satisfy the same functional equation:
\begin{align*}
&H(r+5) -q^{12+4r}H(r) \\
&\quad = -(-1)^{r}q^{10-3\binom{r}{2}-r}j(q^{r+3};q^{5})\left ( q^{2+3r}\sum_{m=0}^{4} q^{-2m^2-m-4mr}
 -q^{1+r}\sum_{m=0}^{4}q^{-2m^2-3m-4mr}\right ). 
\end{align*}
\end{proposition}

We evaluate the Appell function expression in Corollary \ref{corollary:f551-HeckeExpansion} for the respective double-sums in Theorem \ref{theorem:genEulerOneHalfEvenSpin}.

 \begin{proposition} \label{proposition:level12EvenSpinAppellForm} We have that
  \begin{align*}
 -q^{2+3s}& h_{5,5,1}(q^{4s+7},q^{4};q) +q^{1+s} h_{5,5,1}(q^{4s+7},q^{2};q)\\
 &= -q^{2+5s}j(q^{4s+7};q^5)\left (\sum_{k=0}^{2s-1}(-1)^{k}q^{2sk-\binom{k+1}{2}} 
 +2q^{s+4\binom{s}{2}}m ( -q,-1;q^{4}  )\right ).
 \end{align*}
 \end{proposition}

 \begin{proof}[Proof of Theorem \ref{theorem:genEulerOneHalfEvenSpin}]   Proposition \ref{proposition:level12evenSpinFuncEqn} tell us that 
 \begin{equation*}
 F(s+5)-G(s+5)=q^{12+4r}\left ( F(s)-G(s)\right ),
 \end{equation*}
 so we only need to prove Theorem \ref{theorem:genEulerOneHalfEvenSpin} for $0\le s\le 4$.  The case $s=2$ follows from Lemma \ref{lemma:degenerateDoubleSumCoeffs}.  From Corollary \ref{corollary:f551-HeckeExpansion}, we have that
\begin{align*}
 -q^{2+3s}& f_{5,5,1}(q^{4s+7},q^{4};q) +q^{1+s} f_{5,5,1}(q^{4s+7},q^{2};q)\\
&= -q^{2+3s} h_{5,5,1}(q^{4s+7},q^{4};q) +q^{1+s} h_{5,5,1}(q^{4s+7},q^{2};q)\\
&\qquad -\frac{1}{\overline{J}_{0,4}\overline{J}_{0,20}}
\left ( -q^{2+3s} \theta_{5,5,1}(q^{4s+7},q^{4};q) +q^{1+s} \theta_{5,5,1}(q^{4s+7},q^{2};q)\right ).
\end{align*} 
Proposition \ref{proposition:level12EvenSpinAppellForm} and  Lemma \ref{lemma:alternateAppellForms} (\ref{equation:2nd-mu(q)}) give
\begin{align*}
 -q^{2+3s}& f_{5,5,1}(q^{4s+7},q^{4};q) +q^{1+s} f_{5,5,1}(q^{4s+7},q^{2};q)\\
&=-q^{2+5s}j(q^{4s+7};q^5)\left (\sum_{k=0}^{2s-1}(-1)^{k}q^{2sk-\binom{k+1}{2}}
 +q^{s+4\binom{s}{2}}\left ( \frac{1}{2}\mu_{2}(q)+\frac{1}{2}\frac{J_{2,4}^2}{J_{1}^3}\right ) \right )\\
&\qquad -\frac{1}{\overline{J}_{0,4}\overline{J}_{0,20}}
\left ( -q^{2+3s} \theta_{5,5,1}(q^{4s+7},q^{4};q) +q^{1+s} \theta_{5,5,1}(q^{4s+7},q^{2};q)\right ).
\end{align*}
Using (\ref{equation:j-elliptic}) gives us
\begin{equation*}
-q^{2+5s}j(q^{4s+7};q^{5})=-q^{2+5s}j(q^{5(s+1)+2-s};q^{5})=-q^{2+5s}(-1)^{s+1}q^{-5\binom{s+1}{2}-(2-s)(s+1)}j(q^{2-s};q^{5}).
\end{equation*}
Simplifying and using (\ref{equation:j-flip}) gives
\begin{equation*}
-q^{2+5s}j(q^{4s+7};q^{5})=(-1)^{s}q^{-3\binom{s}{2}}j(q^{s+3};q^5).
\end{equation*}
Substituting back in gives
{\allowdisplaybreaks \begin{align*}
 -q^{2+3s}& f_{5,5,1}(q^{4s+7},q^{4};q) +q^{1+s} f_{5,5,1}(q^{4s+7},q^{2};q)\\
&=(-1)^{s}j(q^{s+3};q^5)\left (q^{-3\binom{s}{2}}\sum_{k=0}^{2s-1}(-1)^{k}q^{2sk-\binom{k+1}{2}}
 +q^{\binom{s+1}{2}}\left ( \frac{1}{2}\mu_{2}(q)+\frac{1}{2}\frac{J_{2,4}^4}{J_{1}^3}\right ) \right )\\
&\qquad -\frac{1}{\overline{J}_{0,4}\overline{J}_{0,20}}
\left ( -q^{2+3s} \theta_{5,5,1}(q^{4s+7},q^{4};q) +q^{1+s} \theta_{5,5,1}(q^{4s+7},q^{2};q)\right ).
\end{align*}}%
Proposition \ref{proposition:level12EvenSpinThetaId} then gives the result.
 \end{proof}
 
 \begin{proof}[Proof of Proposition \ref{proposition:level12EvenSpinAppellForm}]
From Corollary \ref{corollary:f551-HeckeExpansion} we have that
\begin{equation}
h_{5,5,1}(x,y;q)=j(x;q^5)m(-q^4x^{-1}y,-1;q^4)+j(y;q)m(-q^{10}xy^{-5},-1;q^{20}),
\end{equation}
Hence
\begin{align*}
 -q^{2+3s}& h_{5,5,1}(q^{4s+7},q^{4};q) +q^{1+s} h_{5,5,1}(q^{4s+7},q^{2};q)\\
 &=-q^{2+3s}\left ( j(q^{4s+7};q^5)m(-q^{1-4s},-1;q^4)+j(q^{4};q)m(-q^{-3+4s},-1;q^{20})\right ) \\
 &\qquad +q^{1+s}\left (j(q^{4s+7};q^5)m(-q^{-1-4s},-1;q^4)+j(q^2;q)m(-q^{7+4s},-1;q^{20}) \right ). 
 \end{align*}
 We note that $j(q^n;q)=0$ for $n\in\mathbb{Z}$ and that the respective Appell functions are defined.  This gives the more compact
\begin{align*}
 -q^{2+3s}& h_{5,5,1}(q^{4s+7},q^{4};q) +q^{1+s} h_{5,5,1}(q^{4s+7},q^{2};q)\\
 &= q^{1+s}j(q^{4s+7};q^5)\left ( m(-q^{-1-4s},-1;q^4)  -q^{1+2s} m(-q^{1-4s},-1;q^4) \right ). 
 \end{align*}
 We use the Appell function property (\ref{equation:mxqz-flip}) and factor out a $q$-term to get
 \begin{align*}
 -q^{2+3s}& h_{5,5,1}(q^{4s+7},q^{4};q) +q^{1+s} h_{5,5,1}(q^{4s+7},q^{2};q)\\
 &=q^{1+s}j(q^{4s+7};q^5)\left ( -q^{1+4s}m(-q^{1+4s},-1;q^4)+q^{6s} m(-q^{-1+4s},-1;q^4)\right )\\
  &= -q^{2+5s}j(q^{4s+7};q^5)\left (m(-q^{1+4s},-1;q^4)-q^{2s-1} m(-q^{-1+4s},-1;q^4)\right ).
 \end{align*}
 Using Lemma \ref{lemma:generalAppellSums} (\ref{equation:generalAppellSumOneModFour}) gives us the result.
 \end{proof}


\begin{proof}[Proof of Proposition \ref{proposition:level12evenSpinFuncEqn} for $F(r)$]
Let us specialize Proposition  \ref{proposition:f-functionaleqn} to $(a,b,c)=(5,5,1)$, with $\ell=-1$ and $k=5$.  Then
\begin{align*}
f_{5,5,1}(x,y;q)&=(-x)^{-1}(-y)^{5}q^{5\binom{-1}{2}-25+\binom{5}{2}}f_{5,5,1}(q^{20}x,y;q)\\
&\quad +\sum_{m=0}^{-1-1}(-x)^mq^{5\binom{m}{2}}j(q^{5m}y;q)+\sum_{m=0}^{4}(-y)^mq^{\binom{m}{2}}j(q^{5m}x;q^5).
\end{align*}
Rewriting with our summation convention (\ref{equation:sumconvention}) gives
\begin{align}
f_{5,5,1}(x,y;q)&=x^{-1}y^{5}q^{-10}f_{5,5,1}(q^{20s}x,y;q)\label{equation:f551-prefunc}\\
&\quad -\sum_{m=-1}^{-1}(-x)^mq^{5\binom{m}{2}}j(q^{5m}y;q)+\sum_{m=0}^{4}(-y)^mq^{\binom{m}{2}}j(q^{5m}x;q^5).\notag
\end{align}
Into (\ref{equation:f551-prefunc}), we substitute $(x,y)=(q^{4r+7},q^4)$ and $(x,y)=(q^{4r+7},q^2)$.  We recall that $j(q^n;q)=0$ for $n\in\mathbb{Z}$. This gives respectively
\begin{align*}
f_{5,5,1}(q^{4r+7},q^4;q)&=q^{-4r+3}f_{5,5,1}(q^{20+4r+7},q^4;q)\\
&\quad +\sum_{m=0}^{4}(-1)^m q^{4m}q^{\binom{m}{2}}j(q^{m5+4r+7};q^5),
\end{align*}
and
\begin{align*}
f_{5,5,1}(q^{4r+7},q^2;q)&=q^{-4r-7}f_{5,5,1}(q^{20+4r+7},q^2;q)\\
&\quad +\sum_{m=0}^{4}(-1)^mq^{2m}q^{\binom{m}{2}}j(q^{m5+4r+7};q^5).
\end{align*}
We isolate the double-sums on the right-hand sides of the two equations.  This yields
\begin{align*}
f_{5,5,1}(q^{27+4r},q^4;q)
&=q^{4r-3}f_{5,5,1}(q^{4r+7},q^4;q) 
 - q^{4r-3}j(q^{4r+7};q^5)\sum_{m=0}^{4} q^{-2m^2-m-4mr},
\end{align*}
and
\begin{align*}
f_{5,5,1}(q^{27+4r},q^2;q)
&=q^{4r+7}f_{5,5,1}(q^{4r+7},q^2;q)
 -q^{4r+7}j(q^{4r+7};q^5)\sum_{m=0}^{4}q^{-2m^2-3m-4mr}.
\end{align*}
We want to combine these two equations such that the left-hand side is in terms of $F(r+5)$ and the right-hand side is in terms of $F(r)$.  To achieve this, we need to multiple both equations by an appropriate power of $q$.  This brings us to
\begin{align*}
-q^{17+3r}f_{5,5,1}(q^{27+4r},q^4;q)
&=-q^{2+3r}q^{12+4r}f_{5,5,1}(q^{4r+7},q^4;q) \\
&\quad +q^{2+3r}q^{12+4r}j(q^{4r+7};q^5)\sum_{m=0}^{4} q^{-2m^2-m-4mr},
\end{align*}
and
\begin{align*}
q^{6+r}f_{5,5,1}(q^{27+4r},q^2;q)
&=q^{1+r}q^{12+4r}f_{5,5,1}(q^{4r+7},q^2;q)\\
&\quad -q^{1+r}q^{12+4r}j(q^{4r+7};q^5)\sum_{m=0}^{4}q^{-2m^2-3m-4mr}.
\end{align*}
Adding the two equations gives
{\allowdisplaybreaks \begin{align*}
&-q^{17+3r}f_{5,5,1}(q^{27+4r},q^4;q)+q^{6+r}f_{5,5,1}(q^{27+4r},q^2;q)\\
&\quad =q^{12+4r}\left ( -q^{2+3r}f_{5,5,1}(q^{4r+7},q^4;q)+q^{1+r}f_{5,5,1}(q^{4r+7},q^2;q)\right ) \\
&\qquad +q^{12+4r}j(q^{4r+7};q^5)\left ( q^{2+3r}\sum_{m=0}^{4} q^{-2m^2-m-4mr}
 -q^{1+r}\sum_{m=0}^{4}q^{-2m^2-3m-4mr}\right ). 
\end{align*}}%
We need to rewrite several aspects of this new equation.  We first use (\ref{equation:j-elliptic}) to obtain
\begin{align*}
j(q^{4r+7};q^5)&=j(q^{5(r+1)+2-r};q^{5})=(-1)^{r+1}q^{-5\binom{r+1}{2}-(r+1)(2-r)}j(q^{r+3};q^{5})\\
&=(-1)^{r+1}q^{-3\binom{r+1}{2}-2r-2}j(q^{r+3};q^{5}).
\end{align*}
Hence
\begin{align*}
&-q^{17+3r}f_{5,5,1}(q^{27+4r},q^4;q)+q^{6+r}f_{5,5,1}(q^{27+4r},q^2;q)\\
&\quad =q^{12+4r}\left ( -q^{2+3r}f_{5,5,1}(q^{4r+7},q^4;q)+q^{1+r}f_{5,5,1}(q^{4r+7},q^2;q)\right ) \\
&\qquad +(-1)^{r+1}q^{10+2r-3\binom{r+1}{2}}j(q^{r+3};q^{5})\left ( q^{2+3r}\sum_{m=0}^{4} q^{-2m^2-m-4mr}
 -q^{1+r}\sum_{m=0}^{4}q^{-2m^2-3m-4mr}\right ). 
\end{align*}
Rewriting the pairs of double-sums in terms of $F(r)$ completes the proof.
\end{proof}


\begin{proof}[Proof of Proposition \ref{proposition:level12evenSpinFuncEqn} for $G(r)$]
To set up the functional equation for $G(r)$, we first see that
{\allowdisplaybreaks \begin{align*}
G(r+5)&=(-1)^{5+r}j(q^{5+r+3};q^5)
\left (q^{\binom{5+r+1}{2}}\frac{1}{2}\mu_2(q)
+ q^{-3\binom{5+r}{2}}\sum_{k=0}^{2(5+r)-1}( -1)^{k}q^{2(5+r)k-\binom{k+1}{2}}\right )\\
&\qquad -(-1)^{5+r}(-q)^{\binom{5+r+1}{2}}\frac{1}{2}\frac{J_{1}^3}{J_{2}J_{4}}j((-q)^{5+r+3};-q^5).
\end{align*}}%
Let us rewrite the simple quotient of theta functions.  Using (\ref{equation:j-elliptic}), we find
{\allowdisplaybreaks \begin{align*}
(-1)^{5+r}&(-q)^{\binom{5+r+1}{2}}\frac{1}{2}\frac{J_{1}^3}{J_{2}J_{4}}j((-q)^{5+r+3};-q^5)\\
&=-(-q)^{10+5(r+1)}(-1)^{r}(-1)(-q)^{-(r+3)}(-q)^{\binom{r+1}{2}}\frac{1}{2}\frac{J_{1}^3}{J_{2}J_{4}}j((-q)^{r+3};-q^5)\\
&=q^{12+4r}(-1)^{r}(-q)^{\binom{r+1}{2}}\frac{1}{2}\frac{J_{1}^3}{J_{2}J_{4}}j((-q)^{r+3};-q^5).
\end{align*}}%
Hence
{\allowdisplaybreaks \begin{align*}
G(5+r)&=(-1)^{r}q^{12+4r}j(q^{r+3};q^5)q^{\binom{r+1}{2}}\frac{1}{2}\mu_2(q)\\
&\qquad + (-1)^{r}q^{-33-16r}j(q^{r+3};q^5)q^{-3\binom{r}{2}}\sum_{k=0}^{2(5+r)-1}( -1)^{k}q^{2(5+r)k-\binom{k+1}{2}}\\
&\qquad -q^{12+4r}(-1)^{r}(-q)^{\binom{r+1}{2}}\frac{1}{2}\frac{J_{1}^3}{J_{2}J_{4}}j((-q)^{r+3};-q^5).
\end{align*}}%
We then rewrite the right-hand side in terms of $G(r)$.  This gives
{\allowdisplaybreaks \begin{align*}
G(r+5)&=q^{12+4r}G(r)-(-1)^{r}q^{12+4r}j(q^{r+3};q^5)q^{-3\binom{r}{2}}\sum_{k=0}^{2r-1}( -1)^{k}q^{2rk-\binom{k+1}{2}}\\
&\qquad + (-1)^{r}q^{-33-16r}j(q^{r+3};q^5)q^{-3\binom{r}{2}}\sum_{k=0}^{2(5+r)-1}( -1)^{k}q^{2(5+r)k-\binom{k+1}{2}}.
\end{align*}}%
Comparing with the functional equation in Proposition \ref{proposition:level12evenSpinFuncEqn}, we see that we want to show
{\allowdisplaybreaks \begin{align*}
(-1)^{r+1}&q^{10+2r-3\binom{r+1}{2}}j(q^{r+3};q^{5})\left ( q^{2+3r}\sum_{m=0}^{4} q^{-2m^2-m-4mr}
 -q^{1+r}\sum_{m=0}^{4}q^{-2m^2-3m-4mr}\right ) \\
 &=-(-1)^{r}q^{12+4r}j(q^{r+3};q^5)q^{-3\binom{r}{2}}\sum_{k=0}^{2r-1}( -1)^{k}q^{2rk-\binom{k+1}{2}}\\
&\qquad + (-1)^{r}q^{-33-16r}j(q^{r+3};q^5)q^{-3\binom{r}{2}}\sum_{k=0}^{2(5+r)-1}( -1)^{k}q^{2(5+r)k-\binom{k+1}{2}}.
\end{align*}}%
We can simplify this.  We focus on the coefficient of $j(q^{r+3};q^5)$, and we can also factor out a power of $q$.  Hence we now only need to show
{\allowdisplaybreaks \begin{align*}
 q^{1+2r}&\sum_{m=0}^{4} q^{-2m^2-m-4mr}
 -\sum_{m=0}^{4}q^{-2m^2-3m-4mr} \\
 &=q^{1+4r}\sum_{k=0}^{2r-1}( -1)^{k}q^{2rk-\binom{k+1}{2}}
  - q^{-44-16r}\sum_{k=0}^{10+2r-1}( -1)^{k}q^{2(5+r)k-\binom{k+1}{2}}.
\end{align*}}%
To achieve this, we prove two things.  For the first, we need to show
\begin{equation}
q^{1+4r}\sum_{k=0}^{2r-1}( -1)^{k}q^{2rk-\binom{k+1}{2}}
  - q^{-44-16r}\sum_{k=10}^{10+2r-1}( -1)^{k}q^{2(5+r)k-\binom{k+1}{2}}=0,\label{equation:level12evenSpinRHSFuncEqnId1}
\end{equation}
and for the second, we need to show
\begin{equation}
 q^{1+2r}\sum_{m=0}^{4} q^{-2m^2-m-4mr}
 -\sum_{m=0}^{4}q^{-2m^2-3m-4mr} 
 =  - q^{-44-14r}\sum_{k=0}^{9}( -1)^{k}q^{2(5+r)k-\binom{k+1}{2}}.\label{equation:level12evenSpinRHSFuncEqnId2}
\end{equation}
To show (\ref{equation:level12evenSpinRHSFuncEqnId1}), a simple shift in indices gives
{\allowdisplaybreaks \begin{align*}
\sum_{k=10}^{10+2r-1}( -1)^{k}q^{2(5+r)k-\binom{k+1}{2}}&=\sum_{k=0}^{2r-1}( -1)^{k+10}q^{2(5+r)(k+10)-\binom{k+10+1}{2}}\\
&=q^{45+20r}\sum_{k=0}^{2r-1}( -1)^{k}q^{2rk-\binom{k+1}{2}}.
\end{align*}}%
To show (\ref{equation:level12evenSpinRHSFuncEqnId2}), let us replace $k$ with $(2m,2m+1)$.  This gives
\begin{align*}
\sum_{k=0}^{9}( -1)^{k}q^{2(5+r)k-\binom{k+1}{2}}
&=\sum_{m=0}^{4}q^{2(5+r)2m-\binom{2m+1}{2}}
-\sum_{m=0}^{4}q^{2(5+r)(2m+1)-\binom{2m+2}{2}}\\
&=\sum_{m=0}^{4}q^{-2m^2+4rm+19m}
-\sum_{m=0}^{4}q^{-2m^2+4rm+17m+2r+9}.
\end{align*}
We reverse the summations with $m\to 4-m$ to get
\begin{equation*}
\sum_{k=0}^{9}( -1)^{k}q^{2(5+r)k-\binom{k+1}{2}}
=q^{16r+44}\sum_{m=0}^{4}q^{-2m^2-4rm-3m}
-q^{18r+45}\sum_{m=0}^{4}q^{-2m^2-4rm-m},
\end{equation*}
which is what we want.
\end{proof}


\subsection{The $1/3$-level string functions: Theorem \ref{theorem:genEulerOneThirdFirstPairEvenSpin}}  We state the necessary propositions to prove Theorem \ref{theorem:genEulerOneThirdFirstPairEvenSpin}, then we give the proof, and then we give the proofs of the propositions.

\smallskip
Let us define the left-hand side of identity in Theorem \ref{theorem:genEulerOneThirdFirstPairEvenSpin} to be
\begin{equation*}
F(r):=-q^{3+4r} f_{7,7,1}(q^{6r+10},q^{5};q) +q^{2+2r} f_{7,7,1}(q^{6r+10},q^{3};q),
\end{equation*}
and the right-hand side to 
\begin{align*}
G(r)&:=(-1)^rj(q^{r+4};q^{7})
 \left ( q^{\binom{r+1}{2}}\cdot q\omega_3(-q)
 +q^{-5\binom{r}{2}}\sum_{k=0}^{r-1}q^{6kr-6\binom{k+1}{2}}\left( q^{k}-q^{-k+2r-1}\right ) \right ) \\
 &\qquad -q^{5r^2-3r+1}\frac{J_{1}^3J_{4}}{J_{2}^3}j(q^{8+16r};q^{28}).
\end{align*}
\begin{proposition}\label{proposition:level13evenSpinFirstPairFuncEqn}  We have that both $F(r)$ and $G(r)$ satisfy the same functional equation:
{\allowdisplaybreaks \begin{align*}
H&(r+7)-q^{24+6r}H(r) \\
& =-(-1)^{r}q^{-5\binom{r}{2}+23}j(q^{r+4};q^{7})\left ( q^{1+2r}
\sum_{m=0}^{6}q^{-3m^2-2m-6mr} -\sum_{m=0}^{6}q^{-3m^2-4m-6mr}\right ) .
\end{align*}}%
\end{proposition}

We evaluate the Appell function expression in Corollary \ref{corollary:f771-HeckeExpansion} for the respective double-sums in Theorem \ref{theorem:genEulerOneThirdFirstPairEvenSpin}.  We have

\begin{proposition}\label{proposition:level13EvenSpinFirstPairAppellForm}  We have that
\begin{align*}
 -q^{3+4s} &h_{7,7,1}(q^{6s+10},q^{5};q) +q^{2+2s} h_{7,7,1}(q^{6s+10},q^{3};q)\\
 &= (-1)^sj(q^{s+4};q^7)\left ( q^{-5\binom{s}{2}}\sum_{k=0}^{s-1}q^{6sk-6\binom{k+1}{2}}\left (q^{k} -q^{-k+2s-1}\right ) 
 +q^{\binom{s+1}{2}}\cdot q\omega_{3}(-q)\right )\\
 &\qquad +(-1)^s2j(q^{s+4};q^7)q^{\binom{s+1}{2}}\left ( -\frac{1}{2}q\frac{J_{6}^3}{J_{2}\overline{J}_{3,6}}
+\frac{J_{6}^3\overline{J}_{2,6}J_{3,6}}{\overline{J}_{0,6}J_{1,6}J_{2,6}\overline{J}_{3,6}}\right ). 
\end{align*}
\end{proposition}

\begin{proof}[Proof of Theorem  \ref{theorem:genEulerOneThirdFirstPairEvenSpin}] Proposition \ref{proposition:level13evenSpinFirstPairFuncEqn} tells us that
\begin{equation*}
F(s+7)-G(s+7)=q^{24+6s}\left (F(s)-G(s) \right ), 
\end{equation*}
so we only need to prove Theorem  \ref{theorem:genEulerOneThirdFirstPairEvenSpin} for $0\le s\le 6$.  The case $s=3$ follows from Lemma \ref{lemma:degenerateDoubleSumCoeffs}.   Corollary \ref{corollary:f771-HeckeExpansion} and Proposition \ref{proposition:level13EvenSpinFirstPairAppellForm} tell us that this is equivalent to showing
\begin{align*}
(-1)^s&2j(q^{s+4};q^7)q^{\binom{s+1}{2}}\left ( -\frac{1}{2}q\frac{J_{6}^3}{J_{2}\overline{J}_{3,6}}
+\frac{J_{6}^3\overline{J}_{2,6}J_{3,6}}{\overline{J}_{0,6}J_{1,6}J_{2,6}\overline{J}_{3,6}}\right )\\
&\qquad -\frac{1}{\overline{J}_{0,6}\overline{J}_{0,42}}
\left ( -q^{3+4s}\theta_{7,7,1}(q^{6s+10},q^5;q) +q^{2+2s}\theta_{7,7,1}(q^{6s+10},q^3;q)\right ) \\
&=-q^{5s^2-3s+1}\frac{J_{1}^3J_{4}}{J_{2}^3}j(q^{8+16s};q^{28}),
\end{align*}
but this is just Proposition \ref{proposition:level13EvenSpinFirstPairThetaId}.
\end{proof}


\begin{proof}[Proof of Proposition \ref{proposition:level13EvenSpinFirstPairAppellForm}]
From Corollary \ref{corollary:f771-HeckeExpansion}, we have that
\begin{align*}
h_{7,7,1}(x,y;q)&=j(x;q^7)m\big({ -}q^{6}yx^{-1},-1;q^{6} \big )
 +j(y;q)m\big({ -}q^{21}xy^{-7},-1;q^{42} \big ).
\end{align*}
Hence
\begin{align*}
 A(q):&=-q^{3+4s} h_{7,7,1}(q^{6s+10},q^{5};q) +q^{2+2s} h_{7,7,1}(q^{6s+10},q^{3};q)\\
 &=-q^{3+4s}\left ( j(q^{6s+10};q^7)m\big( {-}q^{1-6s},-1;q^{6} \big )
  +j(q^{5};q)m\big({ -}q^{-4+6s},-1;q^{42} \big )\right ) \\
&\qquad + q^{2+2s}\left (j(q^{6s+10};q^7)m\big({ -}q^{-1-6s},-1;q^{6} \big )
 +j(q^3;q)m\big({ -}q^{10+6s},-1;q^{42} \big ) \right ). 
\end{align*}
We note that $j(q^n;q)=0$ for $n\in\mathbb{Z}$ and that the respective Appell functions are defined.  This gives the more compact
\begin{equation*}
A(q) = q^{2+2s}j(q^{6s+10};q^7)\left ( m\big( {-}q^{-1-6s},-1;q^{6} \big ) -q^{1+2s}m\big({ -}q^{1-6s},-1;q^{6} \big )\right ). 
\end{equation*}
We use the Appell function property (\ref{equation:mxqz-flip}) and factor out a $q$-term to get
\begin{align*}
 A(q)&= q^{2+2s}j(q^{6s+10};q^7)\left (-q^{1+6s} m\big({ -}q^{1+6s},-1;q^{6} \big ) +q^{8s}m\big( {-}q^{-1+6s},-1;q^{6} \big )\right )\\
  &= -q^{3+8s}j(q^{6s+10};q^7)\left (m\big( {-}q^{1+6s},-1;q^{6} \big ) -q^{2s-1}m\big({ -}q^{-1+6s},-1;q^{6} \big )\right ).
\end{align*}
Using Lemma \ref{lemma:generalAppellSums} (\ref{equation:generalAppellSumOneModSix}) gives us
\begin{align*}
 A(q)&= -q^{3+8s}j(q^{6s+10};q^7)\left ( \sum_{k=0}^{s-1}q^{6sk-6\binom{k+1}{2}}\left (q^{k} -q^{-k+2s-1}\right ) 
 +2q^{s+6\binom{s}{2}}m ( -q,-1;q^{6}  )\right ).
\end{align*}
Rewriting the last Appell function with Lemma \ref{lemma:alternateAppellForms} (\ref{equation:alternateAppellForm3rd-w}) gives
\begin{align*}
 A(q)&= -q^{3+8s}j(q^{6s+10};q^7)\left ( \sum_{k=0}^{s-1}q^{6sk-6\binom{k+1}{2}}\left (q^{k} -q^{-k+2s-1}\right ) 
 +q^{s+6\binom{s}{2}}q\omega_{3}(-q)\right )\\
 &\qquad -2q^{3+8s}j(q^{6s+10};q^7)q^{s+6\binom{s}{2}}\left ( -\frac{1}{2}q\frac{J_{6}^3}{J_{2}\overline{J}_{3,6}}
+\frac{J_{6}^3\overline{J}_{2,6}J_{3,6}}{\overline{J}_{0,6}J_{1,6}J_{2,6}\overline{J}_{3,6}}\right ).
\end{align*}
We use (\ref{equation:j-elliptic}) and (\ref{equation:j-flip}) to see that
\begin{align*}
j(q^{6s+10};q^7)&=j(q^{7(s+1)+4-s};q^7)=(-1)^{s+1}q^{-7\binom{s+1}{2}-(s+1)(4-s)}j(q^{4-s};q^{7})\\
&=(-1)^{s+1}q^{-7\binom{s+1}{2}-(s+1)(4-s)}j(q^{s+3};q^{7}),
\end{align*}
and the result follows.
\end{proof}


\begin{proof}[Proof of Proposition \ref{proposition:level13evenSpinFirstPairFuncEqn} for $F(r)$.]  
Let us specialize Proposition  \ref{proposition:f-functionaleqn} to $(a,b,c)=(7,7,1)$ with $\ell=-1$, and $k=7$, and then use our summation convention (\ref{equation:sumconvention}) on the first sum.  This gives
\begin{align}
f_{7,7,1}(x,y;q)&=x^{-1}y^{7}q^{-21}f_{7,7,1}(q^{42}x,y;q)\label{equation:f771-FirstPairPreFunc}\\
&\quad -\sum_{m=-1}^{-1}(-x)^mq^{7\binom{m}{2}}j(q^{m7}y;q)+\sum_{m=0}^{6}(-y)^mq^{\binom{m}{2}}j(q^{m7}x;q^7).\notag
\end{align}
Into (\ref{equation:f771-FirstPairPreFunc}), we substitute  $(x,y)=(q^{6r+10},q^5)$ and $(x,y)=(q^{6r+10},q^3)$.   We recall that $j(q^n;q)=0$ for $n\in\mathbb{Z}$. This gives respectively
\begin{align*}
f_{7,7,1}(q^{6r+10},q^5;q)&=q^{-6r+4}f_{7,7,1}(q^{42+6r+10},q^{5};q)\\
&\quad +\sum_{m=0}^{6}(-1)^{m}q^{5m}q^{\binom{m}{2}}j(q^{7m+6r+10};q^7),
\end{align*}
and
\begin{align*}
f_{7,7,1}(q^{6r+10},q^3;q)&=q^{-6r-10}f_{7,7,1}(q^{42+6r+10},q^3;q)\\
&\quad +\sum_{m=0}^{6}(-1)^{m}q^{3m}q^{\binom{m}{2}}j(q^{7m+6r+10};q^7).
\end{align*}
We rewrite the theta function in the two sums.  Using (\ref{equation:j-elliptic}) and simplifying produces
{\allowdisplaybreaks \begin{align*}
j(q^{7m+6r+10};q^7)&=j(q^{7(m+r+1)-r+3};q^7)\\
&=(-1)^{m+r+1}q^{-7\binom{m+r+1}{2}}q^{-(m+r+1)(3-r)}j(q^{3-r};q^{7})\\
&=(-1)^{m+r+1}q^{-7\binom{m}{2}-10m-6mr}q^{-5\binom{r}{2}-8r-3}j(q^{r+4};q^{7}).
\end{align*}}%
Our two equations can then be rewritten
\begin{align*}
f_{7,7,1}(q^{6r+10},q^5;q)&=q^{-6r+4}f_{7,7,1}(q^{42+6r+10},q^{5};q)\\
&\quad +(-1)^{r+1}q^{-5\binom{r}{2}-8r-3}j(q^{r+4};q^{7})
\sum_{m=0}^{6}q^{-3m^2-2m-6mr},
\end{align*}
and
\begin{align*}
f_{7,7,1}(q^{6r+10},q^3;q)&=q^{-6r-10}f_{7,7,1}(q^{42+6r+10},q^3;q)\\
&\quad +(-1)^{r+1}q^{-5\binom{r}{2}-8r-3}j(q^{r+4};q^{7})
\sum_{m=0}^{6}q^{-3m^2-4m-6mr}.
\end{align*}
We isolate the double-sums on the right-hand sides of the two equations to get
\begin{align*}
f_{7,7,1}(q^{42+6r+10},q^{5};q)&=q^{6r-4}f_{7,7,1}(q^{6r+10},q^5;q)\\
&\quad +(-1)^{r}q^{-5\binom{r}{2}-2r-7}j(q^{r+4};q^{7})
\sum_{m=0}^{6}q^{-3m^2-2m-6mr},
\end{align*}
and
\begin{align*}
f_{7,7,1}(q^{42+6r+10},q^3;q)&=q^{6r+10}f_{7,7,1}(q^{6r+10},q^3;q)\\
&\quad +(-1)^{r}q^{-5\binom{r}{2}-2r+7}j(q^{r+4};q^{7})
\sum_{m=0}^{6}q^{-3m^2-4m-6mr}.
\end{align*}
We want the sum of the two left-hand sides to look like $F(r+7)$, so we multiply through by an appropriate power of $q$ to get
{\allowdisplaybreaks \begin{align*}
-q^{31+4r}f_{7,7,1}(q^{42+6r+10},q^{5};q)&=-q^{27+10r}f_{7,7,1}(q^{6r+10},q^5;q)\\
&\quad -(-1)^{r}q^{-5\binom{r}{2}+2r+24}j(q^{r+4};q^{7})
\sum_{m=0}^{6}q^{-3m^2-2m-6mr},
\end{align*}}
and
{\allowdisplaybreaks \begin{align*}
q^{16+2r}f_{7,7,1}(q^{42+6r+10},q^3;q)&=q^{26+8r}f_{7,7,1}(q^{6r+10},q^3;q)\\
&\quad +(-1)^{r}q^{-5\binom{r}{2}+23}j(q^{r+4};q^{7})
\sum_{m=0}^{6}q^{-3m^2-4m-6mr}.
\end{align*}}%
Adding the two equations and combining like terms brings us to
\begin{align*}
-&q^{31+4r}f_{7,7,1}(q^{42+6r+10},q^{5};q)+q^{16+2r}f_{7,7,1}(q^{42+6r+10},q^3;q)\\
&=q^{24+6r}\left ( -q^{3+4r}f_{7,7,1}(q^{6r+10},q^5;q)+q^{2+2r}f_{7,7,1}(q^{6r+10},q^3;q)\right ) \\
&\quad -(-1)^{r}q^{-5\binom{r}{2}+23}j(q^{r+4};q^{7})\left ( q^{1+2r}
\sum_{m=0}^{6}q^{-3m^2-2m-6mr} -\sum_{m=0}^{6}q^{-3m^2-4m-6mr}\right ) .
\end{align*}
Rewriting the above in terms of $F(r)$ brings us to the desired result.
\end{proof}


\begin{proof}[Proof of Proposition \ref{proposition:level13evenSpinFirstPairFuncEqn} for $G(r)$]  We first note that
\begin{align*}
G&(r+7)=(-1)^{r+7}j(q^{r+4+7};q^{7})\\
&\qquad \times \left ( q^{\binom{7+r+1}{2}+1}\omega_3(-q)
 +q^{-5\binom{r+7}{2}}\sum_{k=0}^{r+7-1}q^{6k(r+7)-6\binom{k+1}{2}}\left( q^{k}-q^{-k+2(r+7)-1}\right ) \right ) \\
 &\qquad \qquad -q^{5(r+7)^2-3(r+7)+1}\frac{J_{1}^3J_{4}}{J_{2}^3}j(q^{8+16(r+7)};q^{28}).
\end{align*}
Let us rewrite the simple quotient of theta functions.  Using (\ref{equation:j-elliptic}), we find
{\allowdisplaybreaks \begin{align*}
q^{5(r+7)^2-3(r+7)+1}\frac{J_{1}^3J_{4}}{J_{2}^3}j(q^{8+16(r+7)};q^{28})
&=q^{5(r+7)^2-3(r+7)+1}\frac{J_{1}^3J_{4}}{J_{2}^3}j(q^{8+16r+28\cdot 4};q^{28})\\
&=q^{5(r+7)^2-3(r+7)+1}\frac{J_{1}^3J_{4}}{J_{2}^3}q^{-28\binom{4}{2}-4(8+16r)}j(q^{8+16r};q^{28})\\
&=q^{5r^2+3r+25}\frac{J_{1}^3J_{4}}{J_{2}^3}j(q^{8+16r};q^{28})\\
&=q^{24+6r}q^{5r^2-3r+1}\frac{J_{1}^3J_{4}}{J_{2}^3}j(q^{8+16r};q^{28}).
\end{align*}}%
Again employing (\ref{equation:j-elliptic}) allows us to write
{\allowdisplaybreaks \begin{align*}
G(r+7)&=(-1)^{r}q^{24+6r}j(q^{r+4};q^{7}) q^{\binom{r+1}{2}}\cdot q\omega_3(-q)\\
&\qquad  +(-1)^{r}q^{-109-36r}j(q^{r+4};q^{7}) q^{-5\binom{r}{2}}
\sum_{k=0}^{r+7-1}q^{6k(r+7)-6\binom{k+1}{2}}\left( q^{k}-q^{-k+2(r+7)-1}\right )\\
 &\qquad -q^{24+6r}q^{5r^2-3r+1}\frac{J_{1}^3J_{4}}{J_{2}^3}j(q^{8+16r};q^{28}).
\end{align*}}%
Rewriting the right-hand side of the above in terms of $G(r)$ brings us to
{\allowdisplaybreaks \begin{align*}
G(r+7)&=q^{24+6r}G(r)
-(-1)^rq^{24+6r}j(q^{r+4};q^{7})q^{-5\binom{r}{2}}\sum_{k=0}^{r-1}q^{6kr-6\binom{k+1}{2}}\left( q^{k}-q^{-k+2r-1}\right)\\
&\quad  +(-1)^{r}q^{-109-36r}j(q^{r+4};q^{7}) q^{-5\binom{r}{2}}
\sum_{k=0}^{r+7-1}q^{6k(r+7)-6\binom{k+1}{2}}\left( q^{k}-q^{-k+2(r+7)-1}\right ).
\end{align*}}%
Comparing with the functional equation in Proposition \ref{proposition:level13evenSpinFirstPairFuncEqn}, we see that we need to show
{\allowdisplaybreaks \begin{align*}
-(-1)^{r}&q^{-5\binom{r}{2}+23}j(q^{r+4};q^{7})\left ( q^{1+2r}
\sum_{m=0}^{6}q^{-3m^2-2m-6mr} -\sum_{m=0}^{6}q^{-3m^2-4m-6mr}\right ) \\
&=-(-1)^rq^{24+6r}j(q^{r+4};q^{7})q^{-5\binom{r}{2}}\sum_{k=0}^{r-1}q^{6kr-6\binom{k+1}{2}}\left( q^{k}-q^{-k+2r-1}\right)\\
&\quad  +(-1)^{r}q^{-109-36r}j(q^{r+4};q^{7}) q^{-5\binom{r}{2}}
\sum_{k=0}^{r+7-1}q^{6k(r+7)-6\binom{k+1}{2}}\left( q^{k}-q^{-k+2(r+7)-1}\right ).
\end{align*}}%
We break this into two parts, but first we simplify what we need to do.  We isolate the coefficient of $j(q^{r+4};q^7)$, and we also factor out a power of $q$.  Thus we only need to show
{\allowdisplaybreaks \begin{align*}
- q^{1+2r}&\sum_{m=0}^{6}q^{-3m^2-2m-6mr} +\sum_{m=0}^{6}q^{-3m^2-4m-6mr} \\
&=-q^{1+6r}\sum_{k=0}^{r-1}q^{6kr-6\binom{k+1}{2}}\left( q^{k}-q^{-k+2r-1}\right)\\
&\qquad  +q^{-132-36r}
\sum_{k=0}^{r+7-1}q^{6k(r+7)-6\binom{k+1}{2}}\left( q^{k}-q^{-k+2(r+7)-1}\right ).
\end{align*}}%
To show that the above is true, it suffices to show
{\allowdisplaybreaks \begin{align}
-q^{1+6r}&\sum_{k=0}^{r-1}q^{6kr-6\binom{k+1}{2}}\left( q^{k}-q^{-k+2r-1}\right)
\label{equation:level13evenSpinFirstPairRHSFuncEqnId1}\\
&  +q^{-132-36r}
\sum_{k=7}^{r+7-1}q^{6k(r+7)-6\binom{k+1}{2}}\left( q^{k}-q^{-k+2(r+7)-1}\right )=0,\notag
\end{align}}%
and
{\allowdisplaybreaks \begin{align}
- q^{1+2r}&\sum_{m=0}^{6}q^{-3m^2-2m-6mr} +\sum_{m=0}^{6}q^{-3m^2-4m-6mr} 
 \label{equation:level13evenSpinFirstPairRHSFuncEqnId2}\\
&=  q^{-132-36r}
\sum_{k=0}^{6}q^{6k(r+7)-6\binom{k+1}{2}}\left( q^{k}-q^{-k+2(r+7)-1}\right ).\notag
\end{align}}%
To prove (\ref{equation:level13evenSpinFirstPairRHSFuncEqnId1}), a simply shift of indices is all we need.  We have
{\allowdisplaybreaks \begin{align*}
\sum_{k=7}^{r+7-1}&q^{6k(r+7)-6\binom{k+1}{2}}\left( q^{k}-q^{-k+2(r+7)-1}\right )\\
&=\sum_{k=7}^{r+7-1}q^{6k(r+7)-6\binom{k+1}{2}+k}-\sum_{k=7}^{r+7-1}q^{6k(r+7)-6\binom{k+1}{2}-k+2(r+7)-1}\\
&=q^{42r+133}\sum_{k=0}^{r-1}q^{6kr-6\binom{k+1}{2}+k}
-q^{42r+133}\sum_{k=0}^{r-1}q^{6kr-6\binom{k+1}{2}-k+2r-1}\\
&=q^{42r+133} \sum_{k=0}^{r-1}q^{6kr-6\binom{k+1}{2}}\left (q^{k}-q^{-k+2r-1}\right).
\end{align*}}%
To prove (\ref{equation:level13evenSpinFirstPairRHSFuncEqnId2}), we reverse the summation with $k\to 6-m$.  This gives
{\allowdisplaybreaks \begin{align*}
\sum_{k=0}^{6}&q^{6k(r+7)-6\binom{k+1}{2}}\left( q^{k}-q^{-k+2(r+7)-1}\right )\\
&=\sum_{k=0}^{6}q^{6k(r+7)-6\binom{k+1}{2}+k}
-\sum_{k=0}^{6}q^{6k(r+7)-6\binom{k+1}{2}-k+2(r+7)-1}\\
&=q^{36r+132}\sum_{m=0}^{6}q^{-3m^2-6mr-4m}-q^{38r+133}\sum_{m=0}^{6}q^{-3m^2-6mr-2m},
\end{align*}}%
which is what we want.
\end{proof}


\subsection{The $1/3$-level string functions: Theorem \ref{theorem:genEulerOneThirdSecondPairEvenSpin}}  We state the necessary propositions to prove Theorem \ref{theorem:genEulerOneThirdSecondPairEvenSpin}, then we give the proof, and then we give the proofs of the propositions.

\smallskip
Let us define the left-hand side of identity in Theorem \ref{theorem:genEulerOneThirdSecondPairEvenSpin} to be
\begin{equation*}
F(r):= -q^{3+5r} f_{7,7,1}(q^{6r+10},q^{6};q) +q^{1+r} f_{7,7,1}(q^{6r+10},q^{2};q).
\end{equation*}
and the right-hand side to be
{\allowdisplaybreaks \begin{align*}
G(r)&:=(-1)^{r}j(q^{r+4};q^{7})
 \left ( q^{\binom{r+1}{2}}\frac{1}{2}f_{3}(q^2)
 +q^{-5\binom{r}{2}-r}\sum_{k=0}^{r-1}q^{6kr-6\binom{k+1}{2}}\left( q^{2k}-q^{-2k+4r-2}\right ) \right ) \\
 &\qquad \qquad -(-1)^{r}\frac{(-q)^{\binom{r+1}{2}}}{2}\frac{J_{1}^2}{J_{4}}j((-q)^{r+4};-q^7).
\end{align*}}%

\begin{proposition}\label{proposition:level13EvenSpinSecondPairFuncEqn}  We have that both $F(r)$ and $G(r)$ satisfy the same functional equation:
{\allowdisplaybreaks \begin{align*}
H&(r+7)-q^{24+6r}H(r)\\
&=-(-1)^{r}q^{22-r}q^{-5\binom{r}{2}}j(q^{r+4};q^{7})
\left( q^{2+4r}\sum_{m=0}^{6}q^{-3m^2-6mr-m}
-\sum_{m=0}^{6}q^{-3m^2-6mr-5m}\right ).
\end{align*}}%
\end{proposition}

\begin{proposition}\label{proposition:level13EvenSpinSecondPairAppellForm}  We have that
\begin{align*}
 -q^{3+5s} & h_{7,7,1}(q^{6s+10},q^{6};q) +q^{1+s} h_{7,7,1}(q^{6s+10},q^{2};q)\\
&= (-1)^{s}j(q^{s+4};q^7)\left (q^{-5\binom{s}{2}-s}\sum_{k=0}^{s-1}q^{6sk-6\binom{k+1}{2}}\left (q^{2k} -q^{-2k+4s-2}\right ) 
 +q^{\binom{s+1}{2}}\frac{1}{2}f_3(q^2)\right )\\
 &\qquad +(-1)^{s}2q^{\binom{s+1}{2}}j(q^{s+4};q^7)\left ( -\frac{1}{4}\frac{J_{6,12}^2}{J_{2}}
+\frac{J_{6}^3}{\overline{J}_{0,6}J_{2,6}}\right ). 
\end{align*}
\end{proposition}
\begin{proof}[Proof of Theorem \ref{theorem:genEulerOneThirdSecondPairEvenSpin}]  Proposition \ref{proposition:level13EvenSpinSecondPairAppellForm} tells us that
\begin{equation*}
F(s+7)-G(s+7)=q^{24+6s}\left (F(s)-G(s) \right ), 
\end{equation*}
so we only need to prove Theorem  \ref{theorem:genEulerOneThirdSecondPairEvenSpin} for $0\le s\le 6$.  The case $s=3$ follows from Lemma \ref{lemma:degenerateDoubleSumCoeffs}.   Corollary \ref{corollary:f771-HeckeExpansion} and Proposition \ref{proposition:level13EvenSpinSecondPairAppellForm} tell us that this is equivalent to showing
\begin{align*}
(-1)^{s}&2q^{\binom{s+1}{2}}j(q^{s+4};q^7)\left ( -\frac{1}{4}\frac{J_{6,12}^2}{J_{2}}
+\frac{J_{6}^3}{\overline{J}_{0,6}J_{2,6}}\right ) \\
&\qquad -\frac{1}{\overline{J}_{0,6}\overline{J}_{0,42}}
\left ( -q^{3+5s}\theta_{7,7,1}(q^{6s+10},q^6;q) +q^{1+s}\theta_{7,7,1}(q^{6s+10},q^2;q)\right ) \\
&=-(-1)^{s}\frac{(-q)^{\binom{s+1}{2}}}{2}\frac{J_{1}^2}{J_{4}}j((-q)^{s+4};-q^7),
\end{align*}
but this is just Proposition \ref{proposition:level13EvenSpinSecondPairThetaId}.
\end{proof}

\begin{proof}[Proof of Proposition \ref{proposition:level13EvenSpinSecondPairAppellForm}]
From Corollary \ref{corollary:f771-HeckeExpansion} we have that
\begin{align*}
h_{7,7,1}(x,y;q)&=j(x;q^7)m\big({ -}q^{6}yx^{-1},-1;q^{6} \big )
 +j(y;q)m\big( {-}q^{21}xy^{-7},-1;q^{42} \big ).
\end{align*}
Hence
\begin{align*}
 A(q):&=-q^{3+5s}  h_{7,7,1}(q^{6s+10},q^{6};q) +q^{1+s} h_{7,7,1}(q^{6s+10},q^{2};q)\\
 &=-q^{3+5s}\left ( j(q^{6s+10};q^7)m\big({ -}q^{2-6s},-1;q^{6} \big )
 +j(q^6;q)m\big({ -}q^{-11+6s},-1;q^{42} \big )\right ) \\
 &\qquad +q^{1+s}\left ( j(q^{6s+10};q^7)m\big( {-}q^{-2-6s},-1;q^{6} \big )
 +j(q^2;q)m\big( {-}q^{17+6s},-1;q^{42} \big )\right ).
\end{align*}
We note that $j(q^n;q)=0$ for $n\in\mathbb{Z}$ and that the respective Appell functions are defined.  This gives the more compact
\begin{equation*}
A(q) =q^{1+s}j(q^{6s+10};q^7)\left ( m\big( {-}q^{-2-6s},-1;q^{6} \big ) - q^{2+4s}m\big({ -}q^{2-6s},-1;q^{6} \big ) \right ) .
\end{equation*}
We use the Appell function property (\ref{equation:mxqz-flip}) and factor out a $q$-term to get
\begin{align*}
 A(q)&=q^{1+s}j(q^{6s+10};q^7)\left ( -q^{2+6s}m\big( {-}q^{2+6s},-1;q^{6} \big ) + q^{10s}m\big({ -}q^{-2+6s},-1;q^{6} \big ) \right )\\
  &= -q^{3+7s}j(q^{6s+10};q^7)\left (m\big({ -}q^{2+6s},-1;q^{6} \big ) - q^{4s-2}m\big( {-}q^{-2+6s},-1;q^{6} \big ) \right ).
\end{align*}
Using Lemma \ref{lemma:generalAppellSums} (\ref{equation:generalAppellSumTwoModSix}) gives us
\begin{equation*}
A(q)= -q^{3+7s}j(q^{6s+10};q^7)\left (\sum_{k=0}^{s-1}q^{6sk-6\binom{k+1}{2}}\left (q^{2k} -q^{-2k+4s-2}\right ) 
 +2q^{2s+6\binom{s}{2}}m ( -q^{2},-1;q^{6}  )\right ).
\end{equation*}
Rewriting the last Appell function with Lemma \ref{lemma:alternateAppellForms} (\ref{equation:alternateAppellForm3rd-f}) gives
\begin{align*}
A(q)&= -q^{3+7s}j(q^{6s+10};q^7)\left (\sum_{k=0}^{s-1}q^{6sk-6\binom{k+1}{2}}\left (q^{2k} -q^{-2k+4s-2}\right ) 
 +q^{2s+6\binom{s}{2}}\frac{1}{2}f_3(q^2)\right )\\
 &\qquad -2q^{3+7s}j(q^{6s+10};q^7)q^{2s+6\binom{s}{2}}\left ( -\frac{1}{4}\frac{J_{6,12}^2}{J_{2}}
+\frac{J_{6}^3}{\overline{J}_{0,6}J_{2,6}}\right ). 
\end{align*}
We use (\ref{equation:j-elliptic}) and (\ref{equation:j-flip}) to see that
\begin{align*}
j(q^{6s+10};q^7)&=j(q^{7(s+1)+4-s};q^7)=(-1)^{s+1}q^{-7\binom{s+1}{2}-(s+1)(4-s)}j(q^{4-s};q^{7})\\
&=(-1)^{s+1}q^{-7\binom{s+1}{2}-(s+1)(4-s)}j(q^{s+3};q^{7}),
\end{align*}
and the result follows.
\end{proof}


\begin{proof}[Proof of Proposition \ref{proposition:level13EvenSpinSecondPairFuncEqn}]  
This is similar to the proof of Proposition \ref{proposition:level13evenSpinFirstPairFuncEqn}, so it will be omitted.
\end{proof}


\subsection{The $2/3$-level string functions: Theorem \ref{theorem:level23EvenSpinFirstQuad-sEven}}
We state the necessary propositions to prove Theorem \ref{theorem:level23EvenSpinFirstQuad-sEven}, then we give the proof, and then we give the proofs of the propositions.  In the remaining subsections, we need to consider the parity of $s$, in such cases we will just replace $s$ with $2s$ or $2s+1$.

\smallskip
Let us define the left-hand side of the identity in Theorem \ref{theorem:level23EvenSpinFirstQuad-sEven} to be 
\begin{align*}
F(r):=&q^{7+8r}f_{4,4,1}(-q^{23+12r},-q^{13};q^4)-q^{5+4r}f_{4,4,1}(-q^{23+12r},-q^9;q^4)\\
&\quad -q^{13+8r}f_{4,4,1}(-q^{31+12r},-q^{15};q^4)+q^{9+4r}f_{4,4,1}(-q^{31+12r},-q^{11};q^{4}).
\end{align*}
and the right-hand side to be
{\allowdisplaybreaks \begin{align*}
G(r)&:=q^{7+16r}j(-q^{23+12r};q^{16})
\left(q^{2r+12\binom{r}{2}}q^{2}\omega_{3}(-q^2)
+\sum_{k=0}^{r-1}q^{12rk-12\binom{k+1}{2}}\left ( q^{2k}-q^{-2k+4r-2}\right )  \right)\\
&\quad -q^{17+16r}j(-q^{31+12r};q^{16})\\
&\qquad \times \left ( 1-q^{4r}\left( q^{4r+12\binom{r}{2}}\frac{1}{2}f_{3}(q^4)
+\sum_{k=0}^{r-1}q^{12rk-12\binom{k+1}{2}}\left ( q^{4k}-q^{-4k+8r-4}\right )\right) \right )\\
&\quad -(-1)^{r}\frac{q^{\binom{2r}{2}+2}}{2}\frac{J_{1}^2J_{2}}{J_{4}^2}j(-q^{1+4r};q^{16}).
\end{align*}}%
\begin{proposition}\label{proposition:level23evenSpinFirstQuad-sEvenFuncEqn}
We have that both $F(r)$ and $G(r)$ satisfy the same functional equation:
\begin{align*}
H&(r+4)=q^{27+12r}H(r)\\
& -q^{25+4r}j(-q^{7+12r};q^{16})\left ( q^{4r+2}\sum_{m=0}^{3}q^{-6m^2-12mr-4m}
-\sum_{m=0}^{3}q^{-6m^2-12mr-8m}\right ) \\
& +q^{21+4r}j(-q^{15+12r};q^{16})\left ( q^{4r+4}\sum_{m=0}^{3}q^{-6m^2-12mr-10m}
 -\sum_{m=0}^{3}q^{-6m^2-12mr-14m}\right ).
\end{align*}
\end{proposition}

This proposition is similar to what we found in previous subsections, but we will also have to consider the cases $s$ even and $s$ odd.  We consider $s$ even here and $s$ odd in the next section.
\begin{proposition}\label{proposition:level23EvenSpinFirstQuadAppellForm}  We have
{\allowdisplaybreaks \begin{align*}
&q^{7+4s}
 h_{4,4,1}(-q^{6s+23},-q^{13};q^4)
   - q^{5+2s} h_{4,4,1}(-q^{6s+23},-q^{9};q^4)\\
&\qquad  -q^{13+4s}
 h_{4,4,1}(-q^{31+6s},-q^{15};q^4)
  + q^{9+2s} h_{4,4,1}(-q^{31+6s},-q^{11};q^4)\\
 &=  q^{7+8s}   j(-q^{6s+23};q^{16})\left ( m ( -q^{2+6s},-1;q^{12}  ) 
 -q^{2s-2}m ( -q^{-2+6s},-1;q^{12}  ) \right ) \\
 &\qquad 
 -q^{17+8s} j(-q^{31+6s};q^{16})  \left ( 1 
  -q^{2s}\left (  m ( -q^{4+6s},-1;q^{12} ) - q^{4s-4}m ( -q^{-4+6s},-1;q^{12}  )\right ) \right ).
\end{align*}}%
\end{proposition}

For the case $s$ even, we replace $s\to 2s$ to have  
\begin{corollary} \label{corollary:level23EvenSpinFirstQuad-sEvenAppellForm} We have
{\allowdisplaybreaks \begin{align*}
&q^{7+8s}
 h_{4,4,1}(-q^{12s+23},-q^{13};q^4)
   - q^{5+4s} h_{4,4,1}(-q^{12s+23},-q^{9};q^4)\\
&\quad  -q^{13+8s}
 h_{4,4,1}(-q^{31+12s},-q^{15};q^4)
  + q^{9+4s} h_{4,4,1}(-q^{31+12s},-q^{11};q^4)\\
 &=  q^{7+16s}   j(-q^{12s+23};q^{16})\\
 &\qquad  \times 
 \left( \sum_{k=0}^{s-1}q^{12sk-12\binom{k+1}{2}}\left (q^{2k} -q^{-2k+4s-2}\right ) 
   +q^{2s+12\binom{s}{2}}q^2\omega_{3}(-q^2) \right ) \\
&\qquad \qquad +q^{7+16s}   j(-q^{12s+23};q^{16})2q^{2s+12\binom{s}{2}}\left (  -\frac{1}{2}q^2\frac{J_{12}^3}{J_{4}\overline{J}_{6,12}}
+\frac{J_{12}^3\overline{J}_{4,12}J_{6,12}}{\overline{J}_{0,12}J_{2,12}J_{4,12}\overline{J}_{6,12}}\right )  \\
 &\quad 
 -q^{17+16s} j(-q^{31+12s};q^{16}) \\
 &\qquad  \times  \left ( 1 
  -q^{4s}\left (  \sum_{k=0}^{s-1}q^{12sk-12\binom{k+1}{2}}\left (q^{4k} -q^{-4k+8s-4}\right ) +q^{4s+12\binom{s}{2}}
  \frac{1}{2}f_3(q^4) \right )\right )  \\
&\qquad \qquad   +q^{17+16s} j(-q^{31+12s};q^{16}) q^{4s}2q^{4s+12\binom{s}{2}}
 \left ( -\frac{1}{4}\frac{J_{12,24}^2}{J_{4}}
+\frac{J_{12}^3}{\overline{J}_{0,12}J_{4,12}}\right ).
\end{align*}}%
\end{corollary}

\begin{proof}[Proof of Theorem \ref{theorem:level23EvenSpinFirstQuad-sEven}]
Proposition \ref{proposition:level23evenSpinFirstQuad-sEvenFuncEqn} tells us that
\begin{equation*}
F(s+4)-G(s+4)=q^{27+12s}\left (F(s)-G(s) \right ), 
\end{equation*}
so we only need to prove Theorem  \ref{theorem:level23EvenSpinFirstQuad-sEven} for $0\le s\le 3$.  However, Corollary \ref{corollary:f441-HeckeExpansion} and Corollary \ref{corollary:level23EvenSpinFirstQuad-sEvenAppellForm} tell us that this is equivalent to Proposition \ref{proposition:level23EvenSpinFirstQuad-sEvenThetaId}.
\end{proof}

\begin{proof}[Proof of Proposition \ref{proposition:level23EvenSpinFirstQuadAppellForm}]
From Corollary \ref{corollary:f441-HeckeExpansion}, we have that
 \begin{equation*}
h_{4,4,1}(x,y;q)=j(x;q^4)m\big ( -q^{3}yx^{-1},-1;q^{3}  \big)
 +j(y;q)m \big( q^{6}xy^{-4},-1;q^{12}  \big),
\end{equation*}
but in our case we will need the version with $q\to q^{4}$:
 \begin{equation*}
h_{4,4,1}(x,y;q^4)=j(x;q^{16})m\big ( -q^{12}yx^{-1},-1;q^{12} \big )
 +j(y;q)m \big( q^{24}xy^{-4},-1;q^{48} \big ).
\end{equation*}
Hence
{\allowdisplaybreaks \begin{align*}
A(s):&=q^{7+4s}
 h_{4,4,1}(-q^{6s+23},-q^{13};q^4)
   - q^{5+2s} h_{4,4,1}(-q^{6s+23},-q^{9};q^4)\\
&\qquad  -q^{13+4s}
 h_{4,4,1}(-q^{31+6s},-q^{15};q^4)
  + q^{9+2s} h_{4,4,1}(-q^{31+6s},-q^{11};q^4)\\
 &=q^{7+4s}\left ( j(-q^{6s+23};q^{16})m ( -q^{2-6s},-1;q^{12}  )
 +j(-q^{13};q^4)m ( -q^{-5+6s},-1;q^{48}  )\right) \\
 &\qquad -q^{5+2s}  \left ( j(-q^{6s+23};q^{16})m ( -q^{-2-6s},-1;q^{12}  )
 +j(-q^9;q^4)m (- q^{11+6s},-1;q^{48}  )\right) \\
 &\qquad -q^{13+4s}\left ( j(-q^{31+6s};q^{16})m ( -q^{-4-6s},-1;q^{12}  )
 +j(-q^{15};q^4)m ( -q^{-5+6s},-1;q^{48}  )\right) \\
 &\qquad +q^{9+2s}\left ( j(-q^{31+6s};q^{16})m ( -q^{-8-6s},-1;q^{12}  )
 +j(-q^{11};q^4)m ( -q^{11+6s},-1;q^{48}  )\right).
\end{align*}}%
Four of the Appell function terms cancel.  Using (\ref{equation:j-elliptic}) we see that
\begin{equation*}
j(-q^{13};q^4)=q^{-15}j(-q;q^{4}) \ \textup{and} \
j(-q^{15};q^4)=q^{-21}j(-q^3;q^{4}),
\end{equation*}
as well as 
\begin{equation*}
j(-q^{9};q^4)=q^{-6}j(-q;q^4), \ \textup{and} \ 
j(-q^{11};q^4)=q^{-10}j(-q^3;q^4).
\end{equation*}
It follows that four of the Appell functions cancel, and we get
{\allowdisplaybreaks \begin{align*}
A(s)&=  -q^{5+2s}   j(-q^{6s+23};q^{16})\left ( m ( -q^{-2-6s},-1;q^{12}  ) 
 -q^{2+2s}m ( -q^{2-6s},-1;q^{12}  ) \right ) \\
 &\qquad 
 +q^{9+2s} j(-q^{31+6s};q^{16})  \left ( m ( -q^{-8-6s},-1;q^{12}  ) 
  -q^{4+2s} m ( -q^{-4-6s},-1;q^{12} )\right ) .
\end{align*}}%
Using the Appell function property (\ref{equation:mxqz-flip}), gives us
{\allowdisplaybreaks \begin{align*}
A(s)
&=  -q^{5+2s}   j(-q^{6s+23};q^{16})\left ( -q^{2+6s}m ( -q^{2+6s},-1;q^{12}  ) 
 +q^{8s}m ( -q^{-2+6s},-1;q^{12}  ) \right ) \\
 &\qquad 
 +q^{9+2s} j(-q^{31+6s};q^{16})  \left ( -q^{8+6s}m ( -q^{8+6s},-1;q^{12}  ) 
  +q^{8+8s} m ( -q^{4+6s},-1;q^{12} )\right ).
\end{align*}}%
Simplifying gives
{\allowdisplaybreaks \begin{align*}
A(s)&=  q^{7+8s}   j(-q^{6s+23};q^{16})\left ( m ( -q^{2+6s},-1;q^{12}  ) 
 -q^{2s-2}m ( -q^{-2+6s},-1;q^{12}  ) \right ) \\
 &\qquad 
 -q^{17+8s} j(-q^{31+6s};q^{16})  \left ( m ( -q^{8+6s},-1;q^{12}  ) 
  -q^{2s} m ( -q^{4+6s},-1;q^{12} )\right ).
\end{align*}}%
Employing the Appell function property (\ref{equation:mxqz-fnq-x}) yields
{\allowdisplaybreaks \begin{align*}
A(s)&=  q^{7+8s}   j(-q^{6s+23};q^{16})\left ( m ( -q^{2+6s},-1;q^{12}  ) 
 -q^{2s-2}m ( -q^{-2+6s},-1;q^{12}  ) \right ) \\
 &\quad 
 -q^{17+8s} j(-q^{31+6s};q^{16})  \left ( 1+q^{-4+6s}m ( -q^{-4+6s},-1;q^{12}  ) 
  -q^{2s} m ( -q^{4+6s},-1;q^{12} )\right ).
\end{align*}}%
Rearranging terms produces the result.
\end{proof}

\begin{proof}[Proof of Corollary  \ref{corollary:level23EvenSpinFirstQuad-sEvenAppellForm}]
We consider the case $s$ even of Proposition \ref{proposition:level23EvenSpinFirstQuadAppellForm}.  Replacing $s$ with $2s$ gives
{\allowdisplaybreaks \begin{align*}
&q^{7+8s}
 h_{4,4,1}(-q^{12s+23},-q^{13};q^4)
   - q^{5+4s} h_{4,4,1}(-q^{12s+23},-q^{9};q^4)\\
&\qquad  -q^{13+8s}
 h_{4,4,1}(-q^{31+12s},-q^{15};q^4)
  + q^{9+4s} h_{4,4,1}(-q^{31+12s},-q^{11};q^4)\\
 &=  q^{7+16s}   j(-q^{12s+23};q^{16})\left ( m ( -q^{2+12s},-1;q^{12}  ) 
 -q^{4s-2}m ( -q^{-2+12s},-1;q^{12}  ) \right ) \\
 &\qquad 
 -q^{17+16s} j(-q^{31+12s};q^{16}) \\
 &\qquad \qquad \times  \left ( 1 
  -q^{4s}\left (  m ( -q^{4+12s},-1;q^{12} ) - q^{8s-4}m ( -q^{-4+12s},-1;q^{12}  )\right ) \right ).
\end{align*}}%
Using Lemma \ref{lemma:generalAppellSums}, gives us
{\allowdisplaybreaks \begin{align*}
&q^{7+8s}
 h_{4,4,1}(-q^{12s+23},-q^{13};q^4)
   - q^{5+4s} h_{4,4,1}(-q^{12s+23},-q^{9};q^4)\\
&\quad  -q^{13+8s}
 h_{4,4,1}(-q^{31+12s},-q^{15};q^4)
  + q^{9+4s} h_{4,4,1}(-q^{31+12s},-q^{11};q^4)\\
 &=  q^{7+16s}   j(-q^{12s+23};q^{16})\\
 &\qquad  \times 
 \left ( \sum_{k=0}^{s-1}q^{12sk-12\binom{k+1}{2}}\left (q^{2k} -q^{-2k+4s-2}\right ) 
 +2q^{2s+12\binom{s}{2}}m ( -q^2,-1;q^{12}  )\right ) \\
 &\quad 
 -q^{17+16s} j(-q^{31+12s};q^{16}) \\
 &\qquad  \times  \left ( 1 
  -q^{4s}\left (  \sum_{k=0}^{s-1}q^{12sk-12\binom{k+1}{2}}\left (q^{4k} -q^{-4k+8s-4}\right ) 
 +2q^{4s+12\binom{s}{2}}m ( -q^{4},-1;q^{12}  )\right ) \right ).
\end{align*}}%
Using Lemma \ref{lemma:alternateAppellForms} gives the result.
\end{proof}


\begin{proof}[Proof of Proposition \ref{proposition:level23evenSpinFirstQuad-sEvenFuncEqn} for $F(r)$]
Let us specialize Proposition  \ref{proposition:f-functionaleqn} to $(a,b,c)=(4,4,1)$ with $\ell=-1$, $k=4$, and then use our summation convention (\ref{equation:sumconvention}) on the first sum.  This gives
\begin{align*}
f_{4,4,1}(x,y;q)&=-x^{-1}y^4q^{-6}f_{4,4,1}(q^{12}x,y;q)\\
&\quad -\sum_{m=-1}^{-1}(-x)^mq^{4\binom{m}{2}}j(q^{4m}y;q)+\sum_{m=0}^{3}(-y)^mq^{\binom{m}{2}}j(q^{4m}x;q^4).
\end{align*}
Replacing $q\to q^{4}$ gives
\begin{align}
f_{4,4,1}(x,y;q^4)&=-x^{-1}y^4q^{-24}f_{4,4,1}(q^{48}x,y;q^4)\label{equation:f441-FirstQuadPreFunc}\\
&\quad -\sum_{m=-1}^{-1}(-x)^mq^{16\binom{m}{2}}j(q^{16m}y;q^4)+\sum_{m=0}^{3}(-y)^mq^{4\binom{m}{2}}j(q^{16m}x;q^{16}).
\notag
\end{align}

We substitute $(x,y)=(-q^{23+12r},-q^{13})$ into (\ref{equation:f441-FirstQuadPreFunc}).  We have
\begin{align*}
f_{4,4,1}(-q^{23+12r},-q^{13};q^4)&=q^{5-12r}f_{4,4,1}(-q^{71+12r},-q^{13};q^4)\\
&\quad -q^{-7-12r}j(-q^{-3};q^4)
 +\sum_{m=0}^{3}q^{2m^2+11m}j(-q^{16m+23+12r};q^{16}).
 \end{align*}
Applying (\ref{equation:j-elliptic}) to both theta functions gives
\begin{align}
f_{4,4,1}(-q^{23+12r},-q^{13};q^4)
 &=q^{5-12r}f_{4,4,1}(-q^{71+12r},-q^{13};q^4)-q^{-10-12r}j(-q^{3};q^4)
 \label{equation:f441-firstQuadSum1A}\\
&\quad +q^{-7-12r}j(-q^{7+12r};q^{16})\sum_{m=0}^{3}q^{-6m^2-12mr-4m}.\notag
\end{align}

We substitute $(x,y)=(-q^{23+12r},-q^{9})$ into (\ref{equation:f441-FirstQuadPreFunc}) and apply (\ref{equation:j-elliptic}) to the two theta functions.  This gives
\begin{align}
f_{4,4,1}(-q^{23+12r},-q^{9};q^4)&=q^{-11-12r}f_{4,4,1}(-q^{71+12r},-q^{9};q^4)-q^{-17-12r}j(-q^{3};q^4)
 \label{equation:f441-firstQuadSum2A}\\
&\quad +q^{-7-12r}j(-q^{7+12r};q^{16})\sum_{m=0}^{3}q^{-6m^2-12mr-8m}.\notag
\end{align}

We substitute $(x,y)=(-q^{31+12r},-q^{15})$ into (\ref{equation:f441-FirstQuadPreFunc}) and apply (\ref{equation:j-elliptic}) to the two theta functions.  This gives
\begin{align}
f_{4,4,1}(-q^{31+12r},-q^{15};q^4)&=q^{5-12r}f_{4,4,1}(-q^{79+12r},-q^{15};q^4)-q^{-16-12r}j(-q;q^4)
 \label{equation:f441-firstQuadSum3A}\\
&\quad +q^{-15-12r}j(-q^{15+12r};q^{16})\sum_{m=0}^{3}q^{-6m^2-12mr-10m}.\notag
\end{align}

We substitute $(x,y)=(-q^{31+12r},-q^{11})$ into (\ref{equation:f441-FirstQuadPreFunc}) and apply (\ref{equation:j-elliptic}) to the two theta functions.  This gives
\begin{align}
f_{4,4,1}(-q^{31+12r},-q^{11};q^4)&=q^{-11-12r}f_{4,4,1}(-q^{79+12r},-q^{11};q^4)- q^{-21-12r}j(-q;q^4)
 \label{equation:f441-firstQuadSum4A}\\
&\quad +q^{-15-12r}j(-q^{15+12r};q^{16})\sum_{m=0}^{3}q^{-6m^2-12mr-14m}.\notag
\end{align}

In each of the four identities (\ref{equation:f441-firstQuadSum1A})-(\ref{equation:f441-firstQuadSum4A}), we isolate the double-sum on the right-hand side of the equation.  The four rewritten equations are then
\begin{align}
f_{4,4,1}(-q^{71+12r},-q^{13};q^4) &=q^{12r-5}f_{4,4,1}(-q^{23+12r},-q^{13};q^4)+q^{-15}j(-q^{3};q^4)
 \label{equation:f441-firstQuadSum1B}\\
&\quad -q^{-12}j(-q^{7+12r};q^{16})\sum_{m=0}^{3}q^{-6m^2-12mr-4m},\notag
\end{align}

\begin{align}
f_{4,4,1}(-q^{71+12r},-q^{9};q^4)&=q^{12r+11}f_{4,4,1}(-q^{23+12r},-q^{9};q^4) +q^{-6}j(-q^{3};q^4)
 \label{equation:f441-firstQuadSum2B}\\
&\quad -q^{4}j(-q^{7+12r};q^{16})\sum_{m=0}^{3}q^{-6m^2-12mr-8m}.\notag
\end{align}

\begin{align}
f_{4,4,1}(-q^{79+12r},-q^{15};q^4)&=q^{12r-5}f_{4,4,1}(-q^{31+12r},-q^{15};q^4)
+q^{-21}j(-q;q^4)
 \label{equation:f441-firstQuadSum3B}\\
&\quad -q^{-20}j(-q^{15+12r};q^{16})\sum_{m=0}^{3}q^{-6m^2-12mr-10m},\notag
\end{align}

\begin{align}
f_{4,4,1}(-q^{79+12r},-q^{11};q^4)&=q^{12r+11}f_{4,4,1}(-q^{31+12r},-q^{11};q^4) + q^{-10}j(-q;q^4)
 \label{equation:f441-firstQuadSum4B}\\
&\quad -q^{-4}j(-q^{15+12r};q^{16})\sum_{m=0}^{3}q^{-6m^2-12mr-14m}.\notag
\end{align}
We want to add the four equations  (\ref{equation:f441-firstQuadSum1B})-(\ref{equation:f441-firstQuadSum4B}) such that the left-hand side is $F(r+4)$.  To do so, we have to multiply each by an appropriate power of $q$.  The four new equations are
\begin{align*}
q^{39+8r}f_{4,4,1}(-q^{71+12r},-q^{13};q^4) &=q^{34+20r}f_{4,4,1}(-q^{23+12r},-q^{13};q^4)
+q^{24+8r}j(-q^{3};q^4)\\
&\quad -q^{27+8r}j(-q^{7+12r};q^{16})\sum_{m=0}^{3}q^{-6m^2-12mr-4m},
\end{align*}
\begin{align*}
-q^{21+4r}f_{4,4,1}(-q^{71+12r},-q^{9};q^4)&=-q^{32+16r}f_{4,4,1}(-q^{23+12r},-q^{9};q^4) 
-q^{15+4r}j(-q^{3};q^4)\\
&\quad +q^{25+4r}j(-q^{7+12r};q^{16})\sum_{m=0}^{3}q^{-6m^2-12mr-8m},
\end{align*}
\begin{align*}
-q^{45+8r}f_{4,4,1}(-q^{79+12r},-q^{15};q^4)&=-q^{40+20r}f_{4,4,1}(-q^{31+12r},-q^{15};q^4)
-q^{24+8r}j(-q;q^4)\\
&\quad +q^{25+8r}j(-q^{15+12r};q^{16})\sum_{m=0}^{3}q^{-6m^2-12mr-10m},
\end{align*}
\begin{align*}
q^{25+4r}f_{4,4,1}(-q^{79+12r},-q^{11};q^4)&=q^{36+16r}f_{4,4,1}(-q^{31+12r},-q^{11};q^4) + q^{15+4r}j(-q;q^4)\\
&\quad -q^{21+4r}j(-q^{15+12r};q^{16})\sum_{m=0}^{3}q^{-6m^2-12mr-14m}.
\end{align*}
Adding the four equations, writing the result in terms of $F(r)$, and grouping like terms gives the desired result.
\end{proof}


\begin{proof}[Proof of Proposition \ref{proposition:level23evenSpinFirstQuad-sEvenFuncEqn} for $G(r)$]  We have
{\allowdisplaybreaks \begin{align*}
G(r+4)&=q^{71+16r}j(-q^{71+12r};q^{16})\\
&\qquad \times \Big(q^{2r+8+12\binom{r+4}{2}}q^{2}\omega_{3}(-q^2)
+\sum_{k=0}^{r+4-1}q^{12(r+4)k-12\binom{k+1}{2}}\left ( q^{2k}-q^{-2k+4(r+4)-2}\right )  \Big)\\
&\quad -q^{81+16r}j(-q^{79+12r};q^{16})\\
&\qquad \times \Big ( 1-q^{4r+16}\Big( q^{4r+16+12\binom{r+4}{2}}\frac{1}{2}f_{3}(q^4)\\
&\qquad \qquad +\sum_{k=0}^{r+4-1}q^{12(r+4)k-12\binom{k+1}{2}}\left ( q^{4k}-q^{-4k+8(r+4)-4}\right )\Big) \Big)\\
&\quad -(-1)^{r+4}\frac{q^{\binom{2r+8}{2}+2}}{2}\frac{J_{1}^2J_{2}}{J_{4}^2}j(-q^{1+4r+16};q^{16}).
\end{align*}}%
We first apply (\ref{equation:j-elliptic}) to the simple quotient of theta functions.  This gives
{\allowdisplaybreaks \begin{align*}
(-1)^{r+4}&\frac{q^{\binom{2r+8}{2}+2}}{2}\frac{J_{1}^2J_{2}}{J_{4}^2}j(-q^{1+4r+16};q^{16})\\
&=(-1)^{r}\frac{q^{\binom{2r}{2}+16r+\binom{8}{2}+2}}{2}\frac{J_{1}^2J_{2}}{J_{4}^2}q^{-1-4r}j(-q^{1+4r};q^{16})\\
&=q^{27+12r}(-1)^{r}\frac{q^{\binom{2r}{2}+2}}{2}\frac{J_{1}^2J_{2}}{J_{4}^2}j(-q^{1+4r};q^{16})
\end{align*}}%
We use (\ref{equation:j-elliptic}) and distribute to get
{\allowdisplaybreaks \begin{align*}
G(r+4)&=q^{71+16r}q^{-16\binom{3}{2}-3(23+12r)}j(-q^{23+12r};q^{16})
q^{2r+8+12\binom{r}{2}+48r+12\binom{4}{2}}q^{2}\omega_{3}(-q^2)\\
&\quad +q^{71+16r}q^{-16\binom{3}{2}-3(23+12r)}j(-q^{23+12r};q^{16})\\
&\qquad \times \sum_{k=0}^{r+4-1}q^{12(r+4)k-12\binom{k+1}{2}}\left ( q^{2k}-q^{-2k+4(r+4)-2}\right ) \\
&\quad -q^{81+16r}q^{-16\binom{3}{2}-3(31+12r)}j(-q^{31+12r};q^{16})\\
&\quad  +q^{81+16r}q^{-16\binom{3}{2}-3(31+12r)}j(-q^{31+12r};q^{16})q^{4r+16}q^{4r+16+12\binom{r}{2}+48r+12\binom{4}{2}}\frac{1}{2}f_{3}(q^4)\\
&\quad +q^{81+16r}q^{-16\binom{3}{2}-3(31+12r)}j(-q^{31+12r};q^{16})q^{4r+16}\\
&\qquad \times \sum_{k=0}^{r+4-1}q^{12(r+4)k-12\binom{k+1}{2}}\left ( q^{4k}-q^{-4k+8(r+4)-4}\right ) \\
&\quad -q^{27+12r}(-1)^{r}\frac{q^{\binom{2r}{2}+2}}{2}\frac{J_{1}^2J_{2}}{J_{4}^2}j(-q^{1+4r};q^{16}).
\end{align*}}%
We simplify the exponents to get
{\allowdisplaybreaks \begin{align*}
G(r+4)&=q^{27+12r}q^{7+16r}j(-q^{23+12r};q^{16})
q^{2r+12\binom{r}{2}}q^{2}\omega_{3}(-q^2)\\
&\quad +q^{-46-20r}j(-q^{23+12r};q^{16})\sum_{k=0}^{r+4-1}q^{12(r+4)k-12\binom{k+1}{2}}\left ( q^{2k}-q^{-2k+4(r+4)-2}\right ) \\
&\quad -q^{-60-20r}j(-q^{31+12r};q^{16})\\
&\quad  +q^{27+12r}q^{17+16r}j(-q^{31+12r};q^{16})q^{4r}q^{4r+12\binom{r}{2}}\frac{1}{2}f_{3}(q^4)\\
&\quad +q^{-44-16r}j(-q^{31+12r};q^{16})\sum_{k=0}^{r+4-1}q^{12(r+4)k-12\binom{k+1}{2}}\left ( q^{4k}-q^{-4k+8(r+4)-4}\right ) \\
&\quad -q^{27+12r}(-1)^{r}\frac{q^{\binom{2r}{2}+2}}{2}\frac{J_{1}^2J_{2}}{J_{4}^2}j(-q^{1+4r};q^{16}).
\end{align*}}%
We rewrite the above equation in terms of $G(r)$.  This brings us to
{\allowdisplaybreaks \begin{align*}
G(r+4)&=q^{27+12r}G(r)\\
&\quad -q^{27+12r}q^{7+16r}j(-q^{23+12r};q^{16})
\sum_{k=0}^{r-1}q^{12rk-12\binom{k+1}{2}}\left ( q^{2k}-q^{-2k+4r-2}\right ) \\
&\quad +q^{-46-20r}j(-q^{23+12r};q^{16})\sum_{k=0}^{r+4-1}q^{12(r+4)k-12\binom{k+1}{2}}\left ( q^{2k}-q^{-2k+4(r+4)-2}\right ) \\
&\quad +q^{27+12r}q^{17+16r}j(-q^{31+12r};q^{16})-q^{-60-20r}j(-q^{31+12r};q^{16})\\
&\quad  -q^{27+12r}q^{17+16r}j(-q^{31+12r};q^{16})q^{4r}\sum_{k=0}^{r-1}q^{12rk-12\binom{k+1}{2}}\left ( q^{4k}-q^{-4k+8r-4}\right )\\
&\quad +q^{-44-16r}j(-q^{31+12r};q^{16})\sum_{k=0}^{r+4-1}q^{12(r+4)k-12\binom{k+1}{2}}\left ( q^{4k}-q^{-4k+8(r+4)-4}\right ).
\end{align*}}%
Comparing with the functional equation in Proposition \ref{proposition:level23evenSpinFirstQuad-sEvenFuncEqn}, we see that we need to show
{\allowdisplaybreaks \begin{align*}
-q^{25+4r}&j(-q^{7+12r};q^{16})\left ( q^{4r+2}\sum_{m=0}^{3}q^{-6m^2-12mr-4m}
-\sum_{m=0}^{3}q^{-6m^2-12mr-8m}\right ) \\
&\quad +q^{21+4r}j(-q^{15+12r};q^{16})\left ( q^{4r+4}\sum_{m=0}^{3}q^{-6m^2-12mr-10m}
 -\sum_{m=0}^{3}q^{-6m^2-12mr-14m}\right ) \\
&= -q^{27+12r}q^{7+16r}j(-q^{23+12r};q^{16})
\sum_{k=0}^{r-1}q^{12rk-12\binom{k+1}{2}}\left ( q^{2k}-q^{-2k+4r-2}\right ) \\
&\quad +q^{-46-20r}j(-q^{23+12r};q^{16})\sum_{k=0}^{r+4-1}q^{12(r+4)k-12\binom{k+1}{2}}\left ( q^{2k}-q^{-2k+4(r+4)-2}\right ) \\
&\quad +q^{27+12r}q^{17+16r}j(-q^{31+12r};q^{16})-q^{-60-20r}j(-q^{31+12r};q^{16})\\
&\quad  -q^{27+12r}q^{17+16r}j(-q^{31+12r};q^{16})q^{4r}\sum_{k=0}^{r-1}q^{12rk-12\binom{k+1}{2}}\left ( q^{4k}-q^{-4k+8r-4}\right )\\
&\quad +q^{-44-16r}j(-q^{31+12r};q^{16})\sum_{k=0}^{r+4-1}q^{12(r+4)k-12\binom{k+1}{2}}\left ( q^{4k}-q^{-4k+8(r+4)-4}\right ).
\end{align*}}%
We rewrite what we need to show using $(\ref{equation:j-elliptic})$.  This gives
{\allowdisplaybreaks \begin{align*}
-q^{25+4r}&j(-q^{7+12r};q^{16})\left ( q^{4r+2}\sum_{m=0}^{3}q^{-6m^2-12mr-4m}
-\sum_{m=0}^{3}q^{-6m^2-12mr-8m}\right ) \\
&\quad +q^{21+4r}j(-q^{15+12r};q^{16})\left ( q^{4r+4}\sum_{m=0}^{3}q^{-6m^2-12mr-10m}
 -\sum_{m=0}^{3}q^{-6m^2-12mr-14m}\right ) \\
&= -q^{27+16r}j(-q^{7+12r};q^{16})
\sum_{k=0}^{r-1}q^{12rk-12\binom{k+1}{2}}\left ( q^{2k}-q^{-2k+4r-2}\right ) \\
&\quad +q^{-53-32r}j(-q^{7+12r};q^{16})\sum_{k=0}^{r+4-1}q^{12(r+4)k-12\binom{k+1}{2}}\left ( q^{2k}-q^{-2k+4(r+4)-2}\right ) \\
&\quad +q^{29+16r}j(-q^{15+12r};q^{16})-q^{-75-32r}j(-q^{15+12r};q^{16})\\
&\quad  -q^{29+16r}j(-q^{15+12r};q^{16})q^{4r}\sum_{k=0}^{r-1}q^{12rk-12\binom{k+1}{2}}\left ( q^{4k}-q^{-4k+8r-4}\right )\\
&\quad +q^{-59-28r}j(-q^{15+12r};q^{16})\sum_{k=0}^{r+4-1}q^{12(r+4)k-12\binom{k+1}{2}}\left ( q^{4k}-q^{-4k+8(r+4)-4}\right ).
\end{align*}}%
We turn our focus to the coefficients of the two theta functions.  We first consider the coefficient of $j(-q^{7+12r};q^{16})$.  Here we want to show
{\allowdisplaybreaks \begin{align*}
-q^{25+4r}&\left ( q^{4r+2}\sum_{m=0}^{3}q^{-6m^2-12mr-4m}
-\sum_{m=0}^{3}q^{-6m^2-12mr-8m}\right ) \\
&= -q^{27+16r}
\sum_{k=0}^{r-1}q^{12rk-12\binom{k+1}{2}}\left ( q^{2k}-q^{-2k+4r-2}\right ) \\
&\quad +q^{-53-32r}\sum_{k=0}^{r+4-1}q^{12(r+4)k-12\binom{k+1}{2}}\left ( q^{2k}-q^{-2k+4(r+4)-2}\right ).
\end{align*}}%
Let us rewrite this a little bit.  We have
{\allowdisplaybreaks \begin{align*}
-q^{25+4r}&\left ( q^{4r+2}\sum_{m=0}^{3}q^{-6m^2-12mr-4m}
-\sum_{m=0}^{3}q^{-6m^2-12mr-8m}\right ) \\
&= -q^{27+16r}
\sum_{k=0}^{r-1}q^{12rk-12\binom{k+1}{2}}\left ( q^{2k}-q^{-2k+4r-2}\right ) \\
&\quad +q^{-53-32r}\sum_{k=0}^{3}q^{12(r+4)k-12\binom{k+1}{2}}\left ( q^{2k}-q^{-2k+4(r+4)-2}\right )\\
&\quad +q^{-53-32r}\sum_{k=4}^{r+4-1}q^{12(r+4)k-12\binom{k+1}{2}}\left ( q^{2k}-q^{-2k+4(r+4)-2}\right ).
\end{align*}}%
It suffices to show two separate identities are true.  For the first, we want to show
{\allowdisplaybreaks \begin{align}
0&= -q^{27+16r}\sum_{k=0}^{r-1}q^{12rk-12\binom{k+1}{2}}\left ( q^{2k}-q^{-2k+4r-2}\right ) 
\label{equation:level23evenSpinFirstQuadRHSFuncEqnId1}\\
&\quad +q^{-53-32r}\sum_{k=4}^{r+4-1}q^{12(r+4)k-12\binom{k+1}{2}}\left ( q^{2k}-q^{-2k+4(r+4)-2}\right ),\notag
\end{align}}%
and for the second we want to show
{\allowdisplaybreaks \begin{align}
-q^{25+4r}&\left ( q^{4r+2}\sum_{m=0}^{3}q^{-6m^2-12mr-4m}
-\sum_{m=0}^{3}q^{-6m^2-12mr-8m}\right ) 
\label{equation:level23evenSpinFirstQuadRHSFuncEqnId2}\\
&= q^{-53-32r}\sum_{k=0}^{3}q^{12(r+4)k-12\binom{k+1}{2}}\left ( q^{2k}-q^{-2k+4(r+4)-2}\right ).\notag
\end{align}}%
To show (\ref{equation:level23evenSpinFirstQuadRHSFuncEqnId1}), we see that a shift in the summation index gets us
{\allowdisplaybreaks \begin{align*}
q^{-53-32r}&\sum_{k=4}^{r+4-1}q^{12(r+4)k-12\binom{k+1}{2}}\left ( q^{2k}-q^{-2k+4(r+4)-2}\right )\\
&=q^{-53-32r}\sum_{k=0}^{r-1}q^{12(r+4)(k+4)-12\binom{k+4+1}{2}}q^{2(k+4)}\\
&\quad - q^{-53-32r}\sum_{k=0}^{r-1}q^{12(r+4)(k+4)-12\binom{k+4+1}{2}}q^{-2(k+4)+4(r+4)-2}\\
&=q^{27+16r}\sum_{k=0}^{r-1}q^{12kr-12\binom{k+1}{2}+2k}
- q^{25+20r}\sum_{k=0}^{r-1}q^{12kr-12\binom{k+1}{2}-2k}\\
&=q^{27+16r}\left ( \sum_{k=0}^{r-1}q^{12kr-12\binom{k+1}{2}+2k}
-\sum_{k=0}^{r-1}q^{12kr-12\binom{k+1}{2}-2k+4r-2}\right ).
\end{align*}}%

To show (\ref{equation:level23evenSpinFirstQuadRHSFuncEqnId2}), we reverse the order of summation with $k\to 3-m$ to get
{\allowdisplaybreaks \begin{align*}
q^{-53-32r}&\sum_{k=0}^{3}q^{12(r+4)k-12\binom{k+1}{2}}\left ( q^{2k}-q^{-2k+4(r+4)-2}\right )\\
&=q^{25+4r}\sum_{m=0}^{3}q^{-6m^2-12mr-8m}-q^{8r+27}\sum_{m=0}^{3}q^{-6m^2-12mr-4m}\\
&=-q^{25+4r}\left ( q^{4r+2}\sum_{m=0}^{3}q^{-6m^2-12mr-4m}-\sum_{m=0}^{3}q^{-6m^2-12mr-8m}\right).
\end{align*}}%

Now we turn our focus to the coefficient of $j(-q^{15+12r};q^{16})$.  We want to show
{\allowdisplaybreaks \begin{align*}
q^{21+4r}&\left ( q^{4r+4}\sum_{m=0}^{3}q^{-6m^2-12mr-10m}
 -\sum_{m=0}^{3}q^{-6m^2-12mr-14m}\right ) \\
&= q^{29+16r}-q^{-75-32r}\\
&\quad  -q^{29+20r}\sum_{k=0}^{r-1}q^{12rk-12\binom{k+1}{2}}\left ( q^{4k}-q^{-4k+8r-4}\right )\\
&\quad +q^{-59-28r}\sum_{k=0}^{r+4-1}q^{12(r+4)k-12\binom{k+1}{2}}\left ( q^{4k}-q^{-4k+8(r+4)-4}\right ).
\end{align*}}%
Let us rewrite this as
{\allowdisplaybreaks \begin{align*}
q^{21+4r}&\left ( q^{4r+4}\sum_{m=0}^{3}q^{-6m^2-12mr-10m}
 -\sum_{m=0}^{3}q^{-6m^2-12mr-14m}\right ) \\
&= q^{29+16r}-q^{-75-32r}\\
&\quad  -q^{29+20r}\sum_{k=0}^{r-1}q^{12rk-12\binom{k+1}{2}}\left ( q^{4k}-q^{-4k+8r-4}\right )\\
&\quad +q^{-59-28r}\sum_{k=0}^{3}q^{12(r+4)k-12\binom{k+1}{2}}\left ( q^{4k}-q^{-4k+8(r+4)-4}\right )\\
&\quad +q^{-59-28r}\sum_{k=4}^{r+4-1}q^{12(r+4)k-12\binom{k+1}{2}}\left ( q^{4k}-q^{-4k+8(r+4)-4}\right ).
\end{align*}}%
We proceed as we did for the first coefficient.  We note that with the shift $k\to k+4$, that
{\allowdisplaybreaks \begin{align*}
q^{-59-28r}&\sum_{k=4}^{r+4-1}q^{12(r+4)k-12\binom{k+1}{2}}\left ( q^{4k}-q^{-4k+8(r+4)-4}\right )\\
&=q^{29+20r}\sum_{k=0}^{r-1}q^{12rk-12\binom{k+1}{2}}\left ( q^{4k}-q^{-4k+8r-4}\right ).
\end{align*}}%
Secondly, we reverse the summation index with $k\to 3-m$ to get
{\allowdisplaybreaks \begin{align*}
q^{-59-28r}&\sum_{k=0}^{3}q^{12(r+4)k-12\binom{k+1}{2}}\left ( q^{4k}-q^{-4k+8(r+4)-4}\right )\\
&=q^{25+8r}\sum_{m=0}^{3}q^{-6m^2-12mr-10m}
 - q^{29+16r}\sum_{m=0}^{3}q^{-6m^2-12mr-2m}.
\end{align*}}%
We then tweak the summation indices on the right-hand side.  First we shift the sum in the second summation to being from $m=-1$ to $m=2$.  Then we add and subtract terms to bring the summation back to $m=0$ to $m=3$.  This gives
{\allowdisplaybreaks \begin{align*}
q^{-59-28r}&\sum_{k=0}^{3}q^{12(r+4)k-12\binom{k+1}{2}}\left ( q^{4k}-q^{-4k+8(r+4)-4}\right )\\
&=q^{25+8r}\sum_{m=0}^{3}q^{-6m^2-12mr-10m}
 - q^{21+4r}\sum_{m=-1}^{2}q^{-6m^2-12mr-14m}\\
&=q^{25+8r}\sum_{m=0}^{3}q^{-6m^2-12mr-10m}
 - q^{21+4r}\left ( q^{8+12r}+\sum_{m=0}^{3}q^{-6m^2-12mr-14m}- q^{-96-36r}\right )\\ 
&=q^{25+8r}\sum_{m=0}^{3}q^{-6m^2-12mr-10m}-q^{16r+29}+q^{-75-32r}
 - q^{21+4r}\left (\sum_{m=0}^{3}q^{-6m^2-12mr-14m}\right ),
\end{align*}}%
which is what we want.
\end{proof}

\subsection{The $2/3$-level string functions: Theorem \ref{theorem:level23EvenSpinFirstQuad-sOdd}}  We state the necessary propositions to prove Theorem \ref{theorem:level23EvenSpinFirstQuad-sOdd}, then we give the proof, and then we give the proofs of the propositions.  

\begin{proposition}\label{proposition:level23evenSpinFirstQuad-sOddFuncEqn}   In Theorem \ref{theorem:level23EvenSpinFirstQuad-sOdd}, we define the left-hand side of the identity to be $F(r)$ and the right-hand side to be $G(r)$.  We have that both $F(r)$ and $G(r)$ satisfy the same functional equation
\begin{align*}
H&(r+4)=q^{33+12r}H(r)\\
& +q^{27+4r}j(-q^{13+12r};q^{16})\left ( \sum_{m=0}^{3}q^{-6m^2-12mr-14m}
-q^{4+4r}\sum_{m=0}^{3}q^{-6m^2-12mr-10m}\right ) \\
& -q^{18-8r}j(-q^{5+12r};q^{16})\left ( \sum_{m=0}^{3}q^{-6m^2-12mr-20m}
-q^{6+4r}\sum_{m=0}^{3}q^{-6m^2-12mr-16m}\right ).
\end{align*}
\end{proposition}

In the previous subsection, we considered the case $s$ even in Proposition \ref{proposition:level23EvenSpinFirstQuadAppellForm}.  For the case $s$ odd, we replace $s\to 2s+1$ to get 
\begin{corollary} \label{corollary:level23EvenSpinFirstQuad-sOddAppellForm} We have
{\allowdisplaybreaks \begin{align*}
&q^{11+8s}h_{4,4,1}(-q^{29+12s},-q^{13};q^4)-q^{7+4s}h_{4,4,1}(-q^{29+12s},-q^9;q^4)\\
&\qquad -q^{17+8s}h_{4,4,1}(-q^{37+12s},-q^{15};q^4)+q^{11+4s}h_{4,4,1}(-q^{37+12s},-q^{11};q^{4})\\
&\quad =q^{15+16s}j(-q^{29+12s};q^{16})\\
&\qquad \qquad \times \Big(1-q^{4s}\Big ( \sum_{k=0}^{s-1}q^{12sk-12\binom{k+1}{2}}\left (q^{4k} -q^{-4k+8s-4}\right ) 
+2q^{4s+12\binom{s}{2}}m ( -q^{4},-1;q^{12}  )\Big ) \Big)\\
&\qquad -q^{25+16s}j(-q^{37+12s};q^{16})\\
&\qquad \qquad \times \Big ( 1-q^{4s+2}\Big ( 1-q^{8s} \Big ( \sum_{k=0}^{s-1}q^{12sk-12\binom{k+1}{2}}\left (q^{2k} -q^{-2k+4s-2}\right ) \\
&\qquad \qquad \qquad +2q^{2s+12\binom{s}{2}}m ( -q^{2},-1;q^{12}  )\Big )\Big )   \Big ).
\end{align*}}%
\end{corollary}

\begin{proof}[Proof of Theorem  \ref{theorem:level23EvenSpinFirstQuad-sOdd}]
Proposition \ref{proposition:level23evenSpinFirstQuad-sOddFuncEqn} tells us that
\begin{equation*}
F(s+4)-G(s+4)=q^{33+12s}\left (F(s)-G(s) \right ), 
\end{equation*}
so we only need to prove Theorem  \ref{theorem:level23EvenSpinFirstQuad-sOdd} for $0\le s\le 3$.  However, Corollary \ref{corollary:level23EvenSpinFirstQuad-sOddAppellForm} and Lemma \ref{lemma:alternateAppellForms} give us that
{\allowdisplaybreaks \begin{align*}
&q^{11+8s}h_{4,4,1}(-q^{29+12s},-q^{13};q^4)-q^{7+4s}h_{4,4,1}(-q^{29+12s},-q^9;q^4)\\
&\qquad -q^{17+8s}h_{4,4,1}(-q^{37+12s},-q^{15};q^4)+q^{11+4s}h_{4,4,1}(-q^{37+12s},-q^{11};q^{4})\\
&\quad =q^{15+16s}j(-q^{29+12s};q^{16})\\
&\qquad \quad \times \left(1-q^{4s}\left ( \sum_{k=0}^{s-1}q^{12sk-12\binom{k+1}{2}}\left (q^{4k} -q^{-4k+8s-4}
+q^{4s+12\binom{s}{2}} \frac{1}{2}f_3(q^4)\right ) \right )\right ) \\
&\qquad \qquad -2q^{4s}q^{4s+12\binom{s}{2}}q^{15+16s}j(-q^{29+12s};q^{16})\left (-\frac{1}{4}\frac{J_{12,24}^2}{J_{4}}
+\frac{J_{12}^3}{\overline{J}_{0,12}J_{4,12}}\right )  \\
&\qquad -q^{25+16s}j(-q^{37+12s};q^{16})\\
&\qquad \quad \times \Big ( 1-q^{4s+2}\Big ( 1-q^{8s} \left ( \sum_{k=0}^{s-1}q^{12sk-12\binom{k+1}{2}}\left (q^{2k} -q^{-2k+4s-2}\right ) +2q^{2s+12\binom{s}{2}} \frac{1}{2}q^2\omega_{3}(-q^2)\right ) \\
&\qquad \qquad  -2q^{25+16s}q^{4s+2}q^{8s} q^{2s+12\binom{s}{2}}j(-q^{37+12s};q^{16})
\left ( -\frac{1}{2}q^2\frac{J_{12}^3}{J_{4}\overline{J}_{6,12}}
+\frac{J_{12}^3\overline{J}_{4,12}J_{6,12}}{\overline{J}_{0,12}J_{2,12}J_{4,12}\overline{J}_{6,12}}\right ) .
\end{align*}}%
Corollary \ref{corollary:f441-HeckeExpansion} and  Proposition \ref{proposition:level23EvenSpinFirstQuad-sOddThetaId} then give us the result.
\end{proof}

\begin{proof}[Proof of Corollary  \ref{corollary:level23EvenSpinFirstQuad-sOddAppellForm}]
We consider the case $s$ odd of Proposition \ref{proposition:level23EvenSpinFirstQuadAppellForm}.  Replacing $s$ with $2s+1$ gives
{\allowdisplaybreaks \begin{align*}
q^{11+8s}&h_{4,4,1}(-q^{29+12s},-q^{13};q^4)-q^{7+4s}h_{4,4,1}(-q^{29+12s},-q^9;q^4)\\
&\quad -q^{17+8s}h_{4,4,1}(-q^{37+12s},-q^{15};q^4)+q^{11+4s}h_{4,4,1}(-q^{37+12s},-q^{11};q^{4})\\
&=q^{15+16s}j(-q^{29+12s};q^{16}) \left( m ( -q^{8+12s},-1;q^{12}  )  -q^{4s}m ( -q^{4+12s},-1;q^{12}  )\right)\\
&\quad -q^{25+16s}j(-q^{37+12s};q^{16})
 \left ( m ( -q^{10+12s},-1;q^{12} ) - q^{8s}m ( -q^{2+12s},-1;q^{12}  ) \right ) .
\end{align*}}%
For the first line, Appell function property (\ref{equation:mxqz-fnq-x}) and factoring gives us
\begin{align*}
&m ( -q^{8+12s},-1;q^{12}  )  -q^{4s}m ( -q^{4+12s},-1;q^{12}  )\\
&\qquad = 1+q^{-4+12s}m ( -q^{-4+12s},-1;q^{12}  )  -q^{4s}m ( -q^{4+12s},-1;q^{12}  )\\
&\qquad = 1 -q^{4s}\left ( m ( -q^{4+12s},-1;q^{12}  )-q^{-4+8s}m ( -q^{-4+12s},-1;q^{12}  ) \right ). 
\end{align*}
Lemma \ref{lemma:generalAppellSums} (\ref{equation:generalAppellSumTwoModSix}) then gives
\begin{align*}
&m ( -q^{8+12s},-1;q^{12}  )  -q^{4s}m ( -q^{4+12s},-1;q^{12}  )\\
&\qquad =1-q^{4s}\left ( \sum_{k=0}^{s-1}q^{12sk-12\binom{k+1}{2}}\left (q^{4k} -q^{-4k+8s-4}\right ) 
 +2q^{4s+12\binom{s}{2}}m ( -q^{4},-1;q^{12}  )\right ). 
\end{align*}
For the second line, we use (\ref{equation:mxqz-fnq-x}) and Lemma \ref{lemma:generalAppellSums} (\ref{equation:generalAppellSumOneModSix})to get
{\allowdisplaybreaks \begin{align*}
&m ( -q^{10+12s},-1;q^{12} ) - q^{8s}m ( -q^{2+12s},-1;q^{12}  )\\
&\qquad = 1+q^{-2+12s}m ( -q^{-2+12s},-1;q^{12} ) - q^{8s} m ( -q^{2+12s},-1;q^{12}  )\\
&\qquad = 1- q^{8s} \left ( m ( -q^{2+12s},-1;q^{12}  ) -q^{4s-2}m ( -q^{-2+12s},-1;q^{12} ) \right ) \\
&\qquad = 1-q^{8s} \left ( \sum_{k=0}^{s-1}q^{12sk-12\binom{k+1}{2}}\left (q^{2k} -q^{-2k+4s-2}\right ) 
 +2q^{2s+12\binom{s}{2}}m ( -q^{2},-1;q^{12}  )\right ). 
\end{align*}}%
Assembling the two pieces gives the result.
\end{proof}


\begin{proof}[Proof of Proposition \ref{proposition:level23evenSpinFirstQuad-sOddFuncEqn}]  This is similar to the proof of Proposition \ref{proposition:level23evenSpinFirstQuad-sEvenFuncEqn}, so it will be omitted.
\end{proof}

\subsection{The $2/3$-level string functions: Theorem \ref{theorem:level23EvenSpinSecondQuad-sEven}}  We state the necessary propositions to prove Theorem \ref{theorem:level23EvenSpinSecondQuad-sEven}, then we give the proof, and then we give the proofs of the propositions.

\begin{proposition}\label{proposition:level23evenSpinSecondQuad-sEvenFuncEqn}  In Theorem \ref{theorem:level23EvenSpinSecondQuad-sEven}, we define the left-hand side of the identity to be $F(r)$ and the right-hand side to be $G(r)$.  We have that both $F(r)$ and $G(r)$ satisfy the same functional equation:
\begin{align*}
H&(r+4)=q^{27+12r}H(r)\\
& +q^{23+2r}j(-q^{7+12r};q^{16})\left ( \sum_{m=0}^{3}q^{-6m^2-12mr-10m}
-q^{4+8r}\sum_{m=0}^{3}q^{-6m^2-12mr-2m}\right ) \\
& -q^{18+2r}j(-q^{15+12r};q^{16})\left ( \sum_{m=0}^{3}q^{-6m^2-12mr-16m}  
-q^{8+8r}\sum_{m=0}^{3}q^{-6m^2-12mr-8m}\right ). 
\end{align*}
\end{proposition}

The next proposition is similar to what we found in previous subsections, and we will again have to consider the cases $s$ even and $s$ odd.  We consider $s$ even here and $s$ odd in the next section.
\begin{proposition}\label{proposition:level23EvenSpinSecondQuadAppellForm} We have
\begin{align*}
&q^{7+5s}
 h_{4,4,1}(-q^{6s+23},-q^{15};q^4)
   - q^{3+s} h_{4,4,1}(-q^{6s+23},-q^{7};q^4)\\
&\qquad  -q^{14+5s}
 h_{4,4,1}(-q^{31+6s},-q^{17};q^4)
  + q^{6+s} h_{4,4,1}(-q^{31+6s},-q^{9};q^4)\\
&=
 q^{7+7s}  j(-q^{6s+23};q^{16})\left ( m ( -q^{4+6s},-1;q^{12}  ) - q^{4s-4}m ( -q^{-4+6s},-1;q^{12}  ) \right ) \\
&\qquad  
 - q^{16+7s}  j(-q^{31+6s};q^{16})\left ( 1-q^{4s} \left ( m ( -q^{2+6s},-1;q^{12}  ) -q^{2s-2}m ( -q^{-2+6s},-1;q^{12}  ) \right ) \right ) .
\end{align*}
\end{proposition}

For the case $s$ even, we replace $s\to 2s$ to have  
\begin{corollary}  \label{corollary:level23EvenSpinSecondQuad-sEvenAppellForm}  We have
{\allowdisplaybreaks \begin{align*}
&q^{7+10s}
 h_{4,4,1}(-q^{12s+23},-q^{15};q^4)
   - q^{3+2s} h_{4,4,1}(-q^{12s+23},-q^{7};q^4)\\
&\quad  -q^{14+10s}
 h_{4,4,1}(-q^{31+12s},-q^{17};q^4)
  + q^{6+2s} h_{4,4,1}(-q^{31+12s},-q^{9};q^4)\\
  &= q^{7+14s}  j(-q^{12s+23};q^{16})\\
  &\qquad \times \left (\sum_{k=0}^{s-1}q^{12sk-12\binom{k+1}{2}}\left (q^{4k} -q^{-4k+8s-4}\right ) 
 +2q^{4s+12\binom{s}{2}}m ( -q^{4},-1;q^{12}  )\right ) \\
 &\quad - q^{16+14s}  j(-q^{31+12s};q^{16})\\
 &\qquad \times \left (  1-q^{8s} \left ( 
 \sum_{k=0}^{s-1}q^{12sk-12\binom{k+1}{2}}\left (q^{2k} -q^{-2k+4s-2}\right ) 
 +2q^{2s+12\binom{s}{2}}m ( -q^2,-1;q^{12}  ) \right )\right ).
\end{align*}}%
\end{corollary}

\begin{proof}[Proof of Theorem \ref{theorem:level23EvenSpinSecondQuad-sEven}]
Proposition \ref{proposition:level23evenSpinSecondQuad-sEvenFuncEqn} tells us that
\begin{equation*}
F(s+4)-G(s+4)=q^{27+12s}\left (F(s)-G(s) \right ), 
\end{equation*}
so we only need to prove Theorem  \ref{theorem:level23EvenSpinSecondQuad-sEven} for $0\le s\le 3$.  However,  Corollary \ref{corollary:level23EvenSpinFirstQuad-sEvenAppellForm}  and Lemma \ref{lemma:alternateAppellForms} give us that
{\allowdisplaybreaks \begin{align*}
&q^{7+10s}
 h_{4,4,1}(-q^{12s+23},-q^{15};q^4)
   - q^{3+2s} h_{4,4,1}(-q^{12s+23},-q^{7};q^4)\\
&\quad  -q^{14+10s}
 h_{4,4,1}(-q^{31+12s},-q^{17};q^4)
  + q^{6+2s} h_{4,4,1}(-q^{31+12s},-q^{9};q^4)\\
  &= q^{7+14s}  j(-q^{12s+23};q^{16})\\
  &\qquad \times \left (\sum_{k=0}^{s-1}q^{12sk-12\binom{k+1}{2}}\left (q^{4k} -q^{-4k+8s-4}\right ) 
 +q^{4s+12\binom{s}{2}} \frac{1}{2}f_3(q^4)\right ) \\
 &\qquad \qquad +2q^{7+14s} q^{4s+12\binom{s}{2}} j(-q^{12s+23};q^{16})
 \left (-\frac{1}{4}\frac{J_{12,24}^2}{J_{4}}
+\frac{J_{12}^3}{\overline{J}_{0,12}J_{4,12}}\right )  \\
 &\quad - q^{16+14s}  j(-q^{31+12s};q^{16})\\
 &\qquad \times \left (  1-q^{8s} \left ( 
 \sum_{k=0}^{s-1}q^{12sk-12\binom{k+1}{2}}\left (q^{2k} -q^{-2k+4s-2}\right ) 
 +q^{2s+12\binom{s}{2}}q^2\omega_{3}(-q^2)\right )\right )  \\
&\qquad \qquad  +2q^{8s}q^{16+14s} q^{2s+12\binom{s}{2}} j(-q^{31+12s};q^{16})\left ( -\frac{1}{2}q^2\frac{J_{12}^3}{J_{4}\overline{J}_{6,12}}
+\frac{J_{12}^3\overline{J}_{4,12}J_{6,12}}{\overline{J}_{0,12}J_{2,12}J_{4,12}\overline{J}_{6,12}} \right ).
\end{align*}}%
Corollary \ref{corollary:f441-HeckeExpansion} and  Proposition \ref{proposition:level23EvenSpinSecondQuad-sEvenThetaId} then give us the result.
\end{proof}


\begin{proof}[Proof of Proposition \ref{proposition:level23EvenSpinSecondQuadAppellForm}]
From Corollary \ref{corollary:f441-HeckeExpansion}, we have that
 \begin{equation*}
h_{4,4,1}(x,y;q^4)=j(x;q^{16})m\big ( -q^{12}yx^{-1},-1;q^{12} \big )
 +j(y;q)m \big( q^{24}xy^{-4},-1;q^{48} \big ).
\end{equation*}
Hence
{\allowdisplaybreaks \begin{align*}
A(s):&=q^{7+5s}
 h_{4,4,1}(-q^{6s+23},-q^{15};q^4)
   - q^{3+s} h_{4,4,1}(-q^{6s+23},-q^{7};q^4)\\
&\qquad  -q^{14+5s}
 h_{4,4,1}(-q^{31+6s},-q^{17};q^4)
  + q^{6+s} h_{4,4,1}(-q^{31+6s},-q^{9};q^4)\\
&=q^{7+5s}\left ( j(-q^{6s+23};q^{16})m ( -q^{4-6s},-1;q^{12}  )
 +j(-q^{15};q^{4})m ( -q^{-13+6s},-1;q^{148}  )\right ) \\
&\quad   - q^{3+s} \left (j(-q^{6s+23};q^{16})m ( -q^{-4-6s},-1;q^{12}  )
 +j(-q^7;q^4)m ( -q^{19+6s},-1;q^{48}  ) \right ) \\
&\quad  -q^{14+5s} \left (j(-q^{31+6s};q^{16})m ( -q^{-2-6s},-1;q^{12}  )
 +j(-q^{17};q^4)m ( q^{-13+6s},-1;q^{48}  ) \right ) \\
&\quad  + q^{6+s} \left ( j(-q^{31+6s};q^{16})m ( -q^{-10-6s},-1;q^{12}  )
 +j(-q^9;q^{4})m ( q^{19+6s},-1;q^{48}  )\right ) .
\end{align*}}
Four of the Appell function terms summands cancel.  Using (\ref{equation:j-elliptic}), we have that
\begin{equation*}
j(-q^{15};q^{4})=q^{-21}j(-q^{3};q^{4}), \ \textup{and} \ j(-q^{17};q^{4})=q^{-28}j(-q;q^{4})
\end{equation*}
as well as
\begin{equation*}
j(-q^{7};q^{4})=q^{-3}j(-q^{3};q^{4}), \ \textup{and} \ j(-q^{9};q^{4})=q^{-6}j(-q;q^{4}).
\end{equation*}
It follows that four of the Appell functions cancel, and we get
{\allowdisplaybreaks \begin{align*}
A(s)
&=
- q^{3+s}  j(-q^{6s+23};q^{16})\left ( m ( -q^{-4-6s},-1;q^{12}  ) - q^{4+4s}m ( -q^{4-6s},-1;q^{12}  ) \right ) \\
&\quad  
 + q^{6+s}  j(-q^{31+6s};q^{16})\left ( m ( -q^{-10-6s},-1;q^{12}  )-q^{8+4s} m ( -q^{-2-6s},-1;q^{12}  )  \right ) .
\end{align*}}%
Using the Appell function property (\ref{equation:mxqz-flip}), we get
{\allowdisplaybreaks \begin{align*}
A(s)&=
- q^{3+s}  j(-q^{6s+23};q^{16})\left ( -q^{4+6s}m ( -q^{4+6s},-1;q^{12}  ) + q^{10s}m ( -q^{-4+6s},-1;q^{12}  ) \right ) \\
&\quad  
 + q^{6+s}  j(-q^{31+6s};q^{16})\left ( -q^{10+6s}m ( -q^{10+6s},-1;q^{12}  )+q^{10+10s} m ( -q^{2+6s},-1;q^{12}  )  \right ) .
\end{align*}}%
Simplifying gives
{\allowdisplaybreaks \begin{align*}
A(s)&=
 q^{7+7s}  j(-q^{6s+23};q^{16})\left ( m ( -q^{4+6s},-1;q^{12}  ) - q^{4s-4}m ( -q^{-4+6s},-1;q^{12}  ) \right ) \\
&\qquad  
 - q^{16+7s}  j(-q^{31+6s};q^{16})\left ( m ( -q^{10+6s},-1;q^{12}  )-q^{4s} m ( -q^{2+6s},-1;q^{12}  )  \right ).
\end{align*}}%
Employing the Appell function property (\ref{equation:mxqz-fnq-x}) yields
{\allowdisplaybreaks \begin{align*}
A(s)&=
 q^{7+7s}  j(-q^{6s+23};q^{16})\left ( m ( -q^{4+6s},-1;q^{12}  ) - q^{4s-4}m ( -q^{-4+6s},-1;q^{12}  ) \right ) \\
&\quad  
 - q^{16+7s}  j(-q^{31+6s};q^{16})\left ( 1+q^{-2+6s}m ( -q^{-2+6s},-1;q^{12}  )-q^{4s} m ( -q^{2+6s},-1;q^{12}  )  \right ).
\end{align*}}%
Rearranging terms gives the result.  
\end{proof}


\begin{proof}[Proof of Corollary \ref{corollary:level23EvenSpinSecondQuad-sEvenAppellForm}]
We consider the case $s$ even of Proposition \ref{proposition:level23EvenSpinSecondQuadAppellForm}.  Replacing $s$ with $2s$ gives
\begin{align*}
&q^{7+10s}
 h_{4,4,1}(-q^{12s+23},-q^{15};q^4)
   - q^{3+2s} h_{4,4,1}(-q^{12s+23},-q^{7};q^4)\\
&\qquad  -q^{14+10s}
 h_{4,4,1}(-q^{31+12s},-q^{17};q^4)
  + q^{6+2s} h_{4,4,1}(-q^{31+12s},-q^{9};q^4)\\
  &= q^{7+14s}  j(-q^{12s+23};q^{16})\left (m ( -q^{4+12s},-1;q^{12}  ) - q^{8s-4}m ( -q^{-4+12s},-1;q^{12}  )\right ) \\
 &\qquad - q^{16+14s}  j(-q^{31+12s};q^{16})\\
 &\qquad \qquad \times  \left (   1-q^{8s} \left ( m ( -q^{2+12s},-1;q^{12}  ) -q^{4s-2}m ( -q^{-2+12s},-1;q^{12}  ) \right )\right ).
\end{align*}
Using Lemma \ref{lemma:generalAppellSums} gives the result.
\end{proof}


\begin{proof}[Proof of Proposition \ref{proposition:level23evenSpinSecondQuad-sEvenFuncEqn}]  This is similar to the proof of Proposition \ref{proposition:level23evenSpinFirstQuad-sEvenFuncEqn}, so it will be omitted.
\end{proof}


\subsection{The $2/3$-level string functions: Theorem \ref{theorem:level23EvenSpinSecondQuad-sOdd}}   We state the necessary propositions to prove Theorem \ref{theorem:level23EvenSpinSecondQuad-sEven}, then we give the proof, and then we give the proofs of the propositions.

\begin{proposition}\label{proposition:level23evenSpinSecondQuad-sOddFuncEqn}    In Theorem \ref{theorem:level23EvenSpinSecondQuad-sOdd}, we define the left-hand side of the identity to be $F(r)$ and the right-hand side to be $G(r)$.  We have that both $F(r)$ and $G(r)$ satisfy the same functional equation:
\begin{align*}
H&(r+4)=q^{33+12r}H(r)\\
  & +q^{24+2r}j(-q^{13+12r};q^{16})\left ( \sum_{m=0}^{3}q^{-6m^2-12mr-16m}
 -q^{8+8r}\sum_{m=0}^{3}q^{-6m^2-12mr-8m} \right ) \\
& -q^{14-10r}j(-q^{5+12r};q^{16})\left (\sum_{m=0}^{3}q^{-6m^2-12mr-22m}
-q^{12+8r}\sum_{m=0}^{3}q^{-6m^2-12mr-14m}\right ).
\end{align*}
\end{proposition}

In the previous subsection, we considered the case $s$ even in Proposition \ref{proposition:level23EvenSpinSecondQuadAppellForm}.   For the case $s$ odd, we replace $s\to 2s+1$ to have
\begin{corollary} \label{corollary:level23EvenSpinSecondQuad-sOddAppellForm} We have
{\allowdisplaybreaks \begin{align*}
&q^{12+10s}
 h_{4,4,1}(-q^{12s+29},-q^{15};q^4)
   - q^{4+2s} h_{4,4,1}(-q^{12s+29},-q^{7};q^4)\\
&\quad  -q^{19+10s}
 h_{4,4,1}(-q^{37+12s},-q^{17};q^4)
  + q^{7+2s} h_{4,4,1}(-q^{37+12s},-q^{9};q^4)\\
&=
 q^{14+14s}  j(-q^{12s+29};q^{16})\\
 &\qquad \times \Big ( 1- q^{8s}\Big ( \sum_{k=0}^{s-1}q^{12sk-12\binom{k+1}{2}}\left (q^{2k} -q^{-2k+4s-2}\right ) 
 +2q^{2s+12\binom{s}{2}}m ( -q^2,-1;q^{12}  ) \Big) \Big ) \\
&\quad  
 - q^{23+14s}  j(-q^{37+12s};q^{16})\\
 &\qquad \times \Big ( 1-q^{8s+4} \Big ( 1-q^{4s}\Big ( \sum_{k=0}^{s-1}q^{12sk-12\binom{k+1}{2}}\left (q^{4k} -q^{-4k+8s-4}\right ) 
 +2q^{4s+12\binom{s}{2}}m ( -q^{4},-1;q^{12}  )  \Big ) \Big ) \Big ).
\end{align*}}%
\end{corollary}

\begin{proof}[Proof of Theorem \ref{theorem:level23EvenSpinSecondQuad-sOdd}]
Proposition \ref{proposition:level23evenSpinSecondQuad-sOddFuncEqn} tells us that
\begin{equation*}
F(s+4)-G(s+4)=q^{33+12s}\left (F(s)-G(s) \right ), 
\end{equation*}
so we only need to prove Theorem  \ref{theorem:level23EvenSpinSecondQuad-sOdd} for $0\le s\le 3$.  However,  Corollary \ref{corollary:level23EvenSpinSecondQuad-sOddAppellForm}  and Lemma \ref{lemma:alternateAppellForms} give us that
{\allowdisplaybreaks \begin{align*}
&q^{12+10s}
 h_{4,4,1}(-q^{12s+29},-q^{15};q^4)
   - q^{4+2s} h_{4,4,1}(-q^{12s+29},-q^{7};q^4)\\
&\quad  -q^{19+10s}
 h_{4,4,1}(-q^{37+12s},-q^{17};q^4)
  + q^{7+2s} h_{4,4,1}(-q^{37+12s},-q^{9};q^4)\\
&=
 q^{14+14s}  j(-q^{12s+29};q^{16})\\
 &\qquad \times \left( 1- q^{8s}\left( \sum_{k=0}^{s-1}q^{12sk-12\binom{k+1}{2}}\left (q^{k} -q^{-2k+4s-2}\right )
 +q^{2s+12\binom{s}{2}} q^2\omega_{3}(-q^2)\right ) \right )  \\
&\qquad \qquad  -2q^{8s}q^{2s+12\binom{s}{2}} q^{14+14s}  j(-q^{12s+29};q^{16})\left ( -\frac{1}{2}q^2\frac{J_{12}^3}{J_{4}\overline{J}_{6,12}}
+\frac{J_{12}^3\overline{J}_{4,12}J_{6,12}}{\overline{J}_{0,12}J_{2,12}J_{4,12}\overline{J}_{6,12}}   \right ) \\
&\quad  
 - q^{23+14s}  j(-q^{37+12s};q^{16})\\
 &\qquad \times \left ( 1-q^{8s+4} \left ( 1-q^{4s}\left ( \sum_{k=0}^{s-1}q^{12sk-12\binom{k+1}{2}}\left (q^{4k} -q^{-4k+8s-4}\right )
 +q^{4s+12\binom{s}{2}} \frac{1}{2}f_3(q^4) \right )\right ) \right ) \\ 
&\qquad \qquad   - 2q^{23+14s}  q^{8s+4} q^{4s}q^{4s+12\binom{s}{2}}j(-q^{37+12s};q^{16})\left (-\frac{1}{4}\frac{J_{12,24}^2}{J_{4}}
+\frac{J_{12}^3}{\overline{J}_{0,12}J_{4,12}}\right ).
\end{align*}}%
Corollary \ref{corollary:f441-HeckeExpansion} and Proposition \ref{proposition:level23EvenSpinSecondQuad-sOddThetaId} then give us the result.
\end{proof}

\begin{proof}[Proof of Corollary \ref{corollary:level23EvenSpinSecondQuad-sOddAppellForm}]
We consider the case $s$ odd of Proposition \ref{proposition:level23EvenSpinSecondQuadAppellForm}.  Replacing $s$ with $2s+1$ gives
{\allowdisplaybreaks \begin{align*}
&q^{12+10s}
 h_{4,4,1}(-q^{12s+29},-q^{15};q^4)
   - q^{4+2s} h_{4,4,1}(-q^{12s+29},-q^{7};q^4)\\
&\qquad  -q^{19+10s}
 h_{4,4,1}(-q^{37+12s},-q^{17};q^4)
  + q^{7+2s} h_{4,4,1}(-q^{37+12s},-q^{9};q^4)\\
&=
 q^{14+14s}  j(-q^{12s+29};q^{16})\left ( m ( -q^{10+12s},-1;q^{12}  ) - q^{8s}m ( -q^{2+12s},-1;q^{12}  ) \right ) \\
&\qquad  
 - q^{23+14s}  j(-q^{37+12s};q^{16})\left ( 1-q^{8s+4} \left ( m ( -q^{8+12s},-1;q^{12}  ) -q^{4s}m ( -q^{4+12s},-1;q^{12}  ) \right ) \right ). 
 \end{align*}}%
 Using Appell function property (\ref{equation:mxqz-fnq-x}) on both lines gives
{\allowdisplaybreaks  \begin{align*}
&q^{12+10s}
 h_{4,4,1}(-q^{12s+29},-q^{15};q^4)
   - q^{4+2s} h_{4,4,1}(-q^{12s+29},-q^{7};q^4)\\
&\qquad  -q^{19+10s}
 h_{4,4,1}(-q^{37+12s},-q^{17};q^4)
  + q^{7+2s} h_{4,4,1}(-q^{37+12s},-q^{9};q^4)\\
 &=
 q^{14+14s}  j(-q^{12s+29};q^{16})\left ( 1+q^{-2+12s}m ( -q^{-2+12s},-1;q^{12}  ) - q^{8s}m ( -q^{2+12s},-1;q^{12}  ) \right ) \\
&\qquad  
 - q^{23+14s}  j(-q^{37+12s};q^{16})\\
 &\qquad \qquad \times \left ( 1-q^{8s+4} \left ( 1+q^{-4+12s}m ( -q^{-4+12s},-1;q^{12}  ) -q^{4s}m ( -q^{4+12s},-1;q^{12}  ) \right ) \right ).
\end{align*}}%
Pulling out common factors yields
{\allowdisplaybreaks \begin{align*}
&q^{12+10s}
 h_{4,4,1}(-q^{12s+29},-q^{15};q^4)
   - q^{4+2s} h_{4,4,1}(-q^{12s+29},-q^{7};q^4)\\
&\qquad  -q^{19+10s}
 h_{4,4,1}(-q^{37+12s},-q^{17};q^4)
  + q^{7+2s} h_{4,4,1}(-q^{37+12s},-q^{9};q^4)\\
&=
 q^{14+14s}  j(-q^{12s+29};q^{16})\left ( 1- q^{8s}\left ( m ( -q^{2+12s},-1;q^{12}  ) -q^{4s-2}m ( -q^{-2+12s},-1;q^{12}  ) \right ) \right ) \\
&\qquad  
 - q^{23+14s}  j(-q^{37+12s};q^{16})\\
 &\qquad \qquad \times \left ( 1-q^{8s+4} \left ( 1-q^{4s}\left ( m ( -q^{4+12s},-1;q^{12}  ) -q^{8s-4}m ( -q^{-4+12s},-1;q^{12}  )  \right ) \right ) \right ).
\end{align*}}%
Lemma \ref{lemma:generalAppellSums} gives the result
\end{proof}


\begin{proof}[Proof of Proposition \ref{proposition:level23evenSpinSecondQuad-sOddFuncEqn}]  This is similar to the proof of Proposition \ref{proposition:level23evenSpinFirstQuad-sEvenFuncEqn}, so it will be omitted.
\end{proof}

\section{Mock theta function identities for double-sum coefficients from odd-spin}\label{section:mockThetaTheoremsOddSpin}
In this section we consider the generalized Euler identity in the case of odd-spin.  We prove the mock theta conjecture-like identities found in Theorem \ref{theorem:genEulerOneHalfOddSpin} for $1/2$-level, Theorem \ref{theorem:genEulerOneThirdOddSpin} for $1/3$-level, and Theorem \ref{theorem:genEulerTwoThirdsOddSpin} for $2/3$-level.  The proofs are all similar, but here we do not use functional equations, so the proofs are now much shorter.

\subsection{The $1/2$-level string functions: Theorem \ref{theorem:genEulerOneHalfOddSpin}}
We state the necessary propositions to prove Theorem \ref{theorem:genEulerOneHalfOddSpin}, then we give the proof, and then we give the proofs of the propositions.

\smallskip
We compute the case $s=0$ of the pair of double-sums found in Corollary \ref{corollary:levelOneHalfOddSpin}.   We first evaluate the Appell function expressions in Corollary \ref{corollary:f551-HeckeExpansion} for the respective double-sums in Corollary \ref{corollary:levelOneHalfOddSpin}.  We have

\begin{proposition}\label{proposition:level12OddSpinAppellForm}  We have that
\begin{align*}
-q^{3}h_{5,5,1}(q^9,q^4;q)+qh_{5,5,1}(q^9,q^2;q)= j(q;q^5)\left ( 1 -  2m(-q,-1;q^4)\right ).
\end{align*}
\end{proposition}

\begin{proof}[Proof of Theorem \ref{theorem:genEulerOneHalfOddSpin}]
Corollary \ref{corollary:f551-HeckeExpansion}  gives us
\begin{align*}
-q^{3}&f_{5,5,1}(q^9,q^4;q)+qf_{5,5,1}(q^9,q^2;q)\\
&=-q^{3}h_{5,5,1}(q^9,q^4;q)+qh_{5,5,1}(q^9,q^2;q)\\
&\qquad -\frac{1}{\overline{J}_{0,4}\overline{J}_{0,20}}
\left ( -q^{3}\theta_{5,5,1}(q^9,q^4;q)+q\theta_{5,5,1}(q^9,q^2;q)\right ). 
\end{align*}
Proposition \ref{proposition:level12OddSpinAppellForm} and (\ref{equation:2nd-mu(q)Alt}) then give
\begin{align*}
-q^{3}&f_{5,5,1}(q^9,q^4;q)+qf_{5,5,1}(q^9,q^2;q)\\
&=j(q;q^5)\left ( 1 -  2\left (\frac{1}{4}\mu_{2}(q)+\frac{1}{4}\frac{J_{2,4}^4}{J_{1}^3} \right ) \right )\\
&\qquad -\frac{1}{\overline{J}_{0,4}\overline{J}_{0,20}}
\left ( -q^{3}\theta_{5,5,1}(q^9,q^4;q)+q\theta_{5,5,1}(q^9,q^2;q)\right ). 
\end{align*}
The result then follows from Proposition \ref{proposition:level12OddSpinThetaId}.
\end{proof}

\begin{proof}[Proof of Proposition \ref{proposition:level12OddSpinAppellForm}]
 Corollary \ref{corollary:f551-HeckeExpansion} gives us 
\begin{equation}
h_{5,5,1}(x,y;q)=j(x;q^5)m(-q^4x^{-1}y,-1;q^4)+j(y;q)m(-q^{10}xy^{-5},-1;q^{20}),
\end{equation}
Hence
\begin{align*}
-q^{3}&h_{5,5,1}(q^9,q^4;q)+qh_{5,5,1}(q^9,q^2;q)\\
&=-q^{3}\left (j(q^{9};q^5)m(-q^{-1},-1;q^4)+j(q^4;q)m(-q^{-1},-1;q^{20}) \right ) \\
&\qquad +q\left (j(q^{9};q^5)m(-q^{-3},-1;q^4)+j(q^2;q)m(-q^{9},-1;q^{20}) \right ).
\end{align*}
We note that $j(q^{n};q)=0$ for $n\in\mathbb{Z}$.  Hence
\begin{align*}
-q^{3}&h_{5,5,1}(q^9,q^4;q)+qh_{5,5,1}(q^9,q^2;q)\\
&=-q^{3} j(q^{9};q^5)m(-q^{-1},-1;q^4) +q j(q^{9};q^5)m(-q^{-3},-1;q^4).
\end{align*}
Using the Appell function property (\ref{equation:mxqz-flip}) and factoring yields
\begin{equation*}
-q^{3}h_{5,5,1}(q^9,q^4;q)+qh_{5,5,1}(q^9,q^2;q)
=q^{4} j(q^{9};q^5)\left ( m(-q,-1;q^4) - m(-q^{3},-1;q^4)\right ) .
\end{equation*}
Appell function properties (\ref{equation:mxqz-fnq-x}) and then (\ref{equation:mxqz-flip}) give
\begin{equation*}
-q^{3}h_{5,5,1}(q^9,q^4;q)+qh_{5,5,1}(q^9,q^2;q)
=q^{4} j(q^{9};q^5)\left ( - 1 +  2m(-q,-1;q^4)\right ).
\end{equation*}
The result then follows from (\ref{equation:j-elliptic}).
\end{proof}


\subsection{The $1/3$-level string functions: Theorem  \ref{theorem:genEulerOneThirdOddSpin}}
We state the necessary propositions to prove Theorem \ref{theorem:genEulerOneThirdOddSpin}, then we give the proof, and then we give the proofs of the propositions.

\smallskip
We compute the cases $s=0$ of the two pairs of double-sums found in Corollary \ref{corollary:levelOneThirdOddSpin}.   We first evaluate the Appell function expressions in Corollary \ref{corollary:f771-HeckeExpansion} for the respective double-sums in Corollary \ref{corollary:levelOneThirdOddSpin}.  We have
\begin{proposition}\label{proposition:level13OddSpinAppellForm} We have that
\begin{gather*}
-q^{4} h_{7,7,1}(q^{13},q^{5};q) +q^{2} h_{7,7,1}(q^{13},q^{3};q)
=j(q^{6};q^7)\left ( 1- 2m\big({ -}q^{2},-1;q^{6} \big )\right),\\
-q^{5} h_{7,7,1}(q^{13},q^{6};q) +q h_{7,7,1}(q^{13},q^{2};q)
= j(q^{6};q^7)\left ( 1-2m\big( {-}q,-1;q^{6} \big )\right ).
\end{gather*}
\end{proposition}

\begin{proof}[Proof of Theorem  \ref{theorem:genEulerOneThirdOddSpin}] 
We prove the first identity.   Corollary \ref{corollary:f771-HeckeExpansion} gives us
\begin{align*}
-q^{4} &f_{7,7,1}(q^{13},q^{5};q) +q^{2} f_{7,7,1}(q^{13},q^{3};q)\\
&=-q^{4} h_{7,7,1}(q^{13},q^{5};q) +q^{2} h_{7,7,1}(q^{13},q^{3};q)\\
&\qquad  -\frac{1}{\overline{J}_{0,6}\overline{J}_{0,42}}\left ( -q^4\theta_{7,7,1}(q^{13},q^5;q) +q^2\theta_{7,7,1}(q^{13},q^3;q)\right ).
\end{align*}
Proposition \ref{proposition:level13OddSpinAppellForm} and Lemma \ref{lemma:alternateAppellForms} (\ref{equation:alternateAppellForm3rd-f}) then give us 
\begin{align*}
-q^{4} &f_{7,7,1}(q^{13},q^{5};q) +q^{2} f_{7,7,1}(q^{13},q^{3};q)\\
&= j(q;q^7)\left ( 1-\frac{1}{2}f_3(q^2)\right )+\frac{1}{2}\frac{J_{6,12}^2J_{1,7}}{J_{2}}
-2\frac{J_{6}^3J_{1,7}}{\overline{J}_{0,6}J_{2,6}}\\
&\qquad  -\frac{1}{\overline{J}_{0,6}\overline{J}_{0,42}}\left ( -q^4\theta_{7,7,1}(q^{13},q^5;q) +q^2\theta_{7,7,1}(q^{13},q^3;q)\right).
\end{align*}
The result then follows from Proposition \ref{proposition:level13OddSpinFirstPairThetaId}. 

We prove the second identity.  Corollary \ref{corollary:f771-HeckeExpansion} gives us
\begin{align*}
-q^{5} &f_{7,7,1}(q^{13},q^{6};q) +q f_{7,7,1}(q^{13},q^{2};q)\\
&=-q^{5} h_{7,7,1}(q^{13},q^{6};q) +q h_{7,7,1}(q^{13},q^{2};q)\\
&\qquad -\frac{1}{\overline{J}_{0,6}\overline{J}_{0,42}}\left ( -q^5\theta_{7,7,1}(q^{13},q^6;q) +q\theta_{7,7,1}(q^{13},q^2;q)\right ).
\end{align*}
Proposition \ref{proposition:level13OddSpinAppellForm} and Lemma \ref{lemma:alternateAppellForms} (\ref{equation:alternateAppellForm3rd-w}) then give us 
\begin{align*}
-q^{5} &f_{7,7,1}(q^{13},q^{6};q) +q f_{7,7,1}(q^{13},q^{2};q)\\
&= j(q^{6};q^7)\left ( 1-q\omega_{3}(-q)\right ) +q\frac{J_{6}^3J_{1,7}}{J_{2}\overline{J}_{3,6}}
-2\frac{J_{6}^3\overline{J}_{2,6}J_{3,6}J_{1,7}}{\overline{J}_{0,6}J_{1,6}J_{2,6}\overline{J}_{3,6}} \\
&\qquad -\frac{1}{\overline{J}_{0,6}\overline{J}_{0,42}}\left ( -q^5\theta_{7,7,1}(q^{13},q^6;q) +q\theta_{7,7,1}(q^{13},q^2;q)\right ).
\end{align*}
The result then follows from Proposition \ref{proposition:level13OddSpinSecondPairThetaId}. 
\end{proof}

\begin{proof}[Proof of Proposition \ref{proposition:level13OddSpinAppellForm}]
We prove the first identity.  The proof of the second identity is similar and will be omitted.  From Corollary \ref{corollary:f771-HeckeExpansion}, we have that
\begin{align*}
h_{7,7,1}(x,y;q)=j(x;q^7)m\big({ -}q^{6}yx^{-1},-1;q^{6} \big )
 +j(y;q)m\big({ -}q^{21}xy^{-7},-1;q^{42} \big ).
\end{align*}
Hence
\begin{align*}
-q^{4} &h_{7,7,1}(q^{13},q^{5};q) +q^{2} h_{7,7,1}(q^{13},q^{3};q)\\
&=-q^{4}  j(q^{13};q^7)m\big( {-}q^{-2},-1;q^{6} \big )
+q^{2}j(q^{13};q^7)m\big({-}q^{-4},-1;q^{6} \big ).
\end{align*}
Using (\ref{equation:j-elliptic}) and then (\ref{equation:mxqz-flip}) yields
\begin{align*}
-q^{4} &h_{7,7,1}(q^{13},q^{5};q) +q^{2} h_{7,7,1}(q^{13},q^{3};q)\\
&=q^{-2}  j(q^{6};q^7)m\big({ -}q^{-2},-1;q^{6} \big )
-q^{-4}j(q^{6};q^7)m\big({ -}q^{-4},-1;q^{6} \big )\\
&=-  j(q^{6};q^7)m\big( {-}q^{2},-1;q^{6} \big )
+j(q^{6};q^7)m\big( {-}q^{4},-1;q^{6} \big ).
\end{align*}
Using (\ref{equation:mxqz-fnq-x}) and (\ref{equation:mxqz-flip}) gives
\begin{align*}
-q^{4} &h_{7,7,1}(q^{13},q^{5};q) +q^{2} h_{7,7,1}(q^{13},q^{3};q)\\
&=-  j(q^{6};q^7)m\big( {-}q^{2},-1;q^{6} \big )
+j(q^{6};q^7)\left ( 1+q^{-2}m\big( {-}q^{-2},-1;q^{6} \big )\right) \\
&=-  j(q^{6};q^7)m\big( {-}q^{2},-1;q^{6} \big )
+j(q^{6};q^7)\left ( 1- m\big({ -}q^{2},-1;q^{6} \big )\right),
\end{align*}
and the result follows.
\end{proof}


\subsection{The $2/3$-level string functions: Theorem \ref{theorem:genEulerTwoThirdsOddSpin}}
We state the necessary propositions to prove Theorem \ref{theorem:genEulerTwoThirdsOddSpin}, then we give the proof, and then we give the proofs of the propositions.

\smallskip
We compute the cases $s=0$ of the two quads of double-sums found in Corollary \ref{corollary:levelTwoThirdsOddSpin}.   We first evaluate the Appell function expressions in Corollary \ref{corollary:f441-HeckeExpansion} for the respective double-sums in Corollary \ref{corollary:levelTwoThirdsOddSpin}. 

\smallskip
We consider the first quad of double-sums in Corollary \ref{corollary:levelTwoThirdsOddSpin}.  We have
\begin{proposition}\label{proposition:level23OddSpinFirstQuadAppellForm} We have that
{\allowdisplaybreaks \begin{align*}
 q^{8}& h_{4,4,1}(-q^{26},-q^{13};q^4)
  - q^{5}h_{4,4,1}(-q^{26},-q^{9};q^4)\\
 &\qquad  -q^{14} h_{4,4,1}(-q^{34},-q^{15};q^4)
  +q^{9}h_{4,4,1}(-q^{34},-q^{11};q^4)\\
&=      j(-q^{10};q^{16})\left (m ( -q^{5},-1;q^{12} )-q^{-1}m ( -q,-1;q^{12}  ) \right ) \\
&\qquad  -  j(-q^{2};q^{16})\left ( 1-q
+q\left (  m ( -q^{5},-1;q^{12}  )-q^{-1}m ( -q,-1;q^{12}  )\right ) \right ),
\end{align*}}%
\end{proposition}

For the second quad of double-sums in Corollary \ref{corollary:levelTwoThirdsOddSpin}, we have
\begin{proposition}\label{proposition:level23OddSpinSecondQuadAppellForm} We have that
{\allowdisplaybreaks  \begin{align*}
  q^{9}& h_{4,4,1}(-q^{26},-q^{15};q^4)
 - q^{3+s}h_{4,4,1}(-q^{26},-q^{7};q^4)\\
 &\qquad   -q^{16}h_{4,4,1}(-q^{34},-q^{17};q^4)
   +q^{6}h_{4,4,1}(-q^{34},-q^{9};q^4)\\
&=   j(-q^{10};q^{16})\left (1 - \left( m ( -q^{5},-1;q^{12}  )  -q^{-1}m ( -q,-1;q^{12}  )\right )\right )  \\
& \qquad -q j(-q^{2};q^{16})\left (q^{-2} -\left ( m ( -q^{5},-1;q^{12}  )  -q^{-1}m ( -q,-1;q^{12}  ) \right ) \right ).
 \end{align*}}%
\end{proposition}

\begin{proof}[Proof of Theorem \ref{theorem:genEulerTwoThirdsOddSpin}]  We prove the first identity (\ref{equation:level23OddSpinFirstQuadMockTheta}).   Corollary \ref{corollary:f441-HeckeExpansion} and Proposition \ref{proposition:level23OddSpinFirstQuadAppellForm} then give us
{\allowdisplaybreaks \begin{align*}
 q^{8}& f_{4,4,1}(-q^{26},-q^{13};q^4)
  - q^{5}f_{4,4,1}(-q^{26},-q^{9};q^4)\\
 &\qquad  -q^{14} f_{4,4,1}(-q^{34},-q^{15};q^4)
  +q^{9}f_{4,4,1}(-q^{34},-q^{11};q^4)\\
 &=   j(-q^{10};q^{16})\left (m ( -q^{5},-1;q^{12} )-q^{-1}m ( -q,-1;q^{12}  ) \right ) \\
&\qquad  -  j(-q^{2};q^{16})\left ( 1-q
+q\left (  m ( -q^{5},-1;q^{12}  )-q^{-1}m ( -q,-1;q^{12}  )\right ) \right )\\
 &\qquad -\frac{1}{\overline{J}_{0,3}\overline{J}_{0,12}}
 \Big ( q^{8} \theta_{4,4,1}(-q^{26},-q^{13};q^4)
  - q^{5}\theta_{4,4,1}(-q^{26},-q^{9};q^4)\\
 &\qquad \qquad  -q^{14} \theta_{4,4,1}(-q^{34},-q^{15};q^4)
  +q^{9}\theta_{4,4,1}(-q^{34},-q^{11};q^4)\Big).  
\end{align*}}%
The result then follows from Lemma \ref{lemma:alternateAppellFormsChiPsi} and Proposition \ref{proposition:level23OddSpinFirstQuadThetaId}.

We prove the second identity (\ref{equation:level23OddSpinSecondQuadMockTheta}).   Corollary \ref{corollary:f441-HeckeExpansion} and Proposition \ref{proposition:level23OddSpinSecondQuadAppellForm} then give us
{\allowdisplaybreaks \begin{align*}
q^{9}&
 f_{4,4,1}(-q^{26},-q^{15};q^4)
 - q^{3}
 f_{4,4,1}(-q^{26},-q^{7};q^4)\\
 &\qquad   -q^{16}
 f_{4,4,1}(-q^{34},-q^{17};q^4)
   +q^{6}
 f_{4,4,1}(-q^{34},-q^{9};q^4)\\
 &= j(-q^{10};q^{16})\left (1 - \left( m ( -q^{5},-1;q^{12}  )  -q^{-1}m ( -q,-1;q^{12}  )\right )\right )  \\
& \qquad -q j(-q^{2};q^{16})\left (q^{-2} -\left ( m ( -q^{5},-1;q^{12}  )  -q^{-1}m ( -q,-1;q^{12}  ) \right ) \right )\\
&\qquad -\frac{1}{\overline{J}_{0,3}\overline{J}_{0,12}}\Big ( 
q^{9} \theta_{4,4,1}(-q^{26},-q^{15};q^4)
 - q^{3}\theta_{4,4,1}(-q^{26},-q^{7};q^4)\\
 &\qquad \qquad  -q^{16}\theta_{4,4,1}(-q^{34},-q^{17};q^4)
   +q^{6} \theta_{4,4,1}(-q^{34},-q^{9};q^4)\Big ). 
\end{align*}}%
The result then follows from Lemma \ref{lemma:alternateAppellFormsChiPsi} and Proposition \ref{proposition:level23OddSpinSecondQuadThetaId}.
\end{proof}

\begin{proof}[Proof of Proposition \ref{proposition:level23OddSpinFirstQuadAppellForm}]
From Corollary \ref{corollary:f441-HeckeExpansion}, we have that
 \begin{align*}
h_{4,4,1}(x,y;q^4)=j(x;q^{16})m ( -q^{12}yx^{-1},-1;q^{12}  )
 +j(y;q^4)m ( q^{24}xy^{-4},-1;q^{48}  ).
\end{align*}
Hence
{\allowdisplaybreaks \begin{align*}
 A(s):&=q^{8} h_{4,4,1}(-q^{26},-q^{13};q^4)
  - q^{5}h_{4,4,1}(-q^{26},-q^{9};q^4)\\
 &\qquad  -q^{14} h_{4,4,1}(-q^{34},-q^{15};q^4)
  +q^{9}h_{4,4,1}(-q^{34},-q^{11};q^4)\\
 &=q^{8} \left ( j(-q^{26};q^{16})m ( -q^{-1},-1;q^{12}  )
 +j(-q^{13};q^4)m ( q^{-2},-1;q^{48}  )\right ) \\
&\qquad   - q^{5}\left (j(-q^{26};q^{16})m ( -q^{-5},-1;q^{12}  )
 +j(-q^{9};q^4)m ( q^{14},-1;q^{48}  ) \right ) \\
 &\qquad  -q^{14} \left (j(-q^{34};q^{16})m ( -q^{-7},-1;q^{12}  )
 +j(-q^{15};q^4)m ( q^{-2},-1;q^{48}  ) \right ) \\
&\qquad  +q^{9} \left ( j(-q^{34};q^{16})m ( -q^{-11},-1;q^{12}  )
 +j(-q^{11};q^4)m ( q^{14},-1;q^{48}  )\right ). 
\end{align*}}%
Four of the summands cancel.  We rewrite the theta functions using (\ref{equation:j-elliptic}).  This gives
\begin{equation*}
q^{8}j(-q^{13};q^4)=q^{-7}j(-q;q^{4}),\ \textup{and} \ q^{5}j(-q^{9};q^{4})=q^{-1}j(-q;q^{4}),
\end{equation*}
as well as
\begin{equation*}
q^{14}j(-q^{15};q^{4})=q^{-7}j(-q^3;q^4), \ \textup{and} \  q^{9}j(-q^{11};q^4)=q^{-1}j(-q^3;q^{4}).
\end{equation*}
So, four summands cancel, and we arrive at
{\allowdisplaybreaks \begin{align*}
A(s)&=  - q^{5} j(-q^{26};q^{16})\left (m ( -q^{-5},-1;q^{12} )-q^{3}m ( -q^{-1},-1;q^{12}  ) \right ) \\
&\qquad  +q^{9}  j(-q^{34};q^{16})\left ( m ( -q^{-11},-1;q^{12}  )-q^{5}m ( -q^{-7},-1;q^{12}  )\right ).
\end{align*}}%
We then use the Appell function property (\ref{equation:mxqz-flip}) to obtain
{\allowdisplaybreaks \begin{align*}
A(s)&=  - q^{5} j(-q^{26};q^{16})\left (-q^{5}m ( -q^{5},-1;q^{12} )+q^{4}m ( -q,-1;q^{12}  ) \right ) \\
&\qquad  +q^{9}  j(-q^{34};q^{16})\left (-q^{11} m ( -q^{11},-1;q^{12}  )+q^{12}m ( -q^{7},-1;q^{12}  )\right ). 
\end{align*}}%
Using the Appell function property (\ref{equation:mxqz-fnq-x}) and then (\ref{equation:mxqz-flip}) gives
{\allowdisplaybreaks \begin{align*}
A(s)&= - q^{5} j(-q^{26};q^{16})\left (-q^{5}m ( -q^{5},-1;q^{12} )+q^{4}m ( -q,-1;q^{12}  ) \right ) \\
&\qquad  +q^{9}  j(-q^{34};q^{16})\left (-q^{11}\left(  1-m ( -q,-1;q^{12}  )\right )
 +q^{12}\left (1- m ( -q^{5},-1;q^{12}  )\right ) \right ). 
\end{align*}}%
This can be rewritten as
{\allowdisplaybreaks \begin{align*}
A(s)&=  q^{10} j(-q^{26};q^{16})\left (m ( -q^{5},-1;q^{12} )-q^{-1}m ( -q,-1;q^{12}  ) \right ) \\
&\qquad  -q^{20}  j(-q^{34};q^{16})\left ( 1-q
+q\left (  m ( -q^{5},-1;q^{12}  )-q^{-1}m ( -q,-1;q^{12}  )\right ) \right ).
\end{align*}}%
The result then follows from (\ref{equation:j-elliptic}).
\end{proof}


\begin{proof}[Proof of Proposition \ref{proposition:level23OddSpinSecondQuadAppellForm}]
The proof is similar to that of Proposition \ref{proposition:level23OddSpinFirstQuadAppellForm}, so it will be omitted.
\end{proof}


\section*{Acknowledgements}
The work is supported by the Ministry of Science and Higher Education of the Russian
Federation (agreement no. 075-15-2025-343).   We would like to thank Nicolay Borozenets for some helpful comments and suggestions.


\begin{thebibliography}{999999}

\bibitem{AM09} D. Adamovi\'{c}, A. Milas, {\em The $N= 1$ triplet vertex operator superalgebras}, Comm. Math. Phys. 288.1 (2009), no. 1, pp. 225--270.

\bibitem{ACT} C. Ahn, S. W. Chung, S.-H. Tye, {\em New parafermion, $\operatorname{SU}(2)$ coset and $N=2$ superconformal field theories}, Nuclear Phys. B {\bf 365} (1991), no. 1, pp. 191--240.

\bibitem{And1981} G. E. Andrews, {\em Mordell integrals and Ramanujan's ``lost'' notebook,} pp. 10--48, Analytic Number theory, Philadelphia (1980), Lect. Notes Math. {\bf 889} (1981).

\bibitem{And1986} G. E. Andrews, {\em The fifth and seventh order mock theta functions}, Trans. Amer. Math. Soc. {\bf 293} (1986), pp. 113--134.

\bibitem{And2012} G. E. Andrews, {\em $q$-orthogonal polynomials, Rogers--Ramanujan identities, and mock theta functions}, Proc. Stek. Inst. Math. {\bf 276}(1) (2012) 21--32.

\bibitem{AB2018} G. E. Andrews, B. C. Berndt, {\em Ramanujan's lost notebook: Part V.}, Springer, Berlin, 2018.

\bibitem{AG} G. E. Andrews, F. Garvan, {\em Ramanujan's ``lost'' notebook VI:  The mock theta conjectures,} Adv. Math. {\bf 73} (1989), no. 2,  242--255.

\bibitem{ACR}  J. Auger, T. Creutzig, D. Ridout, {\em Modularity of logarithmic parafermion vertex algebras}, Lett. Math. Phys. {\bf 108} (2018), no. 12, pp. 2543--2587.

\bibitem{BoMo2026}  N. Borozenets, E. T. Mortenson,   {\em On string functions of the generalized parafermionic theories, mock theta functions, and false theta functions}, Adv. Math. {\bf 484} (2026), 110684.

\bibitem{BoMo2025} N. Borozenets, E. T. Mortenson,  {\em On string functions of the generalized parafermionic theories, mock theta functions, and false theta functions, II}, arXiv:2502.03074.

\bibitem{BKMZ} K. Bringmann, J. Kaszian, A. Milas, S. Zwegers, {\em Rank two false theta functions and Jacobi forms of negative definite matrix index}, Adv. Appl. Math. {\em 112} (2020), 101946.

\bibitem{BM} K. Bringmann, A. Milas, {\em $\mathcal{W}$-algebras, false theta functions and quantum modular forms, I}, Int. Math. Res. Not. {\bf 2015} (2015), no. 21, pp. 11351--11387.

\bibitem{BKMN} K. Bringmann, J. Kaszian, A. Milas, C. Nazaroglu, {\em Integral representations of rank two false theta functions and their modularity properties}, Res. Math. Sci. {\bf 8} (2021), no. 4, 54.

\bibitem{BrO1} K.  Bringmann,  K.  Ono, {\em The $f(q)$ mock theta function conjecture and partition ranks}, Invent. Math., {\bf 165} (2006), pp. 243--266.

\bibitem{BrO2}  K.  Bringmann,  K. Ono, {\em Dyson's Ranks and Maass forms}, Ann. Math., {\bf 171} (2010), pp. 419--449.

 \bibitem{DMZ}  A. Dabholkar, S. Murthy, D. Zagier, {\em Quantum black holes, wall crossing, and mock modular forms}, arXiv:1208.4074 (2012).
 
 \bibitem{Dyson1944} F. J. Dyson, {\em Some guesses in the theory of partitions}, Eureka, Cambridge {\bf 8} (1944), pp. 10--15.

\bibitem{FG} J. Frye, F. Garvan, {\em Automatic proof of theta-function identities}, Elliptic integrals, elliptic functions and modular forms in quantum field theory, pp. 195--258, Texts Monogr. Symbol. Comput., Springer, Cham, 2019. 

\bibitem{GM} B. Gordon, R. McIntosh, {\em Some eighth order mock theta functions,} J. London Math. Soc. (2) {\bf 62} (2000), pp. 321-335.

\bibitem{H1} D. R. Hickerson, {\em A proof of the mock theta conjectures}, Invent. Math. {\bf 94} (1988), no. 3,  pp. 639--660.

\bibitem{H2} D. R. Hickerson, {\em On the seventh order mock theta functions}, Invent. Math. {\bf 94} (1988), no. 3, pp. 661--677.

\bibitem{HM} D. R. Hickerson, E. T. Mortenson,
 {\em Hecke-type double sums, Appell--Lerch sums, and mock theta functions,~I},
 Proc. London Math. Soc. (3) {\bf 109} (2014), no. 2, 382--422. 
  
 \bibitem{KP84} V. Kac, D. Peterson, {\em Infinite-dimensional Lie algebras, theta functions and modular forms}, Adv. Math. {\bf 53} (1984), 125--264.

\bibitem{KW88advmath} V. Kac, M. Wakimoto, {\em Modular and conformal invariance constraints in representation theory of affine algebras}, Adv. Math. {\bf 70} (1988), no. 2, 156--236.

 \bibitem{KW88} V. Kac, M. Wakimoto, {\em Modular invariant representations of infinite-dimensional Lie algebras and superalgebras}, Proc. Nat. Acad. Sci. USA {\bf 85} (1988), 4956-4960.
 
  
  \bibitem{KW90} V. Kac, M. Wakimoto, {\em Branching functions for winding subalgebras and tensor products}, Acta Appl. Math. {\bf 21} (1990), 3--39. 
  
\bibitem{KM25} S. Konenkov, E. T. Mortenson,  {\em Two $2/5$-level mock theta conjecture-like identities},
 Res. in Number Theory {\bf 11}, 93 (2025)

  \bibitem{L}  S. Lu, {\em Some results on modular invariant representations}, Infinite-dimensional Lie algebras and groups, Adv. Ser. in Math. Phys. {\bf 7} (1988), pp. 235--253. 

\bibitem{M2} R. McIntosh, {\em Second order mock theta functions}, Can. Math. Bull., {\bf 50} (2007), no. 2, pp. 284-290.

\bibitem{MoSa} E. T. Mortenson, A. Sahu,
{\em Expressing $q$-series in terms of building blocks of Hecke-type double-sums},
Int. J. Number Theory {\bf 16} (2023), no. 6, 1429--1451.

\bibitem{MZ2023} E. T. Mortenson, S. Zwegers,
{\em The mixed mock modularity of certain duals of generalized quantum modular forms of Hikami and Lovejoy}, Adv. Math. {\bf 418} (2023), 108944.

\bibitem{RLN} S. Ramanujan, 
{\em The lost notebook and other unpublished papers}, 
Narosa Publishing House, New Delhi, 1987.

\bibitem{SW} A. Schilling, S. O. Warnaar, {\em Conjugate Bailey Pairs}, Contemp. Math., {\bf 197} (2002), pp. 227--255.

\bibitem{Selb1938} A. Selberg, {\em \"Uber die Mock-Thetafunktionen siebenter Ordnung}, Arch. Math. og Naturvidenskab {\bf 41} (1938), pp. 3--15.

\bibitem{W3} G. N. Watson, {\em The final problem: an account of the mock theta functions,} J. London Math. Soc., {\bf 11} (1936), pp. 55-80.

\bibitem{Zag2007} D. B. Zagier, {\em Ramanujan's mock theta functions and their applications [d'apr\`es Zwegers and Bringmann-Ono]}, S\' eminaire Bourbaki, 2007-2008, no. 986.

\bibitem{Zw1} S. P. Zwegers, {\em Mock $\vartheta$-functions and real analytic modular forms}, $q$-series with applications to combinatorics, number theory, and physics (Ed. B. C. Berndt and K. Ono), Contempt. Math. {\bf 291}, Amer. Math. Soc. (2001), pp. 269--277.

\bibitem{Zw2} S. P. Zwegers, {\em Mock theta functions}, Ph.D. Thesis, Universiteit Utrecht, 2002.





\end{thebibliography}
\end{document}